\newif{\ifarxiv}
\arxivtrue

\documentclass[preprint,12pt]{elsarticle}
\pdfoutput=1 

\usepackage{latexsym,amsmath,amssymb,stmaryrd,graphicx}
\usepackage{color}
\usepackage[all]{xy}
\usepackage{url}
\usepackage[inline]{enumitem}

\newcommand\eqdef{\mathrel{\buildrel \text{def}\over=}}

\newcommand\pow{\mathbb{P}}
\newcommand\Pfin{\pow_{\mathrm{fin}}}
\newcommand{\rat}{\mathbb{Q}}
\newcommand{\real}{\mathbb{R}}
\newcommand{\creal}{\overline{\mathbb{R}}_+}
\newcommand{\Rp}{\mathbb{R}_+}

\newcommand\diff{\smallsetminus}
\DeclareMathOperator{\upc}{\uparrow\!}
\DeclareMathOperator{\dc}{\downarrow\!}
\newcommand\nat{\mathbb{N}}

\newcommand\limp{\mathrel{\Rightarrow}}

\newcommand\Lform{\mathcal L}
\newcommand\co{{\mathrm{co}}}
\newcommand\Lformco{\Lform_\co}

\newcommand\uuarrow{\rlap{$\uparrow$}\raise.5ex\hbox{$\uparrow$}}
\newcommand\ddarrow{\rlap{$\downarrow$}\raise.5ex\hbox{$\downarrow$}}

\newcommand\identity[1]{\mathrm{id}_{#1}}

\newcommand\Pred{\mathbb{P}}
\newcommand\Angel{{\mathtt{A}}}
\newcommand\Demon{{\mathtt{D}}}
\newcommand\Nature{{\mathtt{P}}}
\newcommand\AN{{\Angel\Nature}}
\newcommand\DN{{\Demon\Nature}}
\newcommand\ADN{{\Angel\Demon\Nature}}
\newcommand\one{\mathbf 1}

\newcommand\Smyth{{\mathcal Q}}
\newcommand\Hoare{{\mathcal H}}

\newcommand\Vt{\mathsf{V}}
\newcommand\HV{\Hoare_\Vt}
\newcommand\HVz{\Hoare_{0\Vt}}
\newcommand\SV{\Smyth_{\Vt}}
\newcommand\SVz{\Smyth_{0\Vt}}
\newcommand\A{{\mathbf A}}
\newcommand\quasi{\mathrm{q}} 
\newcommand\QL{\Plotkin^\quasi}
\newcommand\QLV{\QL_\Vt}

\newcommand\Aval{\Plotkin^\A}
\newcommand\AvalV{\Plotkin^\A_\Vt}
\newcommand\Plotkin{\mathop{\mathcal P\ell}}
\newcommand\TEMleq{\sqsubseteq^{\text{TEM}}}

\newcommand\Plotkinn{\Plotkin_{\mathcal V}}
\newcommand\PV\Plotkinn 
\newcommand\Val{{\mathbf V}}
\newcommand\Sober{{\mathcal S}}
\newcommand\Open{{\mathcal O}}
\newcommand\egame[1]{\mathfrak{e}_{#1}}
\newcommand\supp{\mathop{\mathrm{supp}}}
\newcommand\bPi{\boldsymbol\Pi}
\newcommand\ugame[1]{\mathfrak{u}_{#1}}

\newcommand\Topcat{\mathbf{Top}}

\newcommand{\interior}[1]{\text{int} ({#1})} 

\newcommand\dsup{\sup\nolimits^{\scriptstyle\uparrow}}
\newcommand\finf{\inf\nolimits^{\scriptstyle\downarrow}}
\newcommand\dcup{\bigcup\nolimits^{\scriptstyle\uparrow}}
\newcommand\fcap{\bigcap\nolimits^{\scriptstyle\downarrow}}
\newcommand\conv{\mathop{\text{conv}}}

\newcommand\cb[1]{\mathbf{#1}} 
\newcommand{\catc}{{\cb{C}}} 
\newcommand{\catd}{{\cb{D}}}
\newcommand{\catk}{{\cb{K}}}
\newcommand{\cati}{{\cb{I}}}

\newcommand\Uuarrow{\mathop{\Uparrow}}

\newtheorem{theorem}{Theorem}[section]
\newtheorem{proposition}[theorem]{Proposition}

\newtheorem{corollary}[theorem]{Corollary}
\newtheorem{lemma}[theorem]{Lemma}
\newproof{proof}{Proof}

\newtheorem{fact}[theorem]{Fact}

\newtheorem{definition}[theorem]{Definition}
\newtheorem{example}[theorem]{Example}

\newtheorem{remark}[theorem]{Remark}

\newcommand\ForAuthors[1]
 {\par\smallskip                     
  \begin{center}
   \fbox
   {\parbox{0.9\linewidth}
    {\raggedright--- #1}
   }
  \end{center}
  \par\smallskip                     %
 }

\journal{Topology and its Applications}

\begin{document}

\begin{frontmatter}



\title{On the Preservation of Projective Limits by Functors of
  Non-Deterministic, Probabilistic, and Mixed Choice}


\author{Jean Goubault-Larrecq}

\address{Universit\'e Paris-Saclay, CNRS, ENS Paris-Saclay,
  Laboratoire M\'ethodes
  Formelles, 91190, Gif-sur-Yvette, France.\\
  \texttt{jgl@lmf.cnrs.fr}
}

\begin{abstract}
  We examine conditions under which projective limits of topological
  spaces are preserved by the continuous valuation functor $\Val$ and
  its subprobability and probability variants (used to represent
  probabilistic choice), by the Smyth hyperspace functor (demonic
  non-deterministic choice), by the Hoare hyperspace functor (angelic
  non-deterministic choice), by Heckmann's $\A$-valuation functor, by
  the quasi-lens functor, by the Plotkin hyperspace functor (erratic
  non-deterministic choice), and by prevision functors and powercone
  functors that implement mixtures of probabilistic and
  non-deterministic choice.
\end{abstract}


\end{frontmatter}


\noindent
\begin{minipage}{0.25\linewidth}
  \includegraphics[scale=0.2]{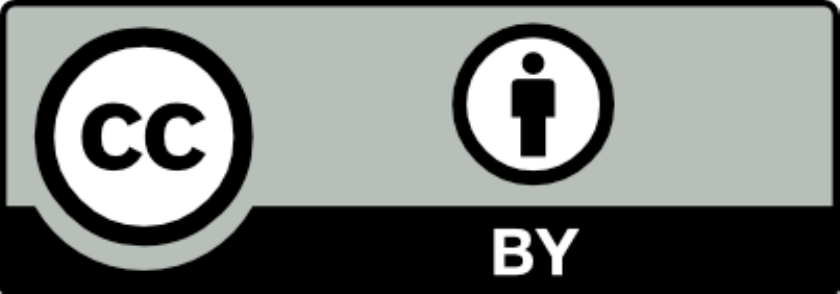}
\end{minipage}
\begin{minipage}{0.74\linewidth}
  \scriptsize
  For the purpose of Open Access, a CC-BY public copyright licence has
  been applied by the authors to the present document and will be
  applied to all subsequent versions up to the Author Accepted
  Manuscript arising from this submission.
\end{minipage}

\section{Introduction}
\label{sec:intro}

A celebrated theorem of Prokhorov \cite{Prohorov:projlim} states that
projective limits of bounded measures exist under what is known as a
uniform tightness assumption.  Bochner \cite{Bochner:harmonic} proved
a similar theorem under a sequential maximality assumption.  The paper
\cite{JGL:kolmogorov} looked at the case of continuous valuations, a
very close cousin to measures, on various kinds of projective limits
of various kinds of (non-Hausdorff) spaces.  In essence, and for now
up to some approximation, what we proved there was that the continuous
valuation functor $\Val$ commutes with various kinds of projective
limits of several kinds of (non-Hausdorff) spaces.  The purpose of
this paper is to examine the case of other standard functors that
implement various forms of non-deterministic, probabilistic, and mixed
choice.  While $\Val$ implements probabilistic choice, we will look at
the Smyth hyperspace functor $\SV$ (demonic non-determinism), the
Hoare powerspace functor $\HV$ (angelic non-determinism), a few
variants of the Plotkin powerspace functor (erratic non-determinism),
as well as prevision functors, or equivalently mixed powerdomains
\cite{JGL-mscs16,KP:mixed}.  This covers all known combinations of
functors implementing probabilistic choice, non-deterministic choice,
and their mixture.

\paragraph{Outline}

We start with some preliminary definitions in
Section~\ref{sec:preliminaries}, and we give a generic account of the
problem we will solve for general endofunctors $T$ on $\Topcat$ in
Section~\ref{sec:general-setting}.  We deal with the case of the
continuous valuations, subprobability valuations and probability
valuations functors in Section~\ref{sec:cont-valu}.  This is pretty
easy: the hard work was done in \cite{JGL:kolmogorov}.  We deal with
another easy situation in Section~\ref{sec:ep-systems}, the case of
ep-systems, an otherwise common setting in domain theory.  We proceed
with the Smyth hyperspace functor $\SV$ in
Section~\ref{sec:smyth-hyperspace}, a model of demonic
non-deterministic choice.  This one is remarkable in the sense that
$\SV$ preserves all projective limits, provided all the spaces are
sober.  We deal with the Hoare hyperspace functor $\HV$---a model of
angelic non-deterministic choice---in
Section~\ref{sec:hoare-hyperspaces}, by reduction to the continuous
valuation functor.  Erratic non-determinism can be modeled by various
related functors.  We deal with Heckmann's $\A$-valuations and with
quasi-lenses in Section~\ref{sec:cont-a-valu}.  While not as well
known as lenses, they have better properties, and their study
essentially reduces to $\SV$ and $\HV$.  We reduce the case of lenses,
namely the usual form of what is known as the Plotkin hyperspace
functor, in Section~\ref{sec:lenses}, by reduction to the case of
quasi-lenses.

All this constitutes part one of the paper.  A second part is devoted
to functors that implement mixtures of probabilistic and
non-deterministic choice.  Those can be implemented by functors of a
specific kind, which we call subcontinuation functors, which we
introduce in Section~\ref{sec:subc-funct}, and which include all the
prevision functors of \cite{Gou-csl07,JGL-mscs16}.  This will allow us
to deal with the superlinear prevision functor, which mixes
probabilistic and demonic non-deterministic choice, by a reduction to
the case of continuous valuations and the Smyth hyperspace, in
Section~\ref{sec:prev-powerc}.  The case of the sublinear prevision
functor---a mixture of probabilistic and angelic non-deterministic
choice---is considerably more complex, and will be dealt with in
Section~\ref{sec:sublinear-previsions} after an intermission
(Section~\ref{sec:interm-pres-local}), where we will prove a number of
required technical auxiliary results.  Those are results of
independent interest: the continuous valuation functors preserve local
compactness and proper maps, any limit of a projective system of
$\odot$-consonant sober spaces and proper maps is $\odot$-consonant
and sober, and any $\omega$-projective limit of locally compact sober
spaces is $\odot$-consonant.  In all those sections (except the
intermediate Section~\ref{sec:interm-pres-local}), we also examine the
related powercone functors
\cite{Mislove:nondet:prob,Tix:PhD,TKP:nondet:prob,DBLP:journals/acta/McIverM01}.
We finish with the fork functor, which implements a mixture of
probabilistic and erratic non-deterministic choice, in
Section~\ref{sec:forks}---this is a pretty easy reduction to the cases
of superlinear and sublinear prevision functors---and the related
powercone functor in Section~\ref{sec:plotk-pred_n-funct}.

We conclude in Section~\ref{sec:conclusion}.

\section{Preliminaries}
\label{sec:preliminaries}

For background on topology, we refer the reader to
\cite{JGL-topology}.  We write $\interior A$ for the interior of $A$,
$cl (A)$ (or $cl_X (A)$) for the closure of $A$ (in a space $X$), and
$\Open X$ for the lattice of open subsets of $X$.  The specialization
preordering $\leq$ of a topological space $X$ is defined on points
$x, y \in X$ by $x \leq y$ if and only if every open neighborhood of
$x$ contains $y$, if and only if $x$ lies in the closure of $\{y\}$.

We will also say that $x$ is \emph{below} $y$ and that $y$ is
\emph{above} $x$ when $x \leq y$.  A space is $T_0$ if and only if
$\leq$ is antisymmetric, $T_1$ if and only if $\leq$ is the equality
relation.

A \emph{base} for a topology (resp., of a topological space) is a
collection of open sets whose unions span all the open sets.
Equivalently, a collection $B$ of open subsets is a base if and only
for every point $x$, for every open neighborhood $U$, there is an
element $V \in B$ such that $x \in V \subseteq U$.  A \emph{subbase}
is a collection of open sets whose finite intersections form a base.
A subbase is said to \emph{generate} the topology.

A \emph{compact} subset $A$ of a space $X$ is one such that one can
extract a finite subcover from any of its open covers.  No separation
property is assumed.  A subset $A$ of $X$ \emph{saturated} if and only
if it is equal to the intersection of its open neighborhoods, or
equivalently if and only if it is upwards-closed in the specialization
preordering of $X$.

A space $X$ is \emph{locally compact} if and only if every point has a
base of compact neighborhoods, or equivalently of compact saturated
neighborhoods, since for any compact subset $K$ of $X$, the upward
closure $\upc K$ of $K$ with respect to the specialization preordering
of $X$ is compact saturated.  Please beware that, in non-Hausdorff
spaces, a compact space may fail to be locally compact.

A space is \emph{coherent} if and only if the intersection of any two
compact saturated subsets is compact (and necessarily saturated).
That, too, is a property that may fail in non-Hausdorff spaces.

A \emph{stably locally compact} space is a coherent, locally compact,
sober space; see below for the definition of sober.  A
\emph{Noetherian space} is a space whose subspaces are all compact.

An \emph{irreducible} closed subset $C$ of $X$ is a non-empty closed
subset such that, for any two closed subsets $C_1$ and $C_2$ of $X$
such that $C \subseteq C_1 \cup C_2$, $C$ is included in $C_1$ or in
$C_2$ already; equivalently, if $C$ intersects two open sets, it must
intersect their intersection.  A space $X$ is \emph{sober} if and only
if it is $T_0$ and every irreducible closed subset is of the form
$\dc x$ for some point $x \in X$.  Every Hausdorff space, for example,
is sober.  The notation $\dc x$ stands for the \emph{downward closure}
of $x$ in $X$, namely the set of points $y$ below $x$.  Symmetrically,
$\upc x$ stands for the \emph{upward closure} of $x$, namely the set
of points $y$ above $x$.  This notation extends to $\upc A$, for any
subset $A$, denoting $\bigcup_{x \in A} \upc x$.

A function $f \colon X \to Y$ between topological spaces is
\emph{continuous} if and only if $f^{-1} (V)$ is open in $X$ for every
$V \in \Open Y$.  It is equivalent to require that this property holds
for every $V$ taken from a given subbase of $Y$. Every continuous map
is monotonic (with respective to the respective specialization
preorderings).

Following \cite{GLHL:trans:words}, we will say that $f$ is \emph{full}
if and only if every open subset of $X$ can be written as $f^{-1} (V)$
for some $V \in \Open Y$---equivalently, if that is the case for just
the sets from a given subbase of $X$.  An injective, full, continuous
map is a \emph{topological embedding}; and a full map from a $T_0$
space is always injective.  (Indeed, if $f \colon X \to Y$ is full and
$X$ is $T_0$, for all $x, x' \in X$ such that $f (x)=f (x')$, for
every open set $U$ of $X$, $U=f^{-1} (V)$ for some $V \in \Open Y$, so
$x \in U$ if and only if $f (x) \in V$ if and only if $f (x') \in V$
if and only if $x' \in U$; hence $x=x'$.)  A \emph{homeomorphism},
namely a bijective, continuous map whose inverse is also continuous, is
the same as a bijective full continuous map (or just surjective, if
its domain is known to be $T_0$).

A family $D$ of elements of a preordered set $P$ is \emph{directed} if
and only if it is non-empty and every pair of elements of $D$ has an
upper bound in $D$.  In case $P$ is a poset, we write $\dsup D$, or
$\dsup_{i \in I} x_i$ when $D = {(x_i)}_{i \in I}$ for the supremum of
a directed family, if it exists; similarly, we write
$\dcup_{i \in I} U_i$ for the union of a directed family of subsets
$U_i$ of a fixed set.  Dually, $D$ is \emph{filtered} if and only if
it is directed with respect to the opposite ordering.  A related
notion is that of \emph{net}, namely a collection
${(x_i)}_{i \in I, \sqsubseteq}$ of points indexed by a set $I$ with a
preordering $\sqsubseteq$ that makes it directed.  A \emph{monotone
  net} in a poset $P$ is a net whose points are taken from $P$, and
such that $i \sqsubseteq j$ implies $x_i \leq x_j$.  The underlying
family $\{x_i \mid i \in I\}$ is then directed.  Conversely, every
directed family $D$ can be seen in a canonical way as a monotone net
by letting $I \eqdef D$, $x_i \eqdef i$, and $\sqsubseteq$ be the
restriction of the ordering $\leq$ on $P$ to $D$.

A function $f \colon P \to Q$ between posets is \emph{monotonic} if
and only if for all $x, x' \in P$, $x \leq x'$ implies
$f (x) \leq f (x')$.  It is \emph{Scott-continuous} if and only if $f$
is monotonic and for every directed family ${(x_i)}_{i \in I}$ with a
supremum $x$ in $P$, the (necessarily directed) family of elements
$f (x_i)$ has $f (x)$ as supremum.  Scott-continuity is equivalent to
continuity with the respective Scott topologies on $P$ and $Q$.  The
\emph{Scott topology} on a poset $P$ consists of those subsets
$U$---the \emph{Scott-open subsets} of $P$---that are upwards closed
($x \in U$ and $x \leq x'$ implies $x' \in U$) and such that every
directed family $D$ that has a supremum in $U$ intersects $U$.  That
is most useful in the context of \emph{dcpo}s (short for
directed-complete posets), namely posets in which every directed
family has a supremum.

A \emph{monotone convergence space} is a $T_0$ space that is a dcpo in
its specialization ordering $\leq$ and whose topology is coarser than
the Scott topology of $\leq$.  Every dco in its Scott topology, every
sober space is a monotone convergence space.

We will introduce other topological concepts along the way, as needed.

A \emph{diagram} in a category $\catc$ is a functor
$F \colon \cati \to \catc$ from a small category $\cati$ to $\catc$.
We let $|\cati|$ denote the set of objects of $\cati$.  A \emph{cone}
of $F$ is a pair $X, {(p_i)}_{i \in |\cati|}$, where $X$ is an object
of $\catc$ and the morphisms $p_i \colon X \to F (i)$, for each
$i \in |\cati|$ are such that for every morphism
$\varphi \colon j \to i$ in $\cati$, $F (\varphi) \circ p_j = p_i$.  A
\emph{limit} of $F$ is a \emph{universal cone} of $F$, namely a cone
such that for every cone $Y, {(q_i)}_{i \in |\cati|}$ of $F$, there is
a unique morphism $f \colon Y \to X$ such that $p_i \circ f = q_i$ for
every object $i$ of $\cati$.  Limits are unique up to isomorphism when
they exist.  All limits exist in $\Topcat$, and the following is the
\emph{canonical limit} of $F$: $X$ is the subspace of
$\prod_{i \in |\cati|} F (i)$ consisting of those tuples $\vec x$ such
that $F (\varphi) (x_j) = x_i$ for every morphism
$\varphi \colon j \to i$ in $\cati$, with $p_i$ mapping $\vec x$ to
$x_i$.  We routinely write $\vec x$ for tuples
${(x_i)}_{i \in |\cati|}$, and $x_i$ for their $i$th components.

The special case of a diagram over the opposite $(I, \sqsupseteq)$ of
a directed preordered set $(I, \sqsubseteq)$ is called a
\emph{projective system}.  
We call (canonical)
\emph{projective limit} any limit (the canonical limit) of a
projective system.  Explicitly, a projective system of topological
spaces, which we will write as
${(p_{ij} \colon X_j \to X_i)}_{i \sqsubseteq j \in I}$, is a
collection of spaces $X_i$ indexed by a directed preordered set
$(I, \sqsubseteq)$, with morphisms $p_{ij} \colon X_j \to X_i$ for all
indices $i \sqsubseteq j$ such that $p_{ii} = \identity {X_i}$ and
$p_{ij} \circ p_{jk} = p_{ik}$ for all $i \sqsubseteq j \sqsubseteq k$
in $I$.  We will familiarly call the maps $p_{ij}$ the \emph{bonding
  maps}.

The canonical projective limit $X, {(p_i)}_{i \in I}$ of
${(p_{ij} \colon X_j \to X_i)}_{i \sqsubseteq j \in I}$ is given by
$\{\vec x \in \prod_{i \in I} X_i \mid \forall i \sqsubseteq j \in I,
p_{ij} (x_j) = x_i\}$, with the subspace topology from the product,
and where $p_i$ is projection onto coordinate $i$.  Explicitly, a base
of that topology is given by the sets $p_i^{-1} (U_i)$, where
$i \in I$ and $U_i$ ranges over any base of the topology of $X_i$.
This can be deduced from Lemma~3.1 of \cite{JGL:kolmogorov} for
example, which states that every open subset $U$ of $X$ is the
directed union $\dcup_{i \in I} p_i^{-1} (U_i)$, where $U_i$ is the
largest open subset of $X_i$ such that $p_i^{-1} (U_i) \subseteq U$.
Directedness comes from the slightly stronger property that for all $i
\sqsubseteq j \in I$, $p_i^{-1} (U_i) \subseteq p_j^{-1} (U_j)$.

When $I$ has a countable cofinal subset, we talk about
$\omega$-projective systems and $\omega$-projective limits.  The
latter are free from certain apparent pathologies: for example, when
every space $X_i$ is non-empty and the maps $p_{ij}$ are surjective,
there are cases where the projective limit is empty
\cite{Henkin:invlimit,Waterhouse:projlim:empty}, but limits of such
$\omega$-projective systems of non-empty spaces with surjective
bonding maps are non-empty.

\section{The general setting}
\label{sec:general-setting}

\begin{definition}
  \label{defn:situation}
  For every endofunctor $T$ on $\Topcat$, we call  \emph{projective
    $T$-situation} the following data:
  \begin{itemize}
  \item a projective system
    ${(p_{ij} \colon X_j \to X_i)}_{i \sqsubseteq j \in I}$ of
    topological spaces;
  \item its canonical projective limit $X, {(p_i)}_{i \in I}$;
  \item the canonical projective limit $Z, {(q_i)}_{i \in I}$ of the
    projective system
    $(T {p_{ij}} \colon T {X_j} \to T {X_i})_{i \sqsubseteq j \in I}$;
  \item the unique continuous map $\varphi \colon T X \to Z$ such that
    $q_i \circ \varphi = T {p_i}$ for every $i \in I$, which we call
    the \emph{comparison map}.
  \end{itemize}
\end{definition}
Given a projective system and its canonical projective limit as in the
first two items above, the third item makes sense:
${(T {p_{ij}} \colon T {X_j} \to T {X_i})}_{i \sqsubseteq j \in I}$ is
a projective system, because $T$ is a functor; and $\varphi$ in the
fourth item is obtained by the universal property of $Z$.  We say that
$T$ \emph{preserves} the projective limit $X, {(p_i)}_{i \in I}$ if
and only if $\varphi$ is a homeomorphism.  In general, a functor $T$
preserves a limit $X, {(p_i)}_{i \in |\cati|}$ of a diagram
$F \colon \cati \to \catc$ if and only if
$TX, {(T{p_i})}_{i \in |\cati|}$ is a limit of $T \circ F$.
    
For various endofunctors $T$, we investigate when the comparison map
$\varphi$ is a homeomorphism.  We notice right away that $\varphi$ is
a topological embedding under some assumptions that may sound awfully
specific (Definition~\ref{defn:T:nice:S}), but which will be enough
for most of our needs.  The difficult part will then be to show that
$\varphi$ is surjective.

\begin{definition}
  \label{defn:T:nice:S}
  Let $R$ be a set.  An endofunctor $T$ on $\Topcat$ is
  \emph{$R$-nice} (or just \emph{nice}) if and only if for
  every topological space $X$, $TX$ has a subbase of open sets
  ${(B_X (r, U))}_{r \in R, U \in \Open X}$, with the following
  two properties:
  \begin{enumerate}
  \item $B_X (r, \_)$ is Scott-continuous from $\Open X$ to
    $\Open {(TX)}$ for every space $X$ and for every $r \in R$;
    \item for every continuous map $f \colon X \to Y$, for every
      $r \in R$, for every $V \in \Open Y$,
      ${(T f)}^{-1} (B_Y (r, V)) = B_X (r, f^{-1} (V))$.
  \end{enumerate}
\end{definition}

\begin{lemma}
  \label{lemma:T:projlim}
  Let $T$ be an $R$-nice endofunctor on $\Topcat$, where $R$ is a
  fixed set.  Given any projective $T$-situation as given in
  Definition~\ref{defn:situation}, the comparison map $\varphi$ is
  full.  If additionally $TX$ is $T_0$, then $\varphi$ is a
  topological embedding.
\end{lemma}
\begin{proof}
  We need to show that for every $r \in R$, for every $U \in \Open X$,
  $B_X (r, U)$ can be written as the inverse image of some open subset
  of $Z$ by $\varphi$.  For each $i \in I$, and every open subset $U$
  of $X$, there is a largest open subset $U_i$ of $X_i$ such that
  $p_i^{-1} (U_i) \subseteq U$.  We can write $U$ as
  $\dcup_{i \in I} p_i^{-1} (U_i)$.
  By property~1, $B_X (r, U)$ is the (directed) union of the sets
  $B_X (r, p_i^{-1} (U_i))$, $i \in I$.  By property~2,
  $B_X (r, p_i^{-1} (U_i)) = {(T {p_i})}^{-1} (B_{X_i} (r, U_i))$.
  Since $q_i \circ \varphi = T {p_i}$, the latter is equal to
  $\varphi^{-1} (q_i^{-1} (B_{X_i} (r, U_i)))$.  Hence
  $B_X (r, U) = \varphi^{-1} (\dcup_{i \in I} q_i^{-1} (B_{X_i} (r,
  U_i)))$.  This shows that $\varphi$ is full.  It is continuous, and
  we recall that any full, continuous map from a $T_0$ space is a
  topological embedding.  \qed
\end{proof}


\section{Continuous valuations}
\label{sec:cont-valu}

We start our series of applications with continuous valuation
functors.  This is a low-hanging fruit: the surjectivity of the
comparison map will come from \cite{JGL:kolmogorov}---under
appropriate assumptions---and then the comparison map will be a
homeomorphism by Lemma~\ref{lemma:T:projlim}.

Let $\creal$ be the set of extended non-negative real numbers
$\Rp \cup \{\infty\}$, with its usual ordering.  When needed, we will
consider it with its Scott topology, whose open sets are the intervals
$]t, \infty]$, $t \in \Rp$, plus $\emptyset$ and $\creal$ itself.

A \emph{continuous valuation} on a space $X$ is a map
$\nu \colon \Open X \to \creal$ that is \emph{strict}
($\nu (\emptyset)=0$), 
\emph{modular} (for all $U, V \in \Open X$,
$\nu (U) + \nu (V) = \nu (U \cup V) + \nu (U \cap V)$) and
Scott-continuous.  
We say that
$\nu$ is \emph{bounded} if and only if $\nu (X) < \infty$, a
\emph{probability} valuation if and only if $\nu (X)=1$, and a
\emph{subprobability} valuation if and only if $\nu (X) \leq 1$.  We
will also consider \emph{locally finite} continuous valuations $\nu$
on $X$, namely those such that for every $x \in X$, there is an open
neighborhood $U$ of $x$ such that $\nu (U) < \infty$, and \emph{tight}
valuations $\nu$, which are those such that for every $r \in \Rp$ and
every $U \in \Open X$ such that $r < \nu (U)$, there is a compact
saturated subset $Q$ of $X$ such that $Q \subseteq U$ and
$r \leq \nu (V)$ for every open neighborhood $V$ of $Q$
\cite[Definition~6.1]{JGL:kolmogorov}.

Every tight valuation is continuous, and the converse holds if $X$ is
consonant \cite[Lemma~6.2]{JGL:kolmogorov}.  We will omit the
definition of consonance for now, and we will state it when we
actually need it, see Section~\ref{sec:proj-limits-cons}.  The notion
arises from \cite{DGL:consonant}, where it was proved that every
regular \v{C}ech-complete space is consonant; every locally compact
space is consonant, too, as well as every LCS-complete space
\cite[Proposition 12.1]{dBGLJL:LCScomplete}.  A space is
\emph{LCS-complete} if and only if it is homeomorphic to a $G_\delta$
subspace of a locally compact sober space; $G_\delta$ is short for a
countable intersection of open subsets.  The class of LCS-complete
spaces includes all locally compact sober spaces, in particular all
continuous dcpos from domain theory, all of M. de Brecht's
quasi-Polish spaces \cite{JGL:kolmogorov} and therefore all Polish
spaces.

Continuous valuations are an alternative to measures that have become
popular in domain theory \cite{jones89,Jones:proba}.  The first
results that connected continuous valuations and measures are due to
Saheb-Djahromi \cite{saheb-djahromi:meas} and Lawson
\cite{Lawson:valuation}.  The current state of the art on this matter
is the following.  In one direction, every measure on the Borel
$\sigma$-algebra of $X$ induces a continuous valuation on $X$ by
restriction to the open sets, if $X$ is hereditarily Lindel\"of
(namely, if every directed family of open sets contains a cofinal
monotone sequence).  This is an easy observation, and one half of
Adamski's theorem \cite[Theorem~3.1]{Adamski:measures}, which states
that a space is hereditary Lindel\"of if and only if every measure on
its Borel $\sigma$-algebra restricts to a continuous valuation on its
open sets.  In the other direction, every continuous valuation on a
space $X$ extends to a measure on the Borel sets provided that $X$ is
an LCS-complete space \cite[Theorem~1]{dBGLJL:LCScomplete}.

Let $\Val X$ denote the space of continuous valuations on a space $X$,
with the following \emph{weak topology}.  That is defined by subbasic
open sets $[U > r] \eqdef \{\nu \in \Val X \mid \nu (U) > r\}$, where
$U \in \Open X$ and $r \in \Rp$.  (Those will be the sets $B_X (r, U)$
needed in order to apply Lemma~\ref{lemma:T:projlim}.)  We define its
subspace $\Val_b X$ of bounded continuous valuations, $\Val_1 X$ of
probability valuations and $\Val_{\leq 1} X$ (subprobability)
similarly.  The specialization ordering of each is the
\emph{stochastic ordering} $\leq$ given by $\nu \leq \nu'$ if and only
if $\nu (U) \leq \nu' (U)$ for every $U \in \Open X$; indeed,
$\nu \leq \nu'$ if and only if for every $U \in \Open X$, for every
$r \in \Rp$, $\nu \in [U > r]$ implies $\nu' \in [U > r]$.

The weak topology is also the coarsest topology that makes the
functions $\nu \mapsto \int h \,d\nu$ continuous from $\Val X$ to
$\creal$ (with its Scott topology), for each continuous map
$h \colon X \to \creal$, see \cite[Theorem~3.3]{Jung:scs:prob} where
this was proved for spaces of probability and subprobability
valuations; the proof is similar for arbitrary continuous valuations.
(Note that, since $\creal$ has the Scott topology, continuous maps
$h \colon X \to \creal$ are what are usually called lower
semicontinuous maps in the mathematical literature.)

For every continuous map $f \colon X \to Y$, for every
$\nu \in \Val X$, there is a continuous valuation $f [\nu] \in \Val Y$
defined by $f [\nu] (V) \eqdef \nu (f^{-1} (V))$ for every
$ V \in \Open Y$.  Additionally, $f [\nu]$ is bounded, resp.\ a
probability valuation, resp.\ a subprobability valuation, if $\nu$ is.
This defines the action on morphisms of endofunctors $\Val$, $\Val_b$,
$\Val_1$, and $\Val_{\leq 1}$ respectively on $\Topcat$.

In Proposition~\ref{prop:V:projlim:basic} below, we summarize the main
results of \cite{JGL:kolmogorov}, namely Theorem~4.2, Theorem~8.1,
Theorem~9.4 and Theorem~10.1 there.  This uses the following notions.

An \emph{embedding-projection pair}, or \emph{ep-pair} for short, is a
pair of continuous maps
$\xymatrix{X \ar@<1ex>[r]^e & Y \ar@<1ex>[l]^p}$ such that
$p \circ e = \identity X$ and $e \circ p \leq \identity Y$.  The
preordering used in the latter inequality is the pointwise preordering
on functions, where points are compared by the specialization
preordering of $Y$.  In that case, $p$ is called a \emph{projection}
of $Y$ onto $X$, and $e$ is the associated \emph{embedding}.
Generally, we call projection any continuous map $p \colon Y \to X$
that has an associated embedding $e$; if $Y$ is $T_0$, then $e$ is
uniquely determined.

An \emph{ep-system} is a functor from $(I, \sqsubseteq)^{op}$ to
$\Topcat^{\text{ep}}$, where $I, \sqsubseteq$ is a directed preorder
and $\Topcat^{\text{ep}}$ is the category whose objects are
topological spaces, and whose morphisms are the ep-pairs.  Explicitly,
this is given by:
\begin{enumerate*}[label=(\roman*)]
\item a family of objects $X_i$ of $\catc$, $i \in I$;
\item ep-pairs $\xymatrix{X_i \ar@<1ex>[r]^{e_{ij}} & X_j
    \ar@<1ex>[l]^{p_{ij}}}$ for all $i \sqsubseteq j$ in $I$,
  satisfying:
\item $e_{ii} = p_{ii} = \identity {X_i}$ for every $i \in I$,
\item $p_{ij} \circ p_{jk} = p_{ik}$, and
\item $e_{jk} \circ e_{ij} = e_{ik}$ for all
  $i \sqsubseteq j \sqsubseteq k$ in $I$.
\end{enumerate*}
Every ep-system has an underlying projective system
${(p_{ij} \colon X_j \to X_i)}_{i \sqsubseteq j \in I}$, and we will
implicitly see every ep-system as a projective system this way.  This
is an abuse of language, and a projective system whose bonding maps
$p_{ij}$ are projections may be such that the matching embedding
$e_{ij}$ fail to satisfy (iii) and (v); this pathological situation
does not happen if every $X_i$ is $T_0$, since in that case every
projection $p_{ij}$ has a unique associated embedding.

A \emph{proper} map is a closed perfect map, where a \emph{closed} map
$f \colon X \to Y$ is one such that $\dc f [F]$ is closed for every
closed subset $F$ of $X$ (not $f [F]$, as one usually requires in
topology), and a \emph{perfect} map $f$ is such that $f^{-1} (Q)$ is
compact saturated for every compact saturated subset $Q$ of $Y$; this
definition of proper maps, which is well-suited to a non-Hausdorff
setting, originates from \cite[Definition~VI-6.20]{GHKLMS:contlatt}.
We will study proper maps in depth in
Section~\ref{sec:proper-maps-quasi}.
\begin{proposition}[\cite{JGL:kolmogorov}]
  \label{prop:V:projlim:basic}
  Let ${(p_{ij} \colon X_j \to X_i)}_{i \sqsubseteq j \in I}$ be a
  projective system of topological spaces, with canonical projective
  limit $X, {(p_i)}_{i \in I}$.  Let $\nu_i$ be continuous valuations
  on $X_i$ for each $i \in I$, and let us assume that
  $\nu_i = p_{ij} [\nu_j]$ for all $i \sqsubseteq j \in I$.  If:
  \begin{enumerate}
  \item the given projective system is an ep-system,
  \item or $I$ has a countable cofinal subset and every $X_i$ is
    locally compact and sober,
  \item or $I$ has a countable cofinal subset, every $\nu_i$ is
    locally finite, and every $X_i$ is LCS-complete,
  \item or every $p_{ij}$ is proper, every $X_i$ is sober, and every
    $\nu_i$ is tight,
  \end{enumerate}
  then there is a unique continuous valuation $\nu$ on $X$ such that
  for every $i \in I$, $\nu_i = p_i [\nu]$.
\end{proposition}

As promised, we apply Lemma~\ref{lemma:T:projlim}, with $R \eqdef \Rp$
and $B_X (r, U) \eqdef [U > r]$, and we obtain the following.
\begin{proposition}
  \label{prop:V:projlim}
  Let $T$ be $\Val$, $\Val_b$, $\Val_1$ or $\Val_{\leq 1}$.  The
  comparison map $\varphi \colon T X \to Z$ of any projective
  $T$-situation is a topological embedding.
\end{proposition}
\begin{proof}
  We check the assumptions of Lemma~\ref{lemma:T:projlim}.  We start
  with property~1 of Definition~\ref{defn:T:nice:S}.  Let $r \in \Rp$.
  For all $U, V \in \Open X$, $U \subseteq V$ implies
  $[U > r] \subseteq [V > r]$, since $\nu (U) \leq \nu (V)$ for every
  continuous valuation $\nu$ on $X$, as part of the requirement of
  Scott-continuity.  For every directed family ${(U_i)}_{i \in I}$ of
  open subsets of $X$ with union $U$, for every continuous valuation
  $\nu$ on $X$, $\nu \in [U > r]$ if and only if $\nu (U) > r$, if and
  only if $\dsup_{i \in I} \nu (U_i) > r$ since $\nu$ is
  Scott-continuous, if and only if $\nu (U_i) > r$ for some $i \in I$,
  if and only if $\nu \in \dcup_{i \in I} [U_i > r]$.  For property~2,
  we note that for every continuous map $f \colon X \to Y$, for every
  open subset $V$ of $Y$ and every $r \in \Rp$,
  ${(\Val f)}^{-1} ([V > r]) = \{\nu \in \Val X \mid f[\nu] (V) > r\}
  = \{\nu \in \Val X \mid \nu (f^{-1} (V)) > r\} = [f^{-1} (V) > r]$.
  Finally, $T X$ is $T_0$, because its specialization preordering is
  the stochastic ordering, which is antisymmetric.  \qed
\end{proof}

\begin{theorem}
  \label{thm:V:projlim}
  Let ${(p_{ij} \colon X_j \to X_i)}_{i \sqsubseteq j \in I}$ be a
  projective system of topological spaces, with canonical projective
  limit $X, {(p_i)}_{i \in I}$.  Let $T$ be one of the functors
  $\Val$, $\Val_b$, $\Val_1$, or $\Val_{\leq 1}$.  If:
  \begin{enumerate}
  \item the projective system is an ep-system,
  \item or $I$ has a countable cofinal subset and each $X_i$ is locally
    compact sober,
  \item or $I$ has a countable cofinal subset, $T$ is $\Val_b$,
    $\Val_1$ or $\Val_{\leq 1}$, and each $X_i$ is LCS-complete,
  \item or every $X_i$ is consonant sober and every $p_{ij}$ is a
    proper map,
  \end{enumerate}
  then
  ${(T {p_{ij}} \colon T {X_j} \to T {X_i})}_{i \sqsubseteq j \in I}$
  is a projective system of topological spaces, and
  $T X, {(T {p_i})}_{i \in I}$ is its projective limit, up to
  homeomorphism.
\end{theorem}
\begin{proof}
  We consider the map $\varphi \colon T X \to Z$ of
  Proposition~\ref{prop:V:projlim}, and we claim that it is
  surjective.  In other words, given ${(\nu_i)}_{i \in I}$ in $Z$, we
  claim that there is a $\nu \in T X$ such that
  $\varphi (\nu) = {(\nu_i)}_{i \in I}$.  We obtain such a $\nu$ in $\Val X$
  by Proposition~\ref{prop:V:projlim:basic}.  In case~3, we require
  $T$ not to be $\Val$, so as to make sure that every $\nu_i$ is
  bounded, hence locally finite.  In case~4, we use the fact that
  every continuous valuation is tight on a consonant space.

  Now that we have built $\nu$, it remains to show that it is not just
  in $\Val X$, but in $T X$.  If $T = \Val_{\leq 1}$, then we pick any
  $i \in I$.  Then $\nu_i (X_i) \leq 1$, and therefore
  $\nu (X) = \nu (p_i^{-1} (X_i)) = p_i [\nu] (X_i) = \nu_i (X_i) \leq
  1$; similarly if $T$ is $\Val_1$ or $\Val_b$.  \qed
\end{proof}


\section{Ep-systems}
\label{sec:ep-systems}

The case of ep-systems is not specific to continuous valuation
functors.  It is well-known that, in categories of continuous dcpos
\cite[Proposition~5.2.4]{AJ:domains} or even of general dcpos
\cite[Theorem~IV-5.5]{GHKLMS:contlatt}, $T$ preserves projective
limits of ep-systems provided that $T$ is locally continuous.

Local continuity does not make sense in $\Topcat$, because $\Topcat$
is not even order-enriched (it is \emph{preorder}-enriched).
Restricting to the full subcategory of monotone convergence spaces
would provide us with a dcpo-enriched category on which the notion of
local continuity would make sense, but that is not necessary.


\begin{proposition}
  \label{prop:ep}
  Let $T$ be a nice endofunctor on $\Topcat$.  Given any projective
  $T$-situation as given in Definition~\ref{defn:situation}, whose
  projective system is an ep-system, and such that $TX_i$ is $T_0$ for
  every $i \in I$ and $TX$ is a monotone convergence space, the
  comparison map $\varphi$ is a homeomorphism.
\end{proposition}
\begin{proof}
  We take all notations from Definition~\ref{defn:situation} and
  Definition~\ref{defn:T:nice:S}.
  Let $e_{ij}$ be embeddings associated with each of the projections
  $p_{ij}$.  By \cite[Lemma~4.1]{JGL:kolmogorov}, each $p_i$ is a
  projection, and there are associated embeddings
  $e_i \colon X_i \to X$, such that $e_j \circ e_{ij} = e_i$ for all
  $i \sqsubseteq j \in I$.  Moreover, for each open subset $U$ of $X$,
  ${((e_i \circ p_i)^{-1} (U))}_{i \in I, \sqsubseteq}$ is a monotone
  net in $\Open X$ 
  and its union is equal to $U$.

  Using Lemma~\ref{lemma:T:projlim}, it remains to show that $\varphi$
  is surjective.
  
  Let ${(t_i)}_{i \in I}$ be any element of $Z$, that is, each $t_i$
  is in $T X_i$ and $T {p_{ij}} (t_j) = t_i$ for all
  $i \sqsubseteq j \in I$.  We claim that the elements
  $T {e_i} (t_i) \in TX$ form a monotone net, namely that
  $T {e_i} (t_i) \leq T {e_j} (t_j)$ for all $i \sqsubseteq j \in I$.
  In order to see this, it suffices to show that for every $r \in R$,
  for every $U \in \Open X$, if $T {e_i} (t_i) \in B_X (r, U)$ then
  $T {e_j} (t_j) \in B_X (r, U)$.  Since $t_i = T {p_{ij}} (t_j)$, the
  assumption $T {e_i} (t_i) \in B_X (r, U)$ means that
  $T {(e_i \circ p_{ij})} (t_j) \in B_X (r, U)$, namely,
  $t_j \in (T {(e_i \circ p_{ij})})^{-1} (B_X (r, U)) = B_{X_j} (r,
  (e_i \circ p_{ij})^{-1} (U))$.  But $e_i \circ p_{ij} \leq e_j$,
  since $e_j \circ e_{ij} = e_i$,
  $e_{ij} \circ p_{ij} \leq \identity {X_j}$ and $e_j$ is monotonic
  (being continuous).  Since $U$ is upwards-closed, it follows that
  $(e_i \circ p_{ij})^{-1} (U) \subseteq e_j^{-1} (U)$.  Using the
  fact that $B_{X_j} (r, \_)$ is Scott-continuous, hence monotonic, we
  obtain that
  $B_{X_j} (r, (e_i \circ p_{ij})^{-1} (U)) \subseteq B_{X_j} (r,
  e_j^{-1} (U))$.  Therefore
  $t_j \in B_{X_j} (r, e_j^{-1} (U)) = (T {e_j})^{-1} (B_X (r, U))$,
  showing that $T {e_j} (t_j) \in B_X (r, U)$.

  Since $TX$ is a monotone convergence space, the monotone net
  ${(T {e_j} (t_j))}_{j \in I, \sqsubseteq}$ has a supremum, which we
  call $t$.  It remains to show that
  $\varphi (t) = {(t_i)}_{i \in I}$, or equivalently, that
  $T {p_i} (t) = t_i$ for every $i \in I$.  The difficult part is to
  show that $T {p_i} (t) \leq t_i$.  In order to see this, let
  $B_{X_i} (r, U_i)$ ($r \in R$, $U_i \in \Open {X_i}$) be any
  subbasic open set containing $T {p_i} (t)$.  Then
  $t \in (T {p_i})^{-1} (B_{X_i} (r, U_i)) = B_X (r, p_i^{-1} (U_i))$.
  The latter is Scott-open since $TX$ is a monotone convergence space,
  so $T {e_j} (t_j) \in B_X (r, p_i^{-1} (U_i))$ for some $j \in I$.
  Let us pick $k \in I$ such that $i, j \sqsubseteq k$.  Then
  $t_j = T {p_{jk}} (t_k)$, so
  $T (e_j \circ p_{jk}) (t_k) \in B_X (r, p_i^{-1} (U_i))$, namely
  $t_k \in B_{X_k} (r, (p_i \circ e_j \circ p_{jk})^{-1} (U_i))$.  Now
  $p_i \circ e_j = p_{ik} \circ p_k \circ e_k \circ e_{jk} = p_{ik}
  \circ e_{jk}$, and therefore
  $p_i \circ e_j \circ p_{jk} = p_{ik} \circ e_{jk} \circ p_{jk} \leq
  p_{ik}$, since $e_{jk} \circ p_{jk} \leq \identity {X_k}$ and
  $p_{ik}$ is (continuous hence) monotonic.  Using the fact that $U_i$
  is upwards-closed, it follows that
  $(p_i \circ e_j \circ p_{jk})^{-1} (U_i) \subseteq p_{ik}^{-1}
  (U_i)$.  Next, $B_{X_k} (r, \_)$ is Scott-continuous hence
  monotonic, so $t_k \in B_{X_k} (r, p_{ik}^{-1} (U_i))$.  This means
  that $T {p_{ik}} (t_k) \in B_{X_i} (r, U_i)$, namely that
  $t_i \in B_{X_i} (r, U_i)$.  As $r$ and $U_i$ are arbitrary, we
  conclude that $T {p_i} (t) \leq t_i$.

  The reverse inequality is easier: $T {e_i} (t_i) \leq t$, so
  $t_i = T (p_i \circ e_i) (t_i) = T {p_i} (T {e_i} (t_i)) \leq T
  {p_i} (t)$, using the fact that $T {p_i}$ is continuous hence
  monotonic.  Since $TX_i$ is $T_0$, we conclude that
  $T {p_i} (t) = t_i$, for every $i \in I$, hence that
  $\varphi (t) = {(t_i)}_{i \in I}$.  \qed
\end{proof}

With this, we obtain another proof of
Theorem~\ref{prop:V:projlim:basic}, item~1, when $T$ is equal to
$\Val$, $\Val_1$, or $\Val_{\leq 1}$ (not $\Val_b$).  It suffices to
observe that $TY$ is sober, hence a $T_0$ space and a monotone
convergence space, for any space $Y$.  The argument is due to R. Tix
\cite[Satz~5.4]{Tix:bewertung}, following ideas by R. Heckmann (see
\cite[Section~2.3]{heckmann96}), in the case where $T=\Val$.  When
$T=\Val_{\leq 1}$ or $T=\Val_1$, we rest on the following remark.
\begin{remark}
  \label{rem:sat:sober}
  The sober subspaces of a sober space $Z$ coincide with the subsets
  that are closed in the strong topology on $Z$
  \cite[Corollary~3.5]{KL:dcompl}.  The latter is also known as the
  Skula topology, and is the smallest one generated by the original
  topology on $Z$ and all the downwards-closed subsets.  In
  particular, any closed subspace of a sober space is sober, any
  saturated subspace of a sober space is sober.

  Hence $\Val_{\leq 1} Y$ is sober, being equal to the closed subspace
  $\Val Y \diff [Y > 1]$ of $\Val Y$, and $\Val_1 Y$ is sober, being
  upwards-closed in $\Val_{\leq 1} Y$.
\end{remark}



\section{The Smyth hyperspace}
\label{sec:smyth-hyperspace}

For every topological space $X$, let $\Smyth_0 X$ be the set of all
compact saturated subsets of $X$ (resp., $\Smyth X$ be its subset of
non-empty compact saturated subsets).  The \emph{upper Vietoris}
topology on that set has basic open subsets $\Box U$ consisting of
those compact saturated subsets of $X$ (resp., and non-empty) that are
included in $U$, where $U$ ranges over the open subsets of $X$.  We
write $\SVz X$ (resp., $\SV X$) for the resulting topological space.
Its specialization ordering is reverse inclusion $\supseteq$.  In
certain cases, and notably in Section~\ref{sec:interm-pres-local}, we
will disambiguate between $\Box U$ as a basic open subset of $\SV X$,
and as a basic open subset of $\SVz X$, and we will write $\Box_0 U$
in the latter case.

The $\SV$ and $\SVz$ constructions have been studied by a number of
people, starting with Smyth \cite{Smyth:power:pred}, and later by
Schalk \cite[Section~7]{schalk:diss} who studied not only this, but
also the variant with the Scott topology, and a localic counterpart.
See also \cite[Sections~6.2.2, 6.2.3]{AJ:domains} or
\cite[Section~IV-8]{GHKLMS:contlatt}, where the accent is rather on
the Scott topology of $\supseteq$.

There is a $\SVz$ endofunctor, and also an $\SV$ endofunctor, on the
category $\Topcat$ of topological spaces.  Its action $\SV f$ on
morphisms $f \colon X \to Y$ is the function that maps every
$Q \in \SVz X$ to $\upc f [Q] \in \SVz Y$ (and similarly with $\SV$).
Here and later, we use the notation $f [Q]$ to denote the image of $Q$
under $f$.  This endofunctor is part of a monad whose unit
$\eta^{\Smyth}_X \colon X \to \SVz X$ maps every $x \in X$ to $\upc x$
and whose multiplication
$\mu^{\Smyth}_X \colon \SVz {\SVz X} \to \SVz X$ maps $\mathcal Q$ to
$\bigcup \mathcal Q$ \cite[Proposition~7.21]{schalk:diss}, and
similarly with $\SV$.


Which projective limits are preserved by those endofunctors is made
easy by relying on \emph{Steenrod's theorem}, as stated by Fujiwara
and Kato \cite[Theorem~2.2.20]{FK:rigid:geometry}: every projective
limit, taken in $\Topcat$, of compact sober spaces is compact and
sober.  A very useful lemma that comes naturally with that result is
the following, which appears as Lemma~7.5 in \cite{JGL:kolmogorov}.
\begin{lemma}
  \label{lemma:steenrod:open}
  Let $Q, {(p_i)}_{i \in I}$ be the canonical projective limit of a
  projective system
  ${(p_{ij} \colon Q_j \to Q_i)}_{i \sqsubseteq j \in I}$ of compact
  sober spaces.  For every $i \in I$, for every open neighborhood $U$
  of $\upc p_i [Q]$ in $Q_i$, there is an index $j \in I$ such that
  $i \sqsubseteq j$ and $\upc p_{ij} [Q_j] \subseteq U$.
\end{lemma}

Let us say that a map $f \colon X \to Y$ between topological spaces is
\emph{almost surjective} if and only if $\upc f [X] = Y$.
\begin{lemma}
  \label{lemma:steenrod:almostsurj}
  Let $Q, {(p_i)}_{i \in I}$ be the canonical projective limit of a
  projective system
  ${(p_{ij} \colon Q_j \to Q_i)}_{i \sqsubseteq j \in I}$ of compact
  sober spaces.  If the bonding maps $p_{ij}$ are almost surjective,
  then the cone maps $p_i$ are also almost surjective, $i \in I$.
\end{lemma}
\begin{proof}
  Let us imagine that $p_i$ is not almost surjective.  There is a
  point $x \in Q_i$ that is not in $\upc p_i [Q]$.  Since the latter
  is saturated, hence equal to the intersection of its open
  neighborhoods, there is an open neighborhood $U$ of $\upc p_i [Q]$
  that does not contain $x$.  By Lemma~\ref{lemma:steenrod:open},
  there is a $j \in I$ above $i$ such that
  $\upc p_{ij} [Q_j] \subseteq U$.  That is impossible: since $p_{ij}$
  is almost surjective, $\upc p_{ij} [Q_j]=Q_i$, but $U$ is a
  proper subset of $Q_i$.  \qed
\end{proof}

\begin{proposition}
  \label{prop:Q:projlim}
  The comparison map $\varphi \colon \SV X \to Z$ of any projective
  $\SV$-situation is a topological embedding.  Similarly with $\SVz$
  in lieu of $\SV$.
\end{proposition}
\begin{proof}
  We only deal with $\SV$.  In order to apply
  Lemma~\ref{lemma:T:projlim}, we verify that $\SV$ is $R$-nice with
  $R$ a one-element set $\{*\}$.  We let $B_X (*, U) \eqdef \Box U$.
  Property~1 of Definition~\ref{defn:T:nice:S} boils down to the fact
  that $U \subseteq V$ implies $\Box U \subseteq \Box V$, which is
  clear, plus the fact that for every directed family
  ${(U_i)}_{i \in I}$ of open subsets of $X$ with union $U$,
  $\Box U = \dcup_{i \in I} \Box {U_i}$.  In order to show that, we
  note that for every $Q \in \SV X$, $Q \in \Box U$ if and only if
  $Q \subseteq \dcup_{i \in I} U_i$, if and only if $Q \subseteq U_i$
  for some $i \in I$ (because $Q$ is compact), if and only if
  $Q \in \dcup_{i \in I} \Box {U_i}$.  As for property~2, for every
  continuous map $f \colon X \to Y$, for every $V \in \Open Y$,
  ${(\SV f)}^{-1} (\Box V) = \{Q \in \SV X \mid \upc f [Q] \subseteq
  V\} = \Box {f^{-1} (V)}$.  Finally, $\SV X$ is $T_0$, because its
  specialization preordering $\supseteq$ is an ordering.  \qed
\end{proof}

We use all this to show that $\SV$ and $\SVz$ preserve projective
limits of \emph{sober} spaces.
\begin{theorem}
  \label{thm:Q:projlim}
  Let ${(p_{ij} \colon X_j \to X_i)}_{i \sqsubseteq j \in I}$ be a
  projective system of topological spaces, with canonical projective
  limit $X, {(p_i)}_{i \in I}$.  If every $X_i$ is sober, then
  ${(\SV {p_{ij}} \colon \SV {X_j} \to \SV {X_i})}_{i \sqsubseteq j
    \in I}$ is a projective system of topological spaces, and
  $\SV X, {(\SV {p_i})}_{i \in I}$ is its projective limit, up to
  homeomorphism.  Similarly with $\SVz$ in lieu of $\SV$.
\end{theorem}
\begin{proof}
  We only deal with the case of $\SV$, and we reuse the notations of
  Definition~\ref{defn:situation}.  We use
  Proposition~\ref{prop:Q:projlim}, so the comparison map
  $\varphi \colon \SV X \to Z$ is a topological embedding.  It remains
  to show that $\varphi$ is surjective.  Let
  $\vec Q \eqdef {(Q_i)}_{i \in I}$ be an element of $Z$.  The family
  ${(p_{ij|Q_j} \colon Q_j \to Q_i)}_{i \sqsubseteq j \in I}$ is a
  projective system of non-empty compact spaces, where each $Q_i$ is
  given the subspace topology of $X_i$.  Since $Q_i$ is a saturated
  subset of a sober space, it is itself sober, by
  Remark~\ref{rem:sat:sober}.  By Steenrod's theorem, the canonical
  projective limit $Q$ is non-empty, compact and sober.  $Q$ is the
  collection of tuples $\vec x \eqdef {(x_i)}_{i \in I}$ where each
  $x_i$ is in $Q_i$ and $p_{ij} (x_j)=x_i$ for all $i \sqsubseteq j$
  in $I$.  In particular, $Q$ is a non-empty subset of $X$.  Being
  compact as a subspace, it is also compact as a subset.  It is also
  upwards-closed, because each $Q_i$ is upwards-closed.

  Now we observe that for all $i \sqsubseteq j \in I$,
  $Q_i = \SV {p_{ij}} (Q_j) = \upc p_{ij} [Q_j]$.  Therefore,
  $p_{ij|Q_j}$ is an almost surjective map from $Q_j$ to $Q_i$, in the
  sense of Lemma~\ref{lemma:steenrod:almostsurj}.  That lemma implies
  that the cone maps $q_i \colon Q \to Q_i$ (mapping every
  $\vec x \in Q$ to $x_i \in Q_i$) are almost surjective, too.  Since
  $q_i$ coincides with $p_i$ on $Q$, $Q_i$ is also equal to
  $\upc p_i [Q] = \SV {p_i} (Q)$.  Since this holds for every
  $i \in I$, $\varphi (Q) = {(Q_i)}_{i \in I}$, showing that $\varphi$ is
  surjective.  Now $\varphi$ is a surjective topological embedding, hence a
  homeomorphism.  \qed
\end{proof}

\begin{corollary}
  \label{corl:Q:projlim}
  Let ${(p_{ij} \colon X_j \to X_i)}_{i \sqsubseteq j \in I}$ be a
  projective system of topological spaces, with canonical projective
  limit $X, {(p_i)}_{i \in I}$, and let every $X_i$ be sober.  Given
  any family of (resp., non-empty) compact saturated subsets $Q_i$ of
  $X_i$, for each $i \in I$, such that $Q_i = \upc p_{ij} [Q_{ij}]$
  for all $i \sqsubseteq j \in I$, there is a unique (resp.,
  non-empty) compact saturated subset $Q$ of $X$ such that
  $Q_i = \upc p_i [Q]$ for every $i \in I$.  \qed
\end{corollary}

The assumption of sobriety is necessary, as the following
counter-example, due to A. H. Stone
\cite[Example~3]{Stone:limproj:compact}, shows.
\begin{example}
  \label{exa:Stone}
  We let $X_n$ be $\nat$ for every natural number $n$ and
  $p_{mn} \colon X_n \to X_m$ be the identity map for all $m \leq n$.
  The topology on $X_n$ is obtained by declaring a subset $C$ closed
  if and only if $C \cap \{n, n+1, \cdots\}$ is finite or equal to the
  whole of $\{n, n+1, \cdots\}$.  In other words, $X_n$ is isomorphic
  to the disjoint sum of $\{0, 1, \cdots, n\}$ with the discrete
  topology with $\{n, n+1, \cdots\}$ with the cofinite topology.  Then
  each $X_n$ is compact (even Noetherian and $T_1$), but its
  projective limit is $\nat$ with the discrete topology, which is not.
  Hence, taking $Q_m \eqdef X_m$ for every $m \in \nat$, the
  conclusion of Corollary~\ref{corl:Q:projlim} would fail: the only
  possible subset $Q$ of $X$ such that $Q_m = \upc p_m [Q]$ for every
  $m \in \nat$ is $X$ itself, and it is not compact.  In other words,
  the topological embedding of $\SV X$ into the projective limit $Z$
  obtained in Proposition~\ref{thm:Q:projlim} is not surjective in
  this example.
\end{example}

\section{The Hoare hyperspaces}
\label{sec:hoare-hyperspaces}

For every topological space $X$, let $\HVz X$ be the set of all closed
subsets of $X$ (resp., $\HV X$ the set of all non-empty closed subsets
of $X$).  We take as a subbase the sets $\Diamond U$ of those closed
subsets of $X$ that intersect $U$, for every $U \in \Open X$.  The
resulting topology is called the \emph{lower Vietoris} topology, and
its specialization ordering is inclusion $\subseteq$.  This is a very
classical space in topology, although it is more often studied in
connection with other topologies, such as the (full) Vietoris
topology.  In domain theory, one usually considers the Scott topology
of inclusion, see \cite[Sections~6.2.2, 6.2.3]{AJ:domains} or
\cite[Section~IV-8]{GHKLMS:contlatt}, yielding the \emph{Hoare
  powerdomain}.  As with Smyth hyperspaces, Schalk was one of the
first to study the Hoare hyperspace $\HV X$, in connection with the
Hoare powerdomain, and their localic counterpart
\cite[Section~6]{schalk:diss}.

There are $\HVz$ and $\HV$ endofunctors on $\Topcat$, whose action
$\HV f$ on morphisms $f \colon X \to Y$ maps every closed subset $F$
of $X$ to $cl (f [F])$.  This endofunctor is part of a monad whose
unit $\eta^{\Hoare}_X \colon X \to \SVz X$ maps every $x \in X$ to
$\dc x$ and whose multiplication
$\mu^{\Hoare}_X \colon \SVz {\SVz X} \to \SVz X$ maps $\mathcal F$ to
$cl (\bigcup \mathcal F)$.

\begin{proposition}
  \label{prop:H:projlim}
  The comparison map $\varphi \colon \HV X \to Z$ of any projective
  $\HV$-situation is a topological embedding.  Similarly with $\HVz$
  in lieu of $\HV$.
\end{proposition}
\begin{proof}
  We apply Lemma~\ref{lemma:T:projlim}, and to this end we verify that
  $\HV$ is $R$-nice with a one-element set $\{*\}$ for $R$.  We let
  $B_X (*, \allowbreak U) \eqdef \Diamond U$.  Property~1 of
  Definition~\ref{defn:T:nice:S} stems from the fact that the
  $\Diamond$ operator commutes with arbitrary unions.  For property~2,
  for every continuous map $f \colon X \to Y$, for every
  $V \in \Open Y$,
  ${(\HV f)}^{-1} (\Diamond V) = \{F \in \HV X \mid cl (f [F]) \cap V
  \neq \emptyset\} = \{F \in \HV X \mid f [F] \cap V \neq \emptyset\}
  = \{F \in \HV X \mid F \cap f^{-1} (V) \neq \emptyset\} = \Diamond
  {f^{-1} (V)}$.  Finally, $\HV X$ is $T_0$, because its
  specialization preordering $\subseteq$ is an ordering.  Similarly
  with $\HVz$.  \qed
\end{proof}


Dealing with $\HVz$ and $\HV$ is more difficult than dealing with
$\SVz$ and $\SV$.  In order to make the approach simpler, we take a
detour through continuous valuations.  Doing so, we run the risk of
obtaining suboptimal results, but we will argue through a list of
examples that they are still reasonably tight.
\begin{definition}
  \label{defn:inf:egame}
  For every topological space $X$, for every closed subset $F$ of $X$,
  let $\infty.\egame F$ map every $U \in \Open X$ to $\infty$ if $U$
  intersects $F$, and to $0$ otherwise.
\end{definition}
The notation comes from the example games $\egame F$ of
\cite{JGL-icalp07}, multiplied by the scalar $\infty$.
\begin{lemma}
  \label{lemma:inf:egame}
  Let $X$ be a topological space.
  \begin{enumerate}
  \item For every $F \in \HVz X$, $\infty.\egame F$ is a tight
    valuation, hence a continuous valuation.
  \item The map $F \mapsto \infty.\egame F$ is a topological embedding
    of $\HV X$ (resp., $\HVz X$) into $\Val X$.
  \item There is a natural transformation $\infty.\egame\relax$ from
    $\HV$ (resp., $\HVz$) to $\Val$, defined on each topological space
    $X$ as $F \mapsto \infty.\egame F$.
  \end{enumerate}
\end{lemma}
\begin{proof}
  1. It is clear that $\nu \eqdef \infty.\egame F$ is strict and
  monotonic.  We claim that it is modular.  Let $U, V \in \Open X$.
  If $U$ or $V$ intersects $F$, then $U \cup V$ does, too, so both
  $\nu (U \cup V) + \nu (U \cap V)$ and $\nu (U) + \nu (V)$ are equal
  to $\infty$.  Otherwise, $F$ cannot intersect $U \cup V$, and
  certainly not $U \cap V$, so both $\nu (U \cup V) + \nu (U \cap V)$
  and $\nu (U) + \nu (V)$ are equal to $0$.  As far as tightness is
  concerned, let $r \in \Rp$ and $U \in \Open X$ such that
  $r < \infty.\egame F (U)$.  We wish to find a compact saturated
  subset $Q$ of $X$ included in $U$ such that
  $r \leq \infty.\egame F (V)$ for every open neighborhood $V$ of $Q$.
  The intersection $U \cap F$ is non-empty, since
  $0 \leq r < \infty.\egame F (U)$.  We pick $x$ from $U \cap F$, and
  we define $Q$ as $\upc x$.  Every open neighborhood $V$ of $Q$
  intersects $F$ at $x$, so $\infty.\egame F (V) = \infty \geq r$.


  2. For every $r \in \Rp$ and every open subset $U$ of $X$,
  $\infty.\egame F \in [U > r]$ if and only if $F$ is in $\Diamond U$,
  and this shows continuity.  This also shows that this map is full,
  since every subbasic open set $\Diamond U$ is the inverse image of,
  say, $[U > 0]$.  Since $\HV X$ (resp., $\HVz X$) is $T_0$, the map
  $F \mapsto \infty.\egame F$ is a topological embedding.

  3.  We only deal with $\HV$.  Naturality means that for every
  continuous map $f \colon X \to Y$, for every closed subset $F$ of
  $X$, $\Val f (\infty.\egame F) = \infty.\egame {\HV f (F)}$.  For
  every open subset $V$ of $Y$,
  $\Val f (\infty.\egame F) (V) = \infty.\egame F (f^{-1} (V))$ is
  equal to $\infty$ if $f^{-1} (V)$ intersects $F$, and to $0$
  otherwise, while $\infty.\egame {\HV f (F)} (V)$ equals $\infty$ if
  $V$ intersects $\HV f (F)$, and to $0$ otherwise.  Now
  $\HV f (F) = cl (f [F])$ intersects $V$ if and only if $f [F]$
  intersects $V$, if and only if $F$ intersects $f^{-1} (V)$.  \qed
\end{proof}

In the other direction, every continuous valuation $\nu$ on a space
$X$ has a \emph{support} $\supp \nu$, defined as the smallest closed
subset $F$ of $X$ such that $\nu (X \diff F)=0$.  Showing that this
exists is easy.  We define the family $\mathcal D$ of open subsets $U$
of $X$ such that $\nu (U)=0$, and we observe that $\mathcal D$ is
non-empty (since $\emptyset \in \mathcal D$), Scott-closed (because
$\nu$ is Scott-continuous), and directed.  For the latter, it is
enough to notice that for all $U, V \in \mathcal D$,
$\nu (U \cup V) + \nu (U \cap V) = \nu (U) + \nu (V) = 0$, so
$\nu (U \cup V) = 0$, whence $U \cup V \in \mathcal D$.  Hence the
supremum $U_\infty$ of $\mathcal D$ is in $\mathcal D$, and then
$\supp \nu \eqdef U_\infty$.  For now, we will be content to note the
following.
\begin{lemma}
  \label{lemma:egame:inf}
  Every continuous valuation $\nu$ on a topological space $X$ that
  only takes the values $0$ and $\infty$ is equal to $\infty.\egame F$
  for a unique closed subset $F$ of $X$, and $F \eqdef \supp \nu$.
\end{lemma}
\begin{proof}
  For every $U \in \Open X$, by definition $\nu (U)=0$ if and only if
  $U$ does not intersect $\supp \nu$.  Since $\nu (U) \neq 0$ is
  equivalent to $\nu (U) = \infty$, $\nu = \egame {\supp \nu}$.  As
  for the uniqueness of $F$, let $F$ and $F'$ be two closed subsets of
  $X$ such that $\infty.\egame F = \infty.\egame {F'}$.  Applying both
  sides to $X \diff F$, we obtain that
  $F' \cap (X \diff F)=\emptyset$, hence $F' \subseteq F$, and using
  $X \diff F'$ instead gives us $F \subseteq F'$. \qed
\end{proof}

\begin{theorem}
  \label{thm:H:projlim}
  Let ${(p_{ij} \colon X_j \to X_i)}_{i \sqsubseteq j \in I}$ be a
  projective system of topological spaces, with canonical projective
  limit $X, {(p_i)}_{i \in I}$.  If:
  \begin{enumerate}
  \item the projective system is an ep-system,
  \item or every $X_i$ is sober and every $p_{ij}$ is a proper map,
  \item or $I$ has a countable cofinal subset and each $X_i$ is
    locally compact sober,
  \end{enumerate}
  then
  ${(\HV {p_{ij}} \colon \HV {X_j} \to \HV {X_i})}_{i \sqsubseteq j
    \in I}$ is a projective system of topological spaces, and
  $\HV X, {(\HV {p_i})}_{i \in I}$ is its projective limit, up to
  homeomorphism.  Similarly with $\HVz$ in lieu of $\HV$.
\end{theorem}
\begin{proof}
  We only deal with $\HV$.  Let $\varphi \colon \HV X \to Z$ be the
  comparison map; that is a topological embedding by
  Proposition~\ref{prop:H:projlim}, and it remains to show that it is
  surjective.  Explicitly, $\varphi$ maps every $F \in \Hoare X$ to
  ${(\HV {p_i} (F)))}_{i \in I}$.  Let
  $\vec F \eqdef {(F_i)}_{i \in I}$ be any element of $Z$.  We show
  that there is a unique (non-empty) closed subset $F$ of $X$ such
  that $F_i = \HV {p_i} (F)$ for every $i \in I$.  By
  Lemma~\ref{lemma:inf:egame}, item~1,
  $\nu_i \eqdef \infty.\egame {F_i}$ is a continuous (even tight)
  valuation on $X_i$ for each $i \in I$.  For all $i \sqsubseteq j$ in
  $I$, we use the natural transformation $\infty.\egame\relax$ of
  Lemma~\ref{lemma:inf:egame}, item~3, in order to obtain that
  $\Val {p_{ij}} (\infty.\egame {F_j}) = \infty.\egame {\HV {p_{ij}}
    (F_j)}$, equivalently, that $p_{ij} [\nu_j] = \nu_i$.

  There is a unique continuous valuation $\nu$ on $X$ such that
  $\nu_i = p_i [\nu]$ for every $i \in I$, by
  Proposition~\ref{prop:V:projlim:basic}.  In case~1, we use case~1 of
  that proposition; in case~2, we use case~4 of the proposition,
  recalling that each $\nu_i$ is tight; in case~3, we use case~2 of
  the proposition.

  For every open subset $U$ of $X$, we can write $U$ as
  $\dcup_{i \in I} p_i^{-1} (U_i)$ where $U_i$ is the largest open
  subset of $X_i$ such that $p_i^{-1} (U_i) \subseteq U$.  Then
  $\nu (U) = \dsup_{i \in I} \nu_i (U_i)$, which implies that
  $\nu (U)$ is equal to $0$ or to $\infty$.  By
  Lemma~\ref{lemma:egame:inf}, $\nu$ is equal to $\infty.\egame F$ for
  a unique closed subset $F$ of $X$.

  For every $i \in I$, using the naturality of the transformation
  $\infty.\egame\relax$ (Lemma~\ref{lemma:egame:inf}, item~3),
  $\Val {p_i} (\infty.\egame F) = \infty.\egame {\HV {p_i} (F)}$,
  namely $p_i [\nu] = \infty.\egame {\HV {p_i} (F)}$.  By the
  uniqueness part of Lemma~\ref{lemma:egame:inf},
  $F_i = \HV {p_i} (F)$.  We also note that $F$ cannot be empty,
  otherwise every $F_i$ would be empty as well (in the case of $\HV$,
  not $\HVz$).  This finishes to show that $\vec F = \varphi (F)$.
  \qed
\end{proof}

\begin{remark}
  \label{rem:H:projlim:ep}
  The case of ep-systems can also be obtained by using
  Proposition~\ref{prop:ep}, and relying on the fact that both $\HV Y$
  and $\HVz Y$ are sober, hence $T_0$ spaces and monotone convergence
  spaces, for every space $Y$ \cite[Proposition~1.7]{Schalk:PhD}.
\end{remark}

\begin{remark}
  \label{rem:ep:proper}
  Every projection is a proper map, as we will argue shortly.  It
  follows that, when every $X_i$ is sober, item~2 of
  Theorem~\ref{thm:H:projlim} subsumes item~1.  In order to see that
  every projection $p$ is proper, let $e$ be its associated embedding;
  the image $\dc p [F]$ of any closed set $F$ is equal to
  $e^{-1} (F)$, hence is closed,
  and the inverse image $p^{-1} (Q)$ of a compact saturated set $Q$ is
  equal to $\upc e [Q]$, which is compact saturated.
\end{remark}


The following examples show that one cannot dispense flatly either
with $I$ having a countable cofinal subset, or with the spaces $X_i$
being sober, or with the spaces $X_i$ being locally compact in item~3
of Theorem~\ref{thm:H:projlim}.

\begin{example}
  \label{ex:Waterhouse:H}
  Consider any example of a projective system
  ${(p_{ij} \colon X_j \to X_i)}_{i \sqsubseteq j \in I}$ of non-empty
  sets with an empty projective limit $X$ and with surjective bonding
  maps $p_{ij}$ \cite{Henkin:invlimit,Waterhouse:projlim:empty}.  We
  give each $X_i$ the discrete topology.  Let $\vec F$ be
  ${(X_i)}_{i \in I}$.  Since each $p_{ij}$ is surjective, we have
  $X_i = p_{ij} [X_j] = cl (p_{ij} [X_j])$ for all
  $i \sqsubseteq j \in I$.  However $\vec F$ does not arise as
  $\varphi (F) = {(cl (p_i [F]))}_{i \in I}$ for any closed subset of
  the (empty) projective limit $X$, since no $X_i$ is empty.
\end{example}

\begin{example}
  \label{exa:Stone:H}
  We reuse Stone's example (Example~\ref{exa:Stone}).  For each $n \in
  \nat$, $F_n \eqdef \{n, n+1, \cdots\}$ is the closure of any of its
  infinite subsets in $X_n$, hence $F_m = cl (p_{mn} [F_n])$ for all
  $m \leq n \in \nat$.  But there is no closed subset $F$ of its
  projective limit ($\nat$, with the discrete topology, every $p_n$
  being the identity map) such that $F_n = cl (p_n [F])$ for every $n
  \in \nat$: if such an $F$ existed, it would be included in every
  $F_n$, hence would be empty, and $F_n \neq cl (p_n [F])$ for any $n
  \in \nat$.  The spaces $X_n$ are locally compact, in fact even
  Noetherian (every subset is compact), but not sober.
\end{example}

While Stone's example is ultimately based on encoding a supremum of
topologies on the same set through projective limits
\cite[Remark~3.2]{JGL:proj:class}, the next example is based on the fact
that filtered intersections of subspaces are projective limits, too.
We make this precise as follows.
\begin{remark}
  \label{rem:inter:projlim}
  Let $I, \sqsubseteq$ be a directed preordered set, let $Y$ be a
  topological space and let $X_i$ be a subspace of $Y$, one for each
  $i \in I$, such that that $i \sqsubseteq j$ implies
  $X_i \supseteq X_j$.  Then
  ${(p_{ij} \colon X_j \to X_i)}_{i \sqsubseteq j \in I}$ forms a
  projective system, where each $p_{ij}$ is the inclusion map.  Let
  $X \eqdef \bigcap_{i \in I} X_i$, with the subspace topology.  There
  are inclusion maps $p_i \colon X \to X_i$, and they turn
  $X, {(p_i)}_{i \in I}$ into a projective limit of
  ${(p_{ij} \colon X_j \to X_i)}_{i \sqsubseteq j \in I}$.
\end{remark}

A space is \emph{Baire} if and only if the intersection of countably
many dense open subsets is dense.  It is \emph{completely Baire} if
and only if every closed subspace is Baire.  
The following implications hold:
\begin{flushright}
  $\left.
    \mbox{
      \begin{tabular}{r}
        locally compact sober\\
        Polish $\limp$ quasi-Polish $\limp$ domain-complete
      \end{tabular}
      }
  \right\}
  \limp$ 
  LCS-complete $\limp$ completely Baire
  $\limp$ Baire,
\end{flushright}
see \cite[Figure~1]{dBGLJL:LCScomplete}.

\begin{lemma}
  \label{lemma:not:cBaire}
  In any non-completely Baire space $Y$, we can find an antitonic
  sequence
  $X_0 \supseteq X_1 \supseteq \cdots \supseteq X_n \supseteq \cdots$
  of subspaces such that each $X_n$ is dense in $X_0$, but whose
  intersection $X$ is not dense in $X_0$.  Letting
  $p_{mn} \colon X_n \to X_m$ be the inclusion map,
  ${(p_{mn} \colon X_n \to X_m)}_{m \leq n \in \nat}$ is a projective
  system with the property that the conclusion of
  Theorem~\ref{thm:H:projlim} fails, namely the comparison map
  $\varphi \colon \HV X \to Z$ (resp., from $\HVz X$ to $Z$) is not
  surjective.
\end{lemma}
\begin{proof}
  There is a closed subset $F$ of $Y$, and there are open subsets
  $U_n$, $n \in \nat$, of $Y$ such that $U_n \cap F$ is dense in $F$
  for every $n \in \nat$, but whose intersection is not dense in $F$.
  Let $X_n \eqdef U_0 \cap \cdots \cap U_{n-1} \cap F$ for each
  $n \in \nat$.  When $n=0$, this means that $X_0 \eqdef F$.

  For each $n \in \nat$, $X_n$ is open in $F$, and is also dense,
  because the intersection of any two dense open sets in $F$ is dense
  and open.  (Quick proof: let $U$ and $V$ be dense and open in $F$,
  we show that $U \cap V$ is dense in $F$ as follows.  An equivalent
  definition of a dense subset $A$ of $F$ is that any non-empty open
  subset $W$ of $F$ should intersect $A$.  Now, for every non-empty
  open subset $W$ of $F$, $V \cap W$ is open, and non-empty since $V$
  is dense, and then $U \cap (V \cap W)$ is non-empty because $U$ is
  dense.  Hence $W$ intersects $U \cap V$.)  It is clear that
  $X_0 \supseteq X_1 \supseteq \cdots \supseteq X_n \supseteq \cdots$,
  and $X \eqdef \bigcap_{n \in \nat} X_n$ is equal to
  $\bigcap_{n \in \nat} (U_n \cap F)$, hence is not dense in $F=X_0$.

  By Remark~\ref{rem:inter:projlim},
  ${(p_{mn} \colon X_n \to X_m)}_{m \leq n \in \nat}$ is a projective
  system, where $p_{mn}$ is the inclusion map, and with a projective
  limit $X, {(p_n)}_{n \in \nat}$ obtained by letting
  $X \eqdef \bigcap_{n \in \nat} X_n$ and $p_n$ be inclusion maps.  It
  makes no difference whether we reason with that projective limit or
  with the canonical projective limit.

  For all $m \leq n \in \nat$, the closure of $p_{mn} [X_n] = X_n$ in
  $X_m$ is the whole of $X_m$.  Indeed, $X_n$ is dense in $F$, so
  every non-empty subset $V$ of $F$ intersects $X_n$, and that implies
  that every non-empty subset $U$ of $X_m$, which we can write as
  $V \cap X_m$ for some (necessarily non-empty) open subset $V$ of
  $F$, will intersect $X_n$.  In particular, if the closure of $X_n$
  were not equal to $X_m$, then its complement in $X_m$ would be such
  an open set $U$, but $U$ certainly does not intersect $X_n$.

  Therefore ${(X_n)}_{n \in \nat}$ is an element of the canonical
  projective limit $Z$ of
  ${(\HV {p_{mn}} \colon \HV {X_n} \to \HV {X_m})}_{m \leq n \in
    \nat}$.  However, it is not in the image of the comparison map
  $\varphi \colon \HV X \to Z$.  If it were, then there would be a
  closed subset $C$ of $X$ such that, in particular, the closure of
  $p_0 [C] = C$ in $X_0 = F$ would be equal to $F$.  But the closure
  of $C$ in $F$ is included in the closure of $X$ in $F$, which is
  strictly contained in $F$, since $X$ is not dense in $F$.  Similarly
  with $\HVz$ in place of $\HV$.  \qed
\end{proof}

\begin{example}
  \label{exa:rat:H}
  The space $\rat$ of rational numbers with its usual metric topology
  is not Baire, hence not completely Baire.  It is Hausdorff, hence
  Lemma~\ref{lemma:not:cBaire} provides us with a case where the
  conclusion of Theorem~\ref{thm:H:projlim} fails, although the index
  set is countable and every space $X_i$ is sober.  For an explicit
  construction, we enumerate the elements of $\rat$ as
  ${(q_n)}_{n \in \nat}$, and we define $X_n$ as
  $\rat \diff \{q_0, \cdots, q_{n-1}\}$.  In that case, the projective
  limit $X = \bigcap_{n \in \nat} X_n$ is empty, and the fact that the
  comparison map $\varphi \colon \HV X \to Z$ (resp., from $\HVz X$ to
  $Z$) is not surjective is particularly clear: the element
  ${(X_n)}_{n \in \nat}$ of $Z$ is not in the image of $\varphi$.
\end{example}

\begin{example}
  \label{exa:ratbot:H}
  As a sequel to Example~\ref{exa:rat:H}, compactness (without
  Hausdorffness) does not help.  In other words, we can find countable
  projective systems of compact sober spaces whose limits are not
  preserved by the $\HV$ and $\HVz$ functors.  (Necessarily, those
  spaces are not locally compact.)  We proceed as follows.  Every
  space $X$ has a \emph{lift} $X_\bot$, obtained by adding a fresh
  point $\bot$ to $X$, and whose open subsets are those of $X$ plus
  $X_\bot$.  Then $\bot$ is least in the specialization preordering of
  $X_\bot$, and $X_\bot$ is compact: every open cover
  ${(U_i)}_{i \in I}$ must be such that $\bot \in U_i$ for some
  $i \in I$, and then $U_i$ covers $X_\bot$ by itself.  The closed
  subsets of $X_\bot$ are the empty set and all the sets
  $C \cup \{\bot\}$ with $C$ closed in $X$.  It is easy to see that
  those that are irreducible are the sets $C \cup \{\bot\}$ such that
  $C$ is empty or irreducible in $X$.  It follows that $X$ is sober if
  and only if $X_\bot$ is.  Hence $\rat_\bot$ is a compact sober
  space.  We enumerate the elements of $\rat$ as
  ${(q_n)}_{n \in \nat}$, and we define $X_n$ as
  $\rat_\bot \diff \{q_0, \cdots, q_{n-1}\}$.  Those are simply the
  lifts of the spaces $X_n$ of Example~\ref{exa:rat:H}, and they are
  therefore all compact and sober.  The projective limit
  $X = \bigcap_{n \in \nat} X_n$ is reduced to $\{\bot\}$.  The
  comparison map $\varphi \colon \HV X \to Z$ (resp., from $\HVz X$ to
  $Z$) is not surjective because ${(X_n)}_{n \in \nat}$ is an element
  of the projective limit $Z$ of
  ${(p_{mn} \colon X_n \to X_m)}_{m \leq n \in \nat}$ ($p_{mn}$ being
  the inclusion map), but is not in the image of the comparison map
  $\varphi$.  Indeed, the only elements of $Z$ are $\emptyset$ and
  $\{\bot\}$, and their images by $\varphi$ are the constant
  $\emptyset$ tuple and the constant $\{\bot\}$ tuple.
\end{example}

As for the tightness of Theorem~\ref{thm:H:projlim}, therefore, there
is still a gap in item~3: $I$ needs to have a countable cofinal subset
and each $X_i$ needs to be sober, as well as something else, and that
something else probably lies between locally compact sober and
completely Baire.  
It is open whether requiring each $X_i$ to be LCS-complete, as in
Theorem~\ref{thm:V:projlim}, item~3, for example, would be enough.

\section{$\A$-valuations and quasi-lenses}
\label{sec:cont-a-valu}

Another standard powerdomain considered in domain theory is the
\emph{Plotkin powerdomain} \cite[Definition~IV-8.11]{GHKLMS:contlatt}.
On continuous coherent dcpos, as well as on $\omega$-continuous dcpo,
this can be realized as the space $\Plotkin X$ of all lenses, with the
Scott topology of an ordering called the topological Egli-Milner
ordering (see Theorem~IV.8.18 in \cite{GHKLMS:contlatt}, or
\cite[Theorem~6.2.19, Theorem~6.2.22]{AJ:domains}).  A \emph{lens} is
a non-empty set of the form $Q \cap C$ where $Q$ is compact saturated
and $C$ is closed in $X$.  The \emph{Vietoris topology} has subbasic
open subsets of the form $\Box U$ (the set of lenses included in $U$)
and $\Diamond U$ (the set of lenses that intersect $U$), for each open
subset $U$ of $X$.  We let $\Plotkinn X$ denote $\Plotkin X$ with the
Vietoris topology.  The specialization ordering of $\Plotkinn X$ is
the \emph{topological Egli-Milner ordering}: $L \TEMleq L'$ if and
only if $\upc L \supseteq \upc L'$ and $cl (L) \subseteq cl (L')$
\cite[Discussion before Fact~4.1]{GL:duality}.  This is an ordering,
not just a preordering, hence $\Plotkinn X$ is $T_0$.

Every lens is compact, and if $X$ is Hausdorff, every non-empty
compact subset of $X$ is a lens.  Hence, when $X$ is Hausdorff,
$\Plotkinn X$ is the familiar space of all non-empty compact subsets
of $X$ with the usual Vietoris topology.

It is profitable not to study $\Plotkinn X$ directly, and to examine
better-behaved variants.  Heckmann observed that one can define a
related notion, with better overall properties, and which look like
continuous valuations: $\A$-valuations \cite{Heckmann:absval}.  Let
$\A \eqdef \{0, \mathbf M, 1\}$, ordered by $0 < \mathbf M < 1$.  An
\emph{$\A$-valuation} on a space $X$ is a Scott-continuous map
$\alpha \colon \Open X \to \A$ such that $\alpha (\emptyset)=0$,
$\alpha (X)=1$, and, for all open subsets $U$ and $V$ of $X$,
\begin{enumerate}
\item if $\alpha (U)=0$ then $\alpha (U \cup V) = \alpha (V)$;
\item if $\alpha (V)=1$ then $\alpha (U \cap V) = \alpha (U)$.
\end{enumerate}
We write $\Aval X$ for the set of all $\A$-valuations on $X$.  The
\emph{Vietoris topology} on $\Aval X$ is generated by the subbasic
open sets $\Box^\A U \eqdef \{\alpha \in \Aval X \mid \alpha (U)=1\}$
and
$\Diamond^\A U \eqdef \{\alpha \in \Aval X \mid \alpha (U) \neq 0\}$,
where $U \in \Open X$.  We write $\AvalV X$ for $\Aval X$ with the
Vietoris topology.  Its specialization ordering is the pointwise
ordering: $\alpha \leq \beta$ if and only if
$\alpha (U) \leq \beta (U)$ for every $U \in \Open X$.

There is an $\A$-valuation functor $\AvalV$ on $\Topcat$.  For every
continuous map $f \colon X \to Y$,
$\AvalV f \colon \AvalV X \to \AvalV Y$ maps every $\A$-valuation
$\alpha$ on $X$ to $f [\alpha]$, defined so that for every
$V \in \Open Y$, $f [\alpha] (V) = \alpha (f^{-1} (V))$, exactly as
with continuous valuations.  This functor is part of a monad, just
like the other functors we will mention below, but we will ignore this
here.  We note that for every open subset $V$ of $Y$,
${(\AvalV f)}^{-1} (\Box^\A V) = \Box^\A {f^{-1} (V)}$ and
${(\AvalV f)}^{-1} (\Diamond^\A V) = \Diamond^\A {f^{-1} (V)}$.

\begin{proposition}
  \label{prop:Aval:projlim}
  The comparison map $\varphi \colon \AvalV X \to Z$ of any projective
  $\AvalV$-situation is a topological embedding.
\end{proposition}
\begin{proof}
  We use Lemma~\ref{lemma:T:projlim} and to this end we verify that
  $\AvalV$ is $R$-nice with $R \eqdef \{\mathsf M, 1\}$,
  $B_X (1, U) \eqdef \Box^\A U$,
  $B_X (\mathsf M, U) \eqdef \Diamond^\A U$.  As in
  Proposition~\ref{prop:V:projlim}, property~1 of
  Definition~\ref{defn:T:nice:S} stems from the Scott-continuity of
  $\A$-valuations, while property~2 is clear from the description we
  gave of ${(\AvalV f)}^{-1}$ right before this proposition.  Finally,
  $\AvalV X$ is $T_0$, by definition of its specialization
  preordering.  \qed
\end{proof}

An intermediate notion is that of \emph{quasi-lens}, which originates
from \cite[Theorem~9.6]{heckmann:2ndorder}.  A \emph{quasi-lens} on a
topological space $X$ is a pair $(Q, C)$ of a compact saturated subset
$Q$ and a closed subset $C$ of $X$ such that:
\begin{enumerate}
\item $Q$ intersects $C$;
\item $Q \subseteq \upc (Q \cap C)$;
\item for every open neighborhood $U$ of $Q$, $C \subseteq cl (U
  \cap C)$.
\end{enumerate}
We write $\QL X$ for the space of quasi-lenses on $X$.  The
\emph{Vietoris topology} on $\QL X$ is generated by the subbasic open
sets $\Box^\quasi U \eqdef \{(Q, C) \in \QL X \mid Q \subseteq U\}$
and
$\Diamond^\quasi U \eqdef \{(Q, C) \in \QL X \mid C \cap U \neq
\emptyset\}$.  We write $\QLV X$ for $\QL X$ with the Vietoris
topology.
\begin{lemma}
  \label{lemma:fQL:cont}
  For every topological space $X$,
  \begin{enumerate}
  \item the inclusion of $\QLV X$ into $\SV X \times \HV X$ is a
    topological embedding;
  \item the specialization preordering on $\QLV X$ is $\supseteq
    \times \subseteq$, which is antisymmetric, so $\QLV X$ is $T_0$.
  \end{enumerate}
\end{lemma}
\begin{proof}
  1. Let $i$ be the inclusion map.  For every open subset $U$ of $X$,
  $i^{-1} (\Box U \times \HV X) = \Box^\quasi U$ and
  $i^{-1} (\SV X \times \Diamond U) = \Diamond^\quasi U$, so $i$ is
  full and continuous.  It is clearly injective, hence a topological
  embedding.

  2. The specialization preordering of a subspace $Y$ of a space $Z$
  is the restriction of the specialization preordering of $Z$ to $Y$.
  \qed
\end{proof}
\begin{lemma}
  \label{lemma:QLV:functor}
  There is a $\QLV$ functor on $\Topcat$, and its action on continuous
  maps $f \colon X \to Y$ is defined by
  $\QLV (f) (Q, C) \eqdef (\SV (f) (Q), \HV (f) (C)) = (\upc f [Q], cl
  (f [C]))$.  For every open subset $V$ of $Y$,
  ${(\QLV f)}^{-1} (\Box^\quasi V) = \Box^\quasi {f^{-1} (V)}$ and
  ${(\QLV f)}^{-1} (\Diamond^\quasi V) = \Diamond^\quasi {f^{-1}
    (V)}$.
\end{lemma}
\begin{proof}
  We need to show that $(\upc f [Q], cl (f [C]))$ is a quasi-lens on
  $Y$, for every quasi-lens $(Q, C)$ on $X$.  Since $Q \cap C$ is
  non-empty, we can pick a point $x$ from it, and then $f (x)$ is in
  both $\upc f [Q]$ and $cl (f [C])$.  Since
  $Q \subseteq \upc (Q \cap C)$, every point $y \eqdef f (x)$ in
  $f [Q]$ (with $x \in Q$), is such that $x' \leq x$ for some
  $x' \in Q \cap C$.  Then $f (x') \leq y$, since $f$ is continuous,
  hence monotonic, and
  $f (x') \in f [Q] \cap f [C] \subseteq \upc f [Q] \cap cl (f [C])$.
  Hence $f [Q] \subseteq \upc f [Q] \cap cl (f [C])$, from which we
  obtain $\upc f [Q] \subseteq \upc (\upc f [Q] \cap cl (f [C]))$.
  Finally, let $V$ be any open neighborhood of $\upc f [Q]$.  Then
  $Q \subseteq f^{-1} (V)$, so $C \subseteq cl (f^{-1} (V) \cap C)$.
  We need to show that $cl (f [C]) \subseteq cl (V \cap cl (f [C]))$,
  and for that it is enough to show that every open set $W$ that
  intersects the left hand-side intersects the right-hand side.  If
  $W$ intersects $cl (f [C])$, it intersects $f [C]$, so $f^{-1} (W)$
  intersects $C$.  Since $C \subseteq cl (f^{-1} (V) \cap C)$,
  $f^{-1} (W)$ also intersects $cl (f^{-1} (V) \cap C)$, hence
  $f^{-1} (V) \cap C$.  It follows that $f^{-1} (W \cap V)$ intersects
  $C$, or alternatively that $W \cap V$ intersects $f [C]$, hence also
  $cl (f [C])$.  Therefore $W$ intersects $V \cap cl (f [C])$, hence
  also $cl (V \cap cl (f [C]))$.

  The fact that $\QLV (f)$ is continuous follows from the fact that
  $\SV (f)$ and $\HV (f)$ are continuous, and from
  Lemma~\ref{lemma:fQL:cont}, item~1.  That can also be deduced from
  the final claims,
  ${(\QLV f)}^{-1} (\Box^\quasi V) = \Box^\quasi {f^{-1} (V)}$ and
  ${(\QLV f)}^{-1} (\Diamond^\quasi V) = \Diamond^\quasi {f^{-1} (V)}$
  for every $V \in \Open Y$, which are easily proved.  \qed
\end{proof}

Just like with $\AvalV$, we have the following.
\begin{proposition}
  \label{prop:QL:projlim}
  The comparison map $\varphi \colon \QLV X \to Z$ of any projective
  $\QLV$-situation is a topological embedding.
\end{proposition}
\begin{proof}
  We use Lemma~\ref{lemma:T:projlim}, showing that $\QLV$ is $R$-nice
  with $R \eqdef \{\mathsf M, 1\}$, $B_X (1, U) \eqdef \Box^\quasi U$,
  $B_X (\mathsf M, U) \eqdef \Diamond^\quasi U$.  Property~1 of
  Definition~\ref{defn:T:nice:S} stems from the fact that
  $\Box^\quasi$ and $\Diamond^\quasi$ are Scott-continuous.  This is
  clear for $\Diamond^\quasi$, which commutes with arbitrary unions.
  For $\Box^\quasi$, let ${(U_i)}_{i \in I}$ be any directed family of
  open subsets of a space $X$: for every lens $(Q, C)$,
  $(Q, C) \in \Box^\quasi {\dcup_{i \in I} U_i}$ if and only if
  $Q \subseteq \dcup_{i \in I} U_i$, if and only if $Q \subseteq U_i$
  for some $i \in I$ (because $Q$ is compact), if and only if
  $(Q, C) \in \dcup_{i \in I} \Box^\quasi {U_i}$.  Property~2 stems
  from the characterization of ${(\QLV f)}^{-1}$ given in the second
  part of Lemma~\ref{lemma:QLV:functor}.  \qed
\end{proof}

We will see the precise relationship between $\QLV X$ and
$\Plotkinn X$ in Section~\ref{sec:lenses}; for now, they simply carry
a resemblance.  The relationship between $\AvalV X$ and $\QLV X$ is
that they are homeomorphic when $X$ is sober
\cite[Fact~5.2]{GL:duality}.  Explicitly, for any space $X$, there is
a function $\quasi_X$ that maps every quasi-lens $(Q, C) \in \QLV X$
to the $\A$-valuation $\alpha$ defined by $\alpha (U) \eqdef 1$ if
$Q \subseteq U$, $0$ if $U \cap C = \emptyset$, $\mathbf M$ otherwise.
It is easy to see that for every open subset $U$ of $X$,
$\quasi_X^{-1} (\Box^\A U) = \Box^\quasi U$ and
$\quasi_X^{-1} (\Diamond^\A U) = \Diamond^\quasi U$, so that
$\quasi_X$ is full and continuous, and since $\QLV X$ is $T_0$,
$\quasi_X$ is a topological embedding.  When $X$ is sober, $\quasi_X$
has an inverse, which maps every $\A$-valuation $\alpha$ to $(Q, C)$
defined by letting $Q$ be the intersection of the open subsets $U$ of
$X$ such that $\alpha (U)=1$ and $F$ be the complement of the largest
open subset $U$ of $X$ such that $\alpha (U)=0$.
\begin{lemma}
  \label{lemma:quasi:nat}
  The collection of maps $\quasi_X$ is natural in $X$.
\end{lemma}
\begin{proof}
  Let $f \colon X \to Y$ be a continuous map.  We need to show that
  for every quasi-lens $(Q, C)$ on $X$, letting
  $\alpha \eqdef \quasi_X (Q, C)$, we have
  $f [\alpha] = \quasi_Y (\upc f [Q], cl (f [C]))$.  For every open
  subset $V$ of $Y$, $\quasi_Y (\upc f [Q], cl (f [C])) (V)$ is equal
  to $1$ (resp., $0$) if and only if $\upc f [Q] \subseteq V$ (resp.,
  $cl (f [C]) \cap V = \emptyset$) if and only if
  $Q \subseteq f^{-1} (V)$ (resp., $f [C] \cap V = \emptyset$, namely
  $C \cap f^{-1} (V) = \emptyset$) if and only if
  $\alpha (f^{-1} (V))=1$ (resp., $\alpha (f^{-1} (V))=0$) if and only
  if $f [\alpha] (V)=1$ (resp, $0$).  \qed
\end{proof}

In order to proceed with projective limits, we recall the notion of
uniform tightness from \cite[Lemma~6.4, Remark~6.6]{JGL:kolmogorov}.
Given a projective system
${(p_{ij} \colon X_j \to X_i)}_{i \sqsubseteq j \in I}$, with
projective limit $X, {(p_i)}_{i \in I}$, a family
${(\nu_i)}_{i \in I}$ of maps $\nu_i \colon \Open {X_i} \to \creal$ is
\emph{uniformly tight} if and only if for every $i \in I$, for every
$U \in \Open {X_i}$, for every $r \in \Rp$ such that $r \ll \nu_i (U)$
(i.e., $r=0$ or $r < \nu_i (U)$), there is a compact saturated subset
$Q$ of $X$ such that $\upc p_i [Q] \subseteq U$ and for every
$j \in I$ with $i \sqsubseteq j$,
$r \leq \nu_j^\bullet (\upc p_j [Q])$.  The notation $\nu_j^\bullet$
stands for the function that maps every compact saturated subset $Q_j$
of $X_j$ to $\inf_V \nu_j (V)$, where $V$ ranges over the open
neighborhoods $V$ of $Q_j$.

By equating $\A$ with the subset $\{0, 1/2, 1\}$ of $\creal$, and
noting that this identification preserves order, suprema, and infima,
this yields a notion of uniform tightness for $\A$-valuations.
Explicitly, and making some simplifications along the way, a family
${(\alpha_i)}_{i \in I}$ of maps from $\Open {X_i}$ to $\A$ is
uniformly tight if and only if for every $i \in I$, for every open
subset $U$ of $X_i$,
\begin{enumerate}[label=(\alph*)]
\item if $\alpha_i (U)=1$ then there is a compact saturated
  subset $Q$ of $X$ such that $\upc p_i [Q] \subseteq U$ and for every
  $j \in I$ with $i \sqsubseteq j$, for every open neighborhood $V$ of
  $\upc p_j [Q]$, $\alpha_j (V)=1$, and
\item if $\alpha_i (U) \neq 0$ then there is a compact saturated
  subset $Q$ of $X$ such that $\upc p_i [Q] \subseteq U$ and for every
  $j \in I$, for every open neighborhood $V$ of $\upc p_j [Q]$,
  $\alpha_j (V) \neq 0$.
\end{enumerate}

\begin{lemma}
  \label{lemma:QL:utight}
  Let ${(p_{ij} \colon X_j \to X_i)}_{i \sqsubseteq j \in I}$ be a
  projective system of topological spaces, with canonical projective
  limit $X, {(p_i)}_{i \in I}$.  Let $(Q_i, C_i)$ be quasi-lenses on
  $X_i$ for each $i \in I$, such that
  $(Q_i, C_i) = \QLV {p_{ij}} (Q_j, C_j)$ for all $i \sqsubseteq j$ in
  $I$.  Let $\alpha_i \eqdef \quasi_{X_i} (Q_i, C_i)$.  If
  $\SV X, {(\SV {p_i})}_{i \in I}$ is a projective limit of
  ${(\SV {p_{ij}} \colon \SV {X_j} \to \SV {X_i})}_{i \sqsubseteq j
    \in I}$ and if $\HV X, {(\HV {p_i})}_{i \in I}$ is a projective
  limit of
  ${(\HV {p_{ij}} \colon \HV {X_j} \to \HV {X_i})}_{i \sqsubseteq j
    \in I}$, then ${(\alpha_i)}_{i \in I}$ is uniformly tight.
\end{lemma}
\begin{proof}
  In order to prove $(a)$, we note that since for all
  $i \sqsubseteq j$ in $I$, $(Q_i, C_i) = \QLV {p_{ij}} (Q_j, C_j)$,
  we have $Q_i = \SV {p_{ij}} (Q_j)$, by
  Lemma~\ref{lemma:QLV:functor}.  Hence ${(Q_i)}_{i \in I}$ is an
  element of the canonical projective limit of
  ${(\SV {p_{ij}} \colon \SV {X_j} \to \SV {X_i})}_{i \sqsubseteq j
    \in I}$.  Since $\SV X, {(\SV {p_i})}_{i \in I}$ is another
  projective limit, by the universal property of projective limits,
  there must be a (unique) element $Q$ of $\SV X$ such that
  $Q_i = \SV {p_i} (Q)$ for every $i \in I$.  Explicitly,
  $Q_i = \upc p_i [Q]$ for every $i \in I$.  Now, let $i \in I$, let
  $U$ be open in $X_i$, and let us assume that $\alpha_i (U)=1$.  By
  definition of $\alpha_i$ as $\quasi_{X_i} (Q_i, C_i)$,
  $Q_i \subseteq U$, so $\upc p_i [Q] \subseteq U$.  For every
  $j \in I$ with $i \sqsubseteq j$ and for every open neighborhood $V$
  of $\upc p_j [Q]$, we have
  $Q_j = \SV {p_j} (Q) = \upc p_j [Q] \subseteq V$, so
  $\alpha_j (V)=1$.

  We turn to $(b)$.  Since $\HV X, {(\HV {p_i})}_{i \in I}$ is a
  projective limit of
  $(\HV {p_{ij}} \colon \HV {X_j} \to \HV {X_i})_{i \sqsubseteq j \in
    I}$, we reason as above and we obtain that there is a (unique)
  element $C$ of $\HV X$ such that $C_i = \HV {p_i} (C)$ for every
  $i \in I$, namely $C_i = cl (p_i [C])$.  Let $i \in I$ and $U$ be an
  open subset of $X_i$ such that $\alpha_i (U) \neq 0$.  Since
  $\alpha_i = \quasi_{X_i} (Q_i, C_i)$, this means that $C_i$
  intersects $U$, equivalently that $cl (p_i [C])$ intersects $U$.
  Hence $p_i [C]$ intersects $U$, showing that there is an element
  $\vec x \in C$ such that $x_i = p_i (\vec x) \in U$.  We let
  $Q \eqdef \upc \vec x$.  Then $\upc p_i [Q] = \upc x_i$ is included
  in $U$.  For every $j \in I$ such that $i \sqsubseteq j$, for every
  open neighborhood $V$ of $\upc p_j [Q]$, $V$ contains
  $x_j = p_j (\vec x)$.  But $x_j$ is in $p_j [C]$, hence in $C_j$, so
  $C_j$ intersects $V$, and therefore $\alpha_j (V) \neq 0$.  \qed
\end{proof}

The point of uniform tightness stems from Lemma~6.5 of
\cite{JGL:kolmogorov}.  For every $\mu \colon \Smyth_0 X \to \creal$,
there is map $\mu^\circ \colon \Open X \to \creal$ defined by
$\mu^\circ (U) \eqdef \dsup_Q \mu (Q)$, where $Q$ ranges over the
compact saturated subsets of $X$ included in $U$.  (Beware that
$\Smyth_0 X$ was written as $\Smyth X$ in \cite{JGL:kolmogorov}.)
Given any projective system
${(p_{ij} \colon X_j \to X_i)}_{i \sqsubseteq j \in I}$ of topological
spaces, given Scott-continuous maps $\nu_i \colon \Open X \to \creal$
for each $i \in I$, such that $\nu_i = p_{ij} [\nu_j]$ for all
$i \sqsubseteq j \in I$, one can define a map
$\mu \colon \Smyth_0 X \to \creal$ by
$\mu (Q) \eqdef \finf_{i \in I} \nu_i^\bullet (\upc p_i [Q])$; the
arrow superscript denotes the fact that the infimum is filtered, in
fact $i \sqsubseteq j \in I$ implies
$\nu_i^\bullet (\upc p_i [Q]) \geq \nu_j^\bullet (\upc p_j [Q])$.
Then Lemma~6.5 of \cite{JGL:kolmogorov} states the equivalence between
three conditions, among which the following two: (1)
${(\nu_i)}_{i \in I}$ is uniformly tight, (3) for every $i \in I$,
$\nu_i = p_i [\mu^\circ]$.  This applies verbatim to $\A$-valuations
$\alpha_i$ for $\nu_i$, modulo our identification of $\A$ with
$\{0, 1/2, 1\} \subseteq \creal$.
\begin{proposition}
  \label{prop:prohorov:Aval}
  Let ${(p_{ij} \colon X_j \to X_i)}_{i \sqsubseteq j \in I}$ be a
  projective system of topological spaces, with canonical projective
  limit $X, {(p_i)}_{i \in I}$.  Let $\alpha_i$ be $\A$-valuations on
  each $X_i$ such that $\alpha_i = p_{ij} [\alpha_j]$ 
  for all $i \sqsubseteq j \in I$.  If ${(\alpha_i)}_{i \in I}$ is
  uniformly tight, then there is a 
  Scott-continuous map $\alpha \colon \Open X \to \creal$ such that
  for every $i \in I$, $\alpha_i = p_i [\alpha]$, and $\alpha$ is an
  $\A$-valuation.
\end{proposition}
\begin{proof}
%
  We define $\mu \colon \Smyth_0 X \to \A$ by
  $\mu (Q) \eqdef \finf_{i \in I} \alpha_i^\bullet (\upc p_i [Q])$ for
  every $Q \in \Smyth_0 X$.  Since ${(\alpha_i)}_{i \in I}$ is
  uniformly tight by assumption, Lemma~6.5 of \cite{JGL:kolmogorov},
  as discussed above, entails that $\alpha_i = p_i [\mu^\circ]$ for
  every $i \in I$.

  We claim that $\mu^\circ$ is Scott-continuous.  We recall that
  $\mu^\circ (U) \eqdef \dsup_Q \mu (Q)$, where $Q$ ranges over the
  compact saturated subsets of $X$ included in $U$, for every
  $U \in \Open X$.  It is easy to see that $\mu^\circ$ is monotonic.
  Let ${(U_j)}_{j \in J}$ be a directed family of open subsets of $X$,
  with union equal to $U$.  Since $\mu^\circ$ is monotonic,
  $\dsup_{j \in J} \mu^\circ (U_j) \leq \mu^\circ (U)$.  In order to
  show the reverse inequality, it suffices to show that for every
  $r < \mu^\circ (U)$, there is an index $j \in J$ such that
  $r \leq \mu^\circ (U_j)$.  Since $r < \mu^\circ (U)$, by definition
  of $\mu^\circ$, there is a compact saturated subset $Q$ of $X$
  included in $U$ such that $r \leq \mu (Q)$.  Since $Q$ is compact
  and ${(U_j)}_{j \in J}$ is directed, $Q \subseteq U_j$ for some
  $j \in J$, and therefore $r \leq \mu^\circ (U_j)$.

  The map $\mu^\circ$ is strict: the only compact saturated set
  included in $\emptyset$ is the empty set, and
  $\mu (\emptyset) = \finf_{i \in I} \nu_i^\bullet (\upc p_i
  [\emptyset]) = \finf_{i \in I} \nu_i^\bullet (\emptyset)$; this is
  equal to $0$, because for every $i \in I$,
  $\nu_i^\bullet (\emptyset)$ is the infimum of the values
  $\nu_i (V)$, where $V$ ranges over the open neighborhoods of
  $\emptyset$, namely $\nu_i (\emptyset) = 0$.

  We show that $\mu^\circ$ satisfies the remaining defining conditions
  for an $\A$-valuation.  For short, we will write $U_i$ for the
  largest open subset of $X_i$ such that $p_i^{-1} (U_i) \subseteq U$,
  for every open subset $U$ of $X$; and similarly $V_i$ for $V$, for
  example.  Since $U = \dcup_{i \in I} p_i^{-1} (U_i)$, and since
  $\mu^\circ$ is Scott-continuous, we have
  $\mu^\circ (U) = \dsup_{i \in I} \mu^\circ (p_i^{-1} (U_i)) =
  \dsup_{i \in I} \alpha_i (U_i)$.

  When $U=X$, the sets $U_i$ are equal to the given spaces $X_i$, so
  $\mu^\circ (X) = \dsup_{i \in I} \alpha_i (X_i) = 1$.

  Let $U$ and $V$ be arbitrary open subsets of $X$.  If
  $\mu^\circ (U)=0$, then $\alpha_i (U_i)=0$ for every $i \in I$.  We
  have $U \cup V = \dcup_{i \in I} p_i^{-1} (U_i \cup V_i)$, so
  $\mu^\circ (U \cup V) = \dsup_{i \in I} \mu^\circ (p_i^{-1} (U_i
  \cup V_i))$ (since $\mu^\circ$ is Scott-continuous)
  $= \dsup_{i \in I} \alpha_i (U_i \cup V_i)$.  That is equal to
  $\dsup_{i \in I} \alpha_i (V_i)$ by Condition~1 of the definition of
  $\A$-valuations; so $\mu^\circ (U \cup V) = \mu^\circ (V)$, provided
  that $\mu^\circ (U)=0$.

  If $\mu^\circ (V)=1$, then $\dsup_{i \in I} \alpha_i (V_i)=1$, and
  since $\alpha_i (V_i)$ can only take the values $0$, $\mathbf M$
  ($=1/2$) and $1$, we must have $\alpha_{i_0} (V_{i_0})=1$ for some
  $i_0 \in I$.  Then $\alpha_i (V_i)=1$ for every $i \sqsupseteq i_0$,
  since then $p_{i_0}^{-1} (V_{i_0}) \subseteq p_i^{-1} (V_i)$, and
  then
  $1 = \alpha_{i_0} (V_{i_0}) = \mu^\circ (p_{i_0}^{-1} (V_{i_0}))
  \leq \mu^\circ (p_i^{-1} (V_i)) = \alpha_i (V_i)$.  We observe that
  $U \cap V = \dcup_{i \in I} p_i^{-1} (U_i \cap V_i)$: every point of
  $U \cap V$ is in $p_i^{-1} (U_i)$ for some $i \in I$, in
  $p_j^{-1} (V_j)$ for some $j \in I$, hence in
  $p_k^{-1} (U_k) \cap p_k^{-1} (V_k) = p_k^{-1} (U_k \cap V_k)$ for
  some $k \in I$ such that $i, j \sqsubseteq k$; the reverse inclusion
  is obvious.  Since $\mu^\circ$ is Scott-continuous, it follows that
  $\mu^\circ (U \cap V) = \dsup_{i \in I} \alpha_i (U_i \cap V_i)$.
  For every $i \sqsupseteq i_0$,
  $\alpha_i (U_i \cap V_i) = \alpha_i (U_i)$ by Condition~2 of the
  definition of $\A$-valuations; so
  $\mu^\circ (U \cap V) \geq \dsup_{i \sqsupseteq i_0} \alpha_i (U_i)
  = \dsup_{i \in I} \alpha_i (U_i)$ (since the family of indices
  $i \sqsupseteq i_0$ is cofinal in $I$) $= \mu^\circ (U)$.  The
  reverse inequality follows by monotonicity of $\mu^\circ$.  \qed
\end{proof}

Putting everything together, we obtain the following.
\begin{theorem}
  \label{thm:QL:projlim}
  Let ${(p_{ij} \colon X_j \to X_i)}_{i \sqsubseteq j \in I}$ be a
  projective system of topological spaces, with canonical projective
  limit $X, {(p_i)}_{i \in I}$.  If every $X_i$ is sober,
  and if $\HV X$ is a projective limit of
  ${(\HV {p_{ij}} \colon \HV {X_j} \to \HV {X_i})}_{i \sqsubseteq j
    \in I}$, then $\QLV X$ is a projective limit of
  ${(\QLV {p_{ij}} \colon \QLV {X_j} \to \QLV {X_i})}_{i \sqsubseteq j
    \in I}$ and $\AvalV X$ is a projective limit of
  ${(\AvalV {p_{ij}} \colon \AvalV {X_j} \to \AvalV {X_i})}_{i
    \sqsubseteq j \in I}$.
\end{theorem}
\begin{proof}
  Let $Z^\sharp$ be the canonical projective limit of
  $(\SV {p_{ij}} \colon \SV {X_j} \to \SV {X_i})_{i \sqsubseteq j \in
    I}$ and $\varphi^\sharp \colon \SV X \to Z^\sharp$ be the
  comparison map; $\varphi^\sharp$ is a homeomorphism by
  Theorem~\ref{thm:Q:projlim}.  Let $Z^\flat$ be the canonical
  projective limit of
  $(\HV {p_{ij}} \colon \HV {X_j} \to \HV {X_i})_{i \sqsubseteq j \in
    I}$ and $\varphi^\flat \colon \HV X \to Z^\flat$ be the comparison
  map; $\varphi^\flat$ is a homeomorphism by assumption.  Finally, let
  $Z^\natural$ be the canonical projective limit of
  $(\QLV {p_{ij}} \colon \QLV {X_j} \to \QLV {X_i})_{i \sqsubseteq j
    \in I}$ and $\varphi^\natural \colon \QLV X \to Z^\natural$ be the
  comparison map. 
  Considering Proposition~\ref{prop:QL:projlim}, 
  $\varphi^\natural$ is a topological embedding, and we need to show
  that it is surjective.

  Let $(Q_i, C_i) \in \QLV {X_i}$ be given for each $i \in I$ so that
  for all $i \sqsubseteq j \in I$,
  $(Q_i, C_i) = \QLV {p_{ij}} (Q_j, C_j)$.  Let
  $\alpha_i \eqdef \quasi_{X_i} (Q_i, C_i)$.  Both $\varphi^\sharp$
  and $\varphi^\flat$ are homeomorphisms, so we can apply
  Lemma~\ref{lemma:QL:utight} and conclude that
  ${(\alpha_i)}_{i \in I}$ is uniformly tight.  By naturality of
  $\quasi_\_$ (Lemma~\ref{lemma:quasi:nat}), we have
  $\alpha_i = p_{ij} [\alpha_j]$ for all $i \sqsubseteq j \in I$, so
  Proposition~\ref{prop:prohorov:Aval} applies, giving us an
  $\A$-valuation $\alpha$ on $X$ such that $\alpha_i = p_i [\alpha]$
  for every $i \in I$.  Any limit of sober spaces, taken in $\Topcat$,
  is sober \cite[Theorem 8.4.13]{JGL-topology}.  Therefore $X$ is
  sober, and because of that, $\quasi_X$ is a homeomorphism.  In
  particular $\alpha = \quasi_X (Q, C)$ for some unique quasi-lens
  $(Q, C)$ on $X$.  For every $i \in I$, the fact that
  $\alpha_i = p_i [\alpha]$ entails that
  $\quasi_{X_i} (Q_i, C_i) = \AvalV {p_i} (\quasi_X (Q, C))$, which is
  equal to $\quasi_{X_i} (\QLV {p_i} (Q, C))$, by naturality of
  $\quasi_\_$ (Lemma~\ref{lemma:quasi:nat}).  Since each $X_i$ is
  sober, $\quasi_{X_i}$ is a homeomorphism, so
  $(Q_i, C_i) = \QLV {p_i} (Q, C)$.

  The case of $\AvalV$ is an immediate consequence, since $\quasi_X$
  and the maps $\quasi_{X_i}$ are homeomorphisms, and using the
  naturality of $\quasi_\_$ once again.  \qed
\end{proof}

\begin{example}
  \label{exa:QL:projlim:sober:nec}
  Sobriety is required in Theorem~\ref{thm:QL:projlim}.  Let us look
  back at Stone's counterexample~\ref{exa:Stone}.  Each $X_n$ is
  compact, so $(X_n, X_n)$ is a quasi-lens.  For all
  $m \leq n \in \nat$, $p_{mn} \colon X_n \to X_m$ is the identity
  map, so $(X_m, X_m) = \QLV {p_{mn}} (X_n, X_n)$.  It follows that
  ${(X_n, X_n)}_{n \in \nat}$ is an element of the projective limit
  $Z$ of
  ${(\QLV {p_{mn}} \colon \QLV {X_n} \to \QLV {X_m})}_{m \leq n \in
    \nat}$.  But it is not in the image of the comparison map
  $\varphi \colon \QLV X \to Z$.  Indeed, $X$ is $\nat$ with the
  discrete topology, so the only quasi-lenses on $X$ are the pairs
  $(A, A)$ where $A$ is a non-empty finite subset of $\nat$,
  and their images by $\varphi$ are the constant $\nat$-indexed tuples
  whose entries are all equal to $(A, A)$.
\end{example}

\begin{remark}
  \label{rem:QL:projlim:nec}
  The requirement that $\HV X$ be a projective limit of
  $(\HV {p_{ij}} \colon \allowbreak \HV {X_j} \to \HV {X_i})_{i
    \sqsubseteq j \in I}$ in Theorem~\ref{thm:QL:projlim} cannot be
  dispensed with if every $X_i$ is not only sober but also pointed and
  if the bonding maps $p_{ij}$ are strict.  A \emph{pointed} space is
  a space with a least element $\bot$ in its specialization ordering,
  or equivalently with an element whose sole open neighborhood is the
  whole space; every pointed space is compact.  A \emph{strict} map is
  a function that maps least elements to least elements.  In that
  case, the projective limit is also pointed.  In a compact space, for
  every non-empty closed subset $C$, $(\upc C, C)$ is a quasi-lens; we
  even have $C \subseteq cl (\upc C \cap C)$, from which it is
  immediate to see that every open neighborhood $U$ of $\upc C$
  satisfies $C \subseteq cl (U \cap C)$.  In a pointed space $Y$, $C$
  contains $\bot$, so $\upc C$ is simply the whole space $Y$.  Now let
  us use the notations of the proof of Theorem~\ref{thm:QL:projlim},
  and let us assume that each $X_i$ is sober and pointed, and that
  each $p_{ij}$ is strict.  Let ${(C_i)}_{i \in I}$ be any element of
  $Z^\flat$.  Then each pair $(X_i, C_i)$ is a quasi-lens, as we have
  just seen.  For all $i \sqsubseteq j \in I$,
  $\upc X_i = \upc {p_{ij}} [\upc X_j]$ because every element of
  $\upc X_i=X_i$ is larger than or equal to its bottom element, which
  is obtained as the image of the bottom element of $X_j$ by the
  strict function $p_{ij}$.  It follows that ${(X_i, C_i)}_{i \in I}$
  is in $Z^\natural$.  If $\varphi^\natural$ is surjective, then there
  is a quasi-lens $(Q, C)$ on $X$ such that
  $(X_i, C_i) = \QLV {p_i} (Q, C)$ for every $i \in I$, in particular
  such that $C_i = \HV {p_i} (C)$ for every $i \in I$.  Therefore
  $\varphi^\flat$ must be surjective, too, hence a homeomorphism, by
  Proposition~\ref{prop:H:projlim}.
\end{remark}

We combine Theorem~\ref{thm:QL:projlim} with
Theorem~\ref{thm:H:projlim}, and we obtain the following.  We remember
from Remark~\ref{rem:ep:proper} that every projection is proper.
\begin{corollary}
  \label{corl:QL:projlim}
  Let ${(p_{ij} \colon X_j \to X_i)}_{i \sqsubseteq j \in I}$ be a
  projective system of topological spaces, with canonical projective
  limit $X, {(p_i)}_{i \in I}$.  If every $X_i$ is sober, and:
  \begin{enumerate}
  \item every $p_{ij}$ is a proper map (e.g., a projection),
  \item or $I$ has a countable cofinal subset and each $X_i$ is
    locally compact,
  \end{enumerate}
  then $\QLV X$ is a projective limit of
  ${(\QLV {p_{ij}} \colon \QLV {X_j} \to \QLV {X_i})}_{i \sqsubseteq j
    \in I}$ and $\AvalV X$ is a projective limit of
  ${(\AvalV {p_{ij}} \colon \AvalV {X_j} \to \AvalV {X_i})}_{i
    \sqsubseteq j \in I}$.
\end{corollary}

The assumptions of Corollary~\ref{corl:QL:projlim} are not quite
tight, as the following corner case demonstrates.
\begin{remark}
  \label{ex:QL:projlim:ep}
  The fact that $X_i$ be sober is not needed when
  ${(p_{ij} \colon X_j \to X_i)}_{i \sqsubseteq j \in I}$ is the
  projective system underlying an ep-system, and in the case of the
  $\AvalV$ functor.  Explicitly, let
  $(\xymatrix{X_i \ar@<1ex>[r]^{e_{ij}} & X_j
    \ar@<1ex>[l]^{p_{ij}}})_{i \sqsubseteq j \in I}$ be an ep-system,
  with canonical projective limit $X, {(p_i)}_{i \in I}$.  Let $Z$ be
  the canonical projective limit of
  ${(\AvalV {p_{ij}} \colon \AvalV {X_j} \to \AvalV {X_i})}_{i
    \sqsubseteq j \in I}$.  Then the comparison map
  $\varphi \colon \AvalV X \to Z$ is a homeomorphism.  This is a
  consequence of Proposition~\ref{prop:ep} and the fact that
  $\AvalV Y$ is sober for every space $Y$
  \cite[Theorem~3.2]{Heckmann:absval}, hence a monotone convergence
  space, and certainly a $T_0$ space.
\end{remark}


\begin{example}
  \label{exa:ratbot:QL}
  Example~\ref{exa:ratbot:H} (where each $X_n$ is defined as
  $\rat_\bot \diff \{q_0, \cdots, q_{n-1}\}$) gives an example of a
  projective system of compact sober spaces, with a countable index
  set, whose projective limit is not preserved by $\QLV$, by
  Remark~\ref{rem:QL:projlim:nec}.  Hence sobriety is not enough in
  case~2 of Corollary~\ref{corl:QL:projlim}.  We recall that sobriety
  itself is needed, see Example~\ref{exa:QL:projlim:sober:nec}.
\end{example}


\section{Lenses}
\label{sec:lenses}

The study of lenses will require us to talk about weakly Hausdorff
spaces, and about quasi-Polish spaces.

A topological space $X$ is \emph{weakly Hausdorff} in the sense of
Lawson and Keimel \cite[Lemma~6.6]{KL:measureext} if and only if for
all $x, y \in X$, every open neighborhood $W$ of $\upc x \cap \upc y$
contains an intersection $U \cap V$ of an open neighborhood $U$ of $x$
and of an open neighborhood $V$ of $y$, equivalently if for all
compact saturated subsets $Q$, $Q'$ of $X$, every open neighborhood
$W$ of $Q \cap Q'$ contains an intersection $U \cap V$ of an open
neighborhood $U$ of $Q$ and of an open neighborhood $V$ of $Q'$.  All
Hausdorff spaces, all stably locally compact spaces are weakly
Hausdorff; see \cite{JGL:wHaus} for further information.

Quasi-Polish spaces were invented by M. de Brecht
\cite{deBrecht:qPolish}, and can be characterized in many ways.  The
original definition is as the topological spaces obtained from
second-countable Smyth-complete quasi-metric spaces in their open ball
topology, just keeping the topology and throwing away the
quasi-metric.  A space is \emph{second-countable} if and only if it
has a countable base.  We will not need to refer to quasi-metric
spaces or Smyth-completeness in the sequel, so we omit the
definitions.  Every Polish space is quasi-Polish, and also every
$\omega$-continuous dcpo from domain theory.

The latter are defined as follows.  In a dcpo $P$, let $x \ll y$
(``$x$ is \emph{way below} $y$'') if and only if every directed family
$D$ such that $y \leq \dsup D$ contains an element $d \in D$ such that
$x \leq d$.  A \emph{basis} for $P$ is a a subset $B$ such that, for
every $x \in P$, $\ddarrow_B x \eqdef \{b \in B \mid b \ll x\}$ is
directed and has $x$ as its supremum.  A dcpo is \emph{continuous} if
and only if it has a basis, and \emph{$\omega$-continuous} if and only
if it has a countable basis.  For example, for any set $I$, the dcpo
$\pow (I)$, ordered by inclusion, is continuous with basis
$\Pfin (I)$.  It is $\omega$-continuous if $I$ is countable.  (That is
in fact an example of an \emph{algebraic} domain, where the set of
\emph{finite elements} $\{x \in P \mid x \ll x\}$ serves as a basis.
Every algebraic domain is continuous.)  In a continuous dcpo $P$ with
basis $B$, the sets $\uuarrow b \eqdef \{x \in P \mid b \ll x\}$,
$b \in B$, form a base of the Scott topology.

We recall that a $G_\delta$ subset of a topological space $X$ is a
countable intersection of open subsets.  A \emph{$\bPi^0_2$ subset} of
$X$ is a countable intersection of UCO subsets, where a \emph{UCO
  subset} of $X$ is a set of the form $U \limp V$, denoting
$\{x \in X \mid \text{if }x\in U\text{ then }x \in V\}$, with
$U, V \in \Open X$.  Every open subset is UCO, so every $G_\delta$
subset is $\bPi^0_2$.  The $\bPi^0_2$ subsets are crucial in
understanding the structure of quasi-Polish spaces as, notably, the
subspaces of a quasi-Polish space that are quasi-Polish are exactly
its $\bPi^0_2$ subsets \cite[Corollary~23]{deBrecht:qPolish}.
\begin{proposition}
  \label{prop:qPolish:model}
  The following are equivalent for a topological space $X$:
  \begin{enumerate}
  \item $X$ is quasi-Polish;
  \item $X$ is homeomorphic to a $\bPi^0_2$ subspace of an
    $\omega$-continuous dcpo;
  \item $X$ is homeomorphic to a $G_\delta$ subspace of an
    $\omega$-continuous dcpo.
  \end{enumerate}
\end{proposition}
The $\omega$-continuous dcpo is given its Scott topology, and the
$G_\delta$ or the $\bPi^0_2$ subspace has the subspace topology; mind
that the latter is not a Scott topology in general.

\begin{proof}
  $3 \limp 2 \limp 1$.  Every $G_\delta$ subspace $X$ of an
  $\omega$-continuous dcpo $P$ is in particular a $\bPi^0_2$ subspace
  of $P$.  Every $\omega$-continuous dcpo is quasi-Polish
  \cite[Corollary~45]{deBrecht:qPolish}, hence so is any of its
  $\bPi^0_2$ subspace.
  

  $1 \limp 3$.  If $X$ is quasi-Polish, then $X$ has an
  $\omega$-quasi-ideal model \cite[Theorem~8.18]{GLN-lmcs17}.  An
  $\omega$-quasi-ideal domain is an algebraic domain with countably
  many finite elements, and in which every element smaller than or
  equal to a finite element is itself finite
  \cite[Definition~8.1]{GLN-lmcs17}.  Clearly every
  $\omega$-quasi-ideal domain is $\omega$-continuous.  An
  $\omega$-quasi-ideal \emph{model} of a space $X$ is an
  $\omega$-quasi-ideal domain $P$ such that $X$ is homeomorphic to the
  subspace of non-finite elements of $P$.  Listing the finite elements
  of $P$ as $p_0$, $p_1$, \ldots, $p_n$, \ldots, the sets $\dc p_n$
  are closed, so their complements $U_n$ are open, and the set of
  non-finite elements of $P$ is $\bigcap_{n \in \nat} U_n$, hence a
  $G_\delta$ subset of $P$.  \qed
\end{proof}

The point in introducing those kinds of spaces is that the space of
lenses $\Plotkinn X$ is naturally isomorphic to the space of
quasi-lenses $\QLV X$, provided that $X$ is weakly Hausdorff or a
quasi-Polish spaces, as we are going to argue.

Lemma~6.1, Proposition~6.2 and Theorem~6.3 and of \cite{JGL:wHaus}
state the following, among other things.
\begin{lemma}[\cite{JGL:wHaus}]
  \label{lemma:thm6.3}
  For every topological space $X$, there is a topological embedding
  $\iota_X \colon \Plotkinn X \to \QLV X$, defined by
  $\iota_X (L) \eqdef (\upc L, cl (L))$, and we have
  ${(\iota_X)}^{-1} (\Box^\quasi U) = \Box U$ and
  ${(\iota_X)}^{-1} (\Diamond^\quasi U) = \Diamond U$ for every
  $U \in \Open X$.  There is a map
  $\varrho_X \colon \QLV X \to \Plotkinn X$ defined by
  $\varrho_X (Q, C) \eqdef Q \cap C$, and
  $\varrho_X \circ \iota_X = \identity X$.  If $X$ satisfies the
  following property:
  \begin{quotation}
    $(*)$ for every compact saturated subset $Q$ of $X$, for every
    closed subset $C$ of $X$, if $C \subseteq cl (U \cap C)$ for every
    open neighborhood $U$ of $Q$, then $C \subseteq cl (Q \cap C)$,
  \end{quotation}
  then $\iota_X$ is a homeomorphism, with inverse $\varrho_X$.
  
  Property $(*)$ is satisfied, in particular, if $X$ is weakly
  Hausdorff.
\end{lemma}

We turn to quasi-Polish spaces, and for that we examine
$\omega$-continuous dcpos first.  For a finite set $E$, we write
$\uuarrow E$ for $\bigcup_{x \in E} \uuarrow x$.  The notation $\fcap$
refers to the intersection of a filtered family.  The following lemma
is folklore.
\begin{lemma}
  \label{lemma:Q:omegacont}
  Let $X$ be an $\omega$-continuous dcpo, with a countable basis $B$.
  For every compact saturated subset $Q$ of $X$, there is a sequence
  of finite subsets $E_n$ of $B$ such that
  $E_{n+1} \subseteq \uuarrow E_n$ for every $n \in \nat$, and such
  that
  $Q = \fcap_{n \in \nat} \uuarrow E_n = \fcap_{n \in \nat} \upc E_n$.
\end{lemma}
\begin{proof}
  Let us first observe that for every open neighborhood $U$ of $Q$,
  there is a finite subset $E$ of $B$ such that
  $Q \subseteq \uuarrow E \subseteq \upc E \subseteq U$.  Indeed, for
  each $x \in Q$, there is a point $b_x \in B$ such that $b_x \in U$
  and $b_x \ll x$.  Then ${(\uuarrow b_x)}_{x \in Q}$ is an open cover
  of $Q$.  We extract a finite subcover ${(\uuarrow b_x)}_{x \in E}$,
  and then $Q \subseteq \uuarrow E \subseteq \upc E \subseteq U$, as
  desired.
  
  Since $Q$ is saturated, $Q$ is also equal to the intersection of its
  open neighborhoods, hence also to the intersection
  $\bigcap \uuarrow E$, where $E$ ranges over the finite subsets of
  $B$ such that $Q \subseteq \uuarrow E$, by the observation we have
  just made.  Since $B$ is countable, there are countably many such
  sets $E$.  Let us enumerate them as $E^0_n$, $n \in \nat$.  We
  define a finite subset $E_n$ of $B$ by induction on $n$ in such a
  way that $Q \subseteq \uuarrow E_n$ and
  $E_{n+1} \subseteq \uuarrow E_n$ for every $n \in \nat$ as follows.
  First, $E_0 \eqdef E^0_0$.  Then, for every $n \in \nat$, we let
  $E_{n+1}$ be any finite subset of $B$ such that
  $Q \subseteq \uuarrow E_{n+1} \subseteq \upc E_{n+1} \subseteq
  \uuarrow E^0_{n+1} \cap \uuarrow E_n$, using our preliminary
  observation.  We have
  $Q \subseteq \fcap_{n \in \nat} \uuarrow E_n \subseteq \fcap_{n \in
    \nat} \upc E_n \subseteq \bigcap_{n \in \nat} \uuarrow E^0_n = Q$,
  whence the claim.  \qed
\end{proof}

\begin{lemma}
  \label{lemma:C:lens:omegacont}
  Every $\omega$-continuous dcpo $P$ (with its Scott topology)
  satisfies Property $(*)$.
\end{lemma}
\begin{proof}
  We fix a compact saturated subset $Q$ of $P$ and a closed subset $C$
  of $P$, such that for every open neighborhood $U$ of $Q$, $C
  \subseteq cl (U \cap C)$.

  Let $B$ be a countable basis of $P$, and $E_n$ be as given in
  Lemma~\ref{lemma:Q:omegacont}.  Let $x$ be any point of $B$ such
  that $\uuarrow x$ intersects $C$.  We build a monotone sequence of
  points ${(x_n)}_{n \in \nat}$ such that $x_n \in B \cap \uuarrow E_n$
  and $\uuarrow x_n$ intersects $C$ for every $n \in \nat$.  This is
  by induction on $n$.  Since $Q \subseteq \uuarrow E_0$, by
  assumption $C \subseteq cl (\uuarrow E_0 \cap C)$.  Since
  $\uuarrow x$ intersects $C$, it also intersects
  $\uuarrow E_0 \cap C$.  Let $y$ be any point in the intersection.
  Since $y$ is in the Scott-open set $\uuarrow x \cap \uuarrow E_0$,
  there is a point $x_0$ of $B$ in $\uuarrow x \cap \uuarrow E_0$ such
  that $x_0 \ll y$.  In particular, $\uuarrow x_0$ intersects $C$ (at
  $y$), and $x_0$ is in $B \cap \uuarrow E_0$.  This starts the
  induction.  Given that $\uuarrow x_n$ intersects $C$, we proceed in
  the same way to obtain $x_{n+1}$.  Since $\uuarrow x_n$ intersects
  $C$ and $Q \subseteq \uuarrow E_{n+1}$, $\uuarrow x_n$ also
  intersects $\uuarrow E_{n+1} \cap C$; we pick $y$ in the
  intersection, and $x_{n+1} \in B$ such that $x_{n+1} \ll y$ and
  $x_{n+1} \in \uuarrow x_n \cap \uuarrow E_{n+1}$.

  Let $x_\infty \eqdef \dsup_{n \in \nat} x_n$.  Since $\uuarrow x_n$
  intersects $C$ and $C$ is downwards-closed, $x_n$ is itself in $C$
  for every $n \in \nat$, so $x_\infty$ is in $C$.  Since
  $x_n \in \uuarrow E_n$ for every $n$, hence also
  $x_\infty \in \uuarrow E_n$, $x_\infty$ is in $Q$, using
  Lemma~\ref{lemma:Q:omegacont}.  Hence $x_\infty$ is in $Q \cap C$.
  Additionally, $x \ll x_0 \leq x_\infty$.  Therefore we have shown
  that for every $x \in B$, if $\uuarrow x$ intersects $C$ then it
  also intersects $Q \cap C$.  Since every Scott-open set is a union
  of sets of the form $\uuarrow x$ with $x \in B$, every open set that
  intersects $C$ also intersects $Q \cap C$.  We conclude that
  $C \subseteq cl (Q \cap C)$.  \qed
\end{proof}

\begin{proposition}
  \label{prop:dB:Kawai:lens}
  Every quasi-Polish space satisfies Property~$(*)$.
\end{proposition}
\begin{proof}
  Let us equate $X$ with a $G_\delta$ subspace of an
  $\omega$-continuous dcpo $Y$, thanks to
  Proposition~\ref{prop:qPolish:model}.  Let us write $\upc_X$,
  $\upc_Y$ for upward closures in $X$, resp.\ $Y$, and let us note
  that $\upc_Y A = \upc_X A$ for every subset $A$ of $X$, because $X$
  is upwards-closed in $Y$.  Let us also write $cl_X$, $cl_Y$ for
  closure in $X$, resp.\ $Y$, and let us note that
  $cl_X (A) = cl_Y (A) \cap X$ for every subset $A$ of $X$.

  Let $Q$ be a compact saturated subset of $X$, let $C$ be a closed
  subset of $X$, and let us assume that for every open neighborhood
  $U$ of $Q$, $C \subseteq cl_X (U \cap C)$.  We let
  $C' \eqdef cl_Y (C)$, and we claim that for every open neighborhood
  $V$ of $Q$ in $Y$, $C' \subseteq cl_Y (V \cap C')$.  Let
  $U \eqdef V \cap X$, an open neighborhood of $Q$ in $X$.  By
  assumption, $C \subseteq cl_X (U \cap C)$, in particular
  $C \subseteq cl_Y (U \cap C)$, and since
  $U \cap C = V \cap X \cap C = V \cap C \subseteq V \cap C'$,
  $C \subseteq cl_Y (V \cap C')$.  Taking closures in $Y$,
  $C' \subseteq cl_Y (V \cap C')$.

  We note that $Q$ is compact saturated in $X$, hence in $Y$, and that
  $C'$ is closed in $Y$.  By Lemma~\ref{lemma:C:lens:omegacont}, $Y$
  satisfies Property~$(*)$, so $C' \subseteq cl_Y (Q \cap C')$.  But
  $Q \cap C' = Q \cap X \cap C' = Q \cap C$, so
  $C = C' \cap X \subseteq cl_Y (Q \cap C) \cap X = cl_X (Q \cap C)$.
  \qed
\end{proof}

\begin{theorem}
  \label{thm:dB:Kawai:P}
  For every quasi-Polish space $X$, the spaces $\QLV X$ and
  $\Plotkinn X$ are quasi-Polish, and homeomorphic through $\iota_X$
  and $\varrho_X$.
\end{theorem}
\begin{proof}
  Theorem~5.1 of \cite{dBK:comm} states that, when $X$ is
  quasi-Polish, so is $\Plotkinn X$.  By
  Proposition~\ref{prop:dB:Kawai:lens}, $X$ satisfies Property~$(*)$,
  so we may apply Lemma~\ref{lemma:thm6.3} and conclude.  \qed
\end{proof}

\begin{figure}
  \centering
  \includegraphics[scale=0.4]{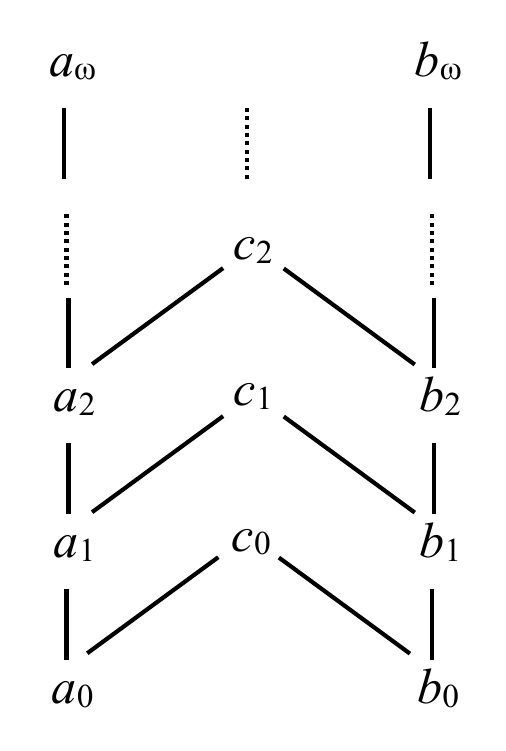}
  \caption{Knijnenburg's dcpo}
  \label{fig:knijnenburg}
\end{figure}

\begin{remark}
  \label{rem:knijnenburg}
  One may wonder whether every quasi-Polish space would simply just be
  weakly Hausdorff, in which case Theorem~\ref{thm:dB:Kawai:P} would
  be implied by Lemma~\ref{lemma:thm6.3}.  That is not true.  Consider
  the dcpo of Figure~\ref{fig:knijnenburg}, due to Peter Knijnenburg
  \cite[Example~6.1]{Knijnenburg:powerdomain}.  (We have only removed
  the bottom element from the original example.)  Its elements are
  $a_n$ and $b_n$ for every $n \in \nat \cup \{\omega\}$, and $c_n$
  for every $n \in \nat$ (not $\omega$), all pairwise distinct.  The
  ordering is given by: $a_m \leq a_n$, $b_m \leq b_n$,
  $a_m \leq c_n$, $b_m \leq c_n$ if and only if $m \leq n$; all other
  pairs of elements are incomparable.  This is an $\omega$-continuous
  dcpo, even an $\omega$-algebraic dcpo, whose finite elements are all
  elements except $a_\omega$ and $b_\omega$.  Every open neighborhood
  of $a_\omega$ intersects every open neighborhood of $b_\omega$, so
  it is not weakly Hausdorff.  One may also note that, in a weakly
  Hausdorff space, for every lens $L$, $\dc L = cl (L)$
  \cite[Theorem~6.4]{JGL:wHaus}, and that is not the case here, as
  Knijnenburg notices:
  $L \eqdef \{c_n \mid n \in \nat\} \cup \{b_\omega\}$ is a non-empty
  compact saturated subset, hence a lens, but $cl (L)$ is the whole
  space, while $\dc L$ is the whole space minus $a_\omega$.
\end{remark}

\begin{remark}
  \label{rem:anti:knijnen}
  Conversely, not every weakly Hausdorff space is quasi-Polish.  For
  example, $\rat$ with its metric topology is Hausdorff, but not
  Baire, hence not quasi-Polish.
\end{remark}

There is a $\Plotkinn$ endofunctor on $\Topcat$, and we will enquire
which projective limits it preserves.
\begin{lemma}
  \label{lemma:iota:nat}
  $\Plotkinn$ is an endofunctor on $\Topcat$, whose action on
  morphisms $f \colon X \to Y$ is given by
  $\Plotkinn f (L) \eqdef \upc f [L] \cap cl (f [L])$ for every lens
  $L$ on $X$; for every open subset $V$ of $Y$,
  ${(\Plotkinn f)}^{-1} (\Box V) = \Box {f^{-1} (V)}$ and
  ${(\Plotkinn f)}^{-1} (\Diamond V) = \Diamond {f^{-1} (V)}$.  The
  collection of maps $\iota_X$, when $X$ ranges over all topological
  spaces, is a natural transformation from $\Plotkinn$ to $\QLV$.
\end{lemma}
\begin{proof}
  Let $f \colon X \to Y$ be any continuous map.  For every lens $L$ on
  $X$, $f [L]$ is compact, so $\upc f [L]$ is compact saturated, and
  $cl (f [L])$ is clearly closed.  Given any point $x \in L$, $f (x)$
  is in $\upc f [L]$ and in $cl (f [L])$, so $\upc f [L]$ intersects
  $cl (f [L])$, showing that $\upc f [L] \cap cl (f [L])$ is a lens.

  For every open subset $V$ of $Y$, ${(\Plotkinn f)}^{-1} (\Box V)$ is
  the collection of lenses $L$ on $X$ such that
  $\upc f [L] \cap cl (f [L]) \subseteq V$.  If so, then
  $f [L] \subseteq \upc f [L] \cap cl (f [L]) \subseteq V$, so
  $L \subseteq f^{-1} (V)$, namely $V \in \Box {f^{-1} (V)}$.
  Conversely, if $L \subseteq f^{-1} (V)$, namely if
  $f [L] \subseteq V$, then $\upc f [L] \subseteq V$ since $V$ is
  upwards-closed, so $\upc f [L] \cap cl (f [L]) \subseteq V$.
  Therefore ${(\Plotkinn f)}^{-1} (\Box V) = \Box {f^{-1} (V)}$.  The
  set ${(\Plotkinn f)}^{-1} (\Diamond V)$ is the collection of lenses
  $L$ on $X$ such that $\upc f [L] \cap cl (f [L])$ intersects $V$.
  If so, then $cl (f [L])$ intersects $V$, so $f [L]$ intersects $V$,
  namely $L \cap f^{-1} (V) \neq \emptyset$, or equivalently
  $L \in \Diamond {f^{-1} (L)}$.  Conversely, if $f [L]$ intersects
  $V$, then the larger set $\upc f [L] \cap cl (f [L])$ also intersects
  $V$.  Therefore
  ${(\Plotkinn f)}^{-1} (\Diamond V) = \Diamond {f^{-1} (V)}$. All
  this shows that $\Plotkinn f$ is continuous.

  In order to show that
  $\Plotkinn {\identity X} = \identity {\Plotkinn X}$, we need to show
  that $L = \upc L \cap cl (L)$ for every lens $L$ on $X$.  This is
  the fact that $\varrho_X \circ \iota_X = \identity X$, see
  Lemma~\ref{lemma:thm6.3}.  We also need to show that
  $\Plotkinn {(g \circ f)} = \Plotkinn g \circ \Plotkinn f$ for all
  continuous maps $f \colon X \to Y$ and $g \colon Y \to Z$.  We
  realize that the inverse image of any subbasic open set $\Box W$
  (resp.\ $\Diamond W$), $W \in \Open Z$, by any side of the equality
  is equal to $\Box {(f^{-1} (g^{-1} (W)))}$ (resp.,
  $\Diamond {(f^{-1} (g^{-1} (W)))}$) Hence the inverse images of any
  open subset of $\Plotkinn Z$ by the two sides of the equality are
  the same.  But any two continuous maps from a space to a $T_0$ space
  with that property are equal.

  Finally, we need to verify that
  $\QLV f \circ \iota_X = \iota_Y \circ \Plotkinn f$, for every
  continuous map $f \colon X \to Y$.  We use the same trick.  Using
  the first part of Lemma~\ref{lemma:thm6.3} notably, the inverse
  image of every subbasic open subset $\Box^\quasi V$ (resp.,
  $\Diamond^\quasi V$) of $\QLV Y$ by each side of the equality is
  $\Box {f^{-1} (V)}$, resp.\ $\Diamond {f^{-1} (V)}$.  We conclude
  that the equality holds, since $\QLV Y$ is $T_0$.  \qed
\end{proof}

\begin{proposition}
  \label{prop:Plotkinn:projlim}
  The comparison map $\varphi \colon \Plotkinn X \to Z$ of any
  projective $\Plotkinn$-situation is a topological embedding.
\end{proposition}
\begin{proof}
  We use Lemma~\ref{lemma:T:projlim}, first checking that $\Plotkinn$
  is $R$-nice with $R \eqdef \{\mathsf M, 1\}$,
  $B_X (1, U) \eqdef \Box U$, $B_X (\mathsf M, U) \eqdef \Diamond U$.
  Property~1 of Definition~\ref{defn:T:nice:S} stems from the fact
  that $\Box$ and $\Diamond$ are Scott-continuous; $\Diamond$ even
  commutes with arbitrary unions, and the argument for $\Box$ is as in
  Proposition~\ref{prop:QL:projlim}, considering that every lens $L$
  is compact: for every directed family ${(U_i)}_{i \in I}$ of open
  subsets of $X$, $L \in \Box {\dcup_{i \in I} U_i}$ if and only if
  $L \subseteq \dcup_{i \in I} U_i$, if and only if $L \subseteq U_i$
  for some $i \in I$ (by compactness), if and only if
  $L \in \dcup_{i \in I} \Box {U_i}$.  Property~2 follows from the
  characterization of ${(\Plotkinn f)}^{-1}$ given in the first part
  of Lemma~\ref{lemma:iota:nat}.  \qed
\end{proof}

\begin{theorem}
  \label{thm:Plotkinn:projlim}
  Let ${(p_{ij} \colon X_j \to X_i)}_{i \sqsubseteq j \in I}$ be a
  projective system of topological spaces, with canonical projective
  limit $X, {(p_i)}_{i \in I}$.  If every $X_i$ is sober, if $\iota_X$
  is surjective, and if $\HV X$ is a projective limit of
  ${(\HV {p_{ij}} \colon \HV {X_j} \to \HV {X_i})}_{i \sqsubseteq j
    \in I}$, then $\Plotkinn X$ is a projective limit of
  ${(\Plotkinn {p_{ij}} \colon \Plotkinn {X_j} \to \Plotkinn
    {X_i})}_{i \sqsubseteq j \in I}$.
\end{theorem}
\begin{proof}
  By Proposition~\ref{prop:Plotkinn:projlim}, it suffices to show that
  the comparison map $\varphi \colon \Plotkinn X \to Z$ is surjective,
  where $Z$ is the canonical projective limit of
  ${(\Plotkinn {p_{ij}} \colon \Plotkinn {X_j} \to \Plotkinn
    {X_i})}_{i \sqsubseteq j \in I}$.  Let ${(L_i)}_{i \in I}$ be an
  element of the latter.  We form the quasi-lenses
  $(Q_i, C_i) \eqdef \iota_{X_i} (L_i)$ for each $i \in I$.  For all
  $i \sqsubseteq j \in I$, $L_i = \Plotkinn {p_{ij}} (L_j)$, so
  $(Q_i, C_i) = \QLV {p_{ij}} (Q_j, C_j)$ by naturality of $\iota$
  (Lemma~\ref{lemma:iota:nat}).  Using Theorem~\ref{thm:QL:projlim},
  there is a (unique) quasi-lens $(Q, C)$ on $X$ such that
  $(Q_i, C_i) = \QLV {p_i} (Q, C)$ for every $i \in I$.  Since we are
  assuming that $\iota_X$ is surjective, there is a lens $L$ on $X$
  such that $(Q, C) = \iota_X (L)$.  Then, for every $i \in I$,
  $\iota_{X_i} (L_i) = (Q_i, C_i) = \QLV {p_i} (\iota_X (L)) =
  \iota_{X_i} (\Plotkinn {p_i} (L))$, by naturality of $\iota$.  Since
  $\iota_{X_i}$ is injective (being a topological embedding, see
  Lemma~\ref{lemma:thm6.3}), $L_i = \Plotkinn {p_i} (L)$.  \qed
\end{proof}

One case where $\iota_X$ is surjective, or equivalently, a
homeomorphism, where $X$ is as in Theorem~\ref{thm:Plotkinn:projlim},
is when $X$ is weakly Hausdorff, by Lemma~\ref{lemma:thm6.3}.  This
happens notably when every $X_i$ is locally strongly sober, as we now
argue.  The original definition of a locally strongly sober space is a
space in which the collection of limits of every convergent
ultrafilter is the closure of a unique point
\cite[Definition~VI-6.12]{GHKLMS:contlatt}.  A space is locally
strongly sober if and only if it is sober, coherent, and weakly
Hausdorff \cite[Theorem~3.5]{JGL:wHaus}, and every projective limit of
locally strongly sober spaces is locally strongly sober
\cite[Theorem~5.1]{JGL:proj:class}.  We note that every Hausdorff
space, every stably locally compact space is weakly Hausdorff
\cite[Proposition~2.2]{JGL:wHaus}, and since they are all sober and
coherent, they are all locally strongly sober.

Another case where $\iota_X$ is surjective is when $X$ is
quasi-Polish, using Proposition~\ref{prop:dB:Kawai:lens} and
Lemma~\ref{lemma:thm6.3}.  Now any projective limit of quasi-Polish
spaces is quasi-Polish \cite[Proposition~9.5]{JGL:kolmogorov}, so we
are in this situation if every $X_i$ is quasi-Polish.  We recall that
every Polish space is quasi-Polish.

Hence, combining Theorem~\ref{thm:Plotkinn:projlim} with
Theorem~\ref{thm:H:projlim} (or Corollary~\ref{corl:QL:projlim}), we
obtain the following.
\begin{corollary}
  \label{corl:Plotkinn:projlim}
  Let ${(p_{ij} \colon X_j \to X_i)}_{i \sqsubseteq j \in I}$ be a
  projective system of topological spaces, with canonical projective
  limit $X, {(p_i)}_{i \in I}$.  If every $X_i$ is locally strongly
  sober, or if every $X_i$ is quasi-Polish, and:
  \begin{enumerate}
  \item every $p_{ij}$ is a proper map (e.g., a projection),
  \item or $I$ has a countable cofinal subset and each $X_i$ is
    locally compact,
  \end{enumerate}
  then $\Plotkinn X$ is a projective limit of
  ${(\Plotkinn {p_{ij}} \colon \Plotkinn {X_j} \to \Plotkinn
    {X_i})}_{i \sqsubseteq j \in I}$.
\end{corollary}

In case~2, we note that the combination of the requirements of $X_i$
being locally compact and locally strongly sober is equivalent to
requiring that $X_i$ be stably locally compact
\cite[Proposition~VI-6.15, Corollary~VI-6.16]{GHKLMS:contlatt}.
Requiring instead that $X_i$ be locally compact and quasi-Polish is
equivalent to requiring that $X_i$ be locally compact, sober and
second-countable.  Indeed, every quasi-Polish space is
second-countable by definition, and conversely every locally compact,
sober, second-countable space is quasi-Polish
\cite[Theorem~44]{deBrecht:qPolish}.

\section{Subcontinuation functors}
\label{sec:subc-funct}

All the functors we will consider from now on are subcontinuation
functors, in a sense we define below.  (All the functors we have
examined until now are naturally isomorphic to subcontinuation
functors, too, but it was easier to deal with them as we did.)  We
will see that, whenever $T$ is a subcontinuation functor, the
comparison maps $\varphi \colon TX \to Z$ are topological embeddings,
and even homeomorphisms when $X$ is obtained as a limit of an
ep-system.

For every space $X$, let $\Lform X$ be the set of lower semicontinuous
maps from $X$ to $\creal$, namely the set of continuous maps from $X$
to $\creal$ where the latter is given the Scott topology of its usual
ordering.  $\Lform X$ is ordered pointwise, and we give it the Scott
topology of that ordering.

\begin{definition}
  \label{defn:subcont}
  A \emph{subcontinuation functor} $T$ is an endofunctor on $\Topcat$
  such that:
  \begin{itemize}
  \item for every space $X$, $TX$ is a subspace of the space $KX$ of
    lower semicontinuous maps from $\Lform X$ to $\creal$, with the
    topology generated by subbasic open sets
    $[h > r] \eqdef \{F \in TX \mid F (h) > r\}$, $h \in \Lform X$
    (the subspace topology induced by the inclusion in the product
    $\creal^{\Lform X}$);
  \item for every morphism $f \colon X \to Y$, $Tf$ maps every
    $F \in TX$ to the function $h \in \Lform Y \mapsto F (h \circ f)$.
  \end{itemize}
\end{definition}
$K$ itself, which maps every space $X$ to the space $KX$ of lower
semicontinuous maps from $\Lform X$ to $\creal$, is the largest
subcontinuation functor, which one may call the \emph{continuation}
functor.  The name is by analogy with the continuation monad used in
the denotational semantics of programming languages, with answer type
$\creal$.

We need the following.
\begin{lemma}
  \label{lemma:hi}
  Let ${(p_{ij} \colon X_j \to X_i)}_{i \sqsubseteq j \in I}$ be a
  projective system in $\Topcat$, with canonical projective limit
  $X, {(p_i)}_{i \in I}$.  For every $h \in \Lform X$,
  \begin{enumerate}
  \item there is a largest function $h_i \in \Lform {X_i}$ such that
    $h_i \circ p_i \leq h$, for every $i \in I$;
  \item for all $i \sqsubseteq j \in I$, $h_i \circ p_{ij} \leq h_j$;
  \item for all $i \sqsubseteq j \in I$, $h_i \circ p_i \leq h_j \circ
    p_j$;
  \item for every $r \in \Rp$, $h^{-1} (]r, \infty]) = \dcup_{i \in I}
    p_i^{-1} (h_i^{-1} (]r, \infty]))$;
  \item $\dsup_{i \in I} (h_i \circ p_i) = h$.
  \end{enumerate}
\end{lemma}
\begin{proof}
  1. Any pointwise supremum of lower semicontinuous maps is lower
  semicontinuous, including the empty supremum, which is the constant
  zero map.  Let $\mathcal F \eqdef \{k \in \Lform {X_i} \mid k \circ p_i
  \leq h\}$, and $h_i$ be its pointwise supremum.  Then $h_i \in
  \Lform {X_i}$, and it is clear that $h_i \circ p_i \leq h$, so $h_i$
  is the largest element of $\mathcal F$.

  2. We have $h_i \circ p_i = (h_i \circ p_{ij}) \circ p_j \leq h$.
  By the maximality of $h_j$, $h_i \circ p_{ij} \leq h_j$.

  3. By item~2, post-composing with $p_j$.

  4. Let $V \eqdef h_i^{-1} (]r, \infty])$.  Letting $\chi_V$ be the
  characteristic function of $V$, $r \chi_V$ is a lower semicontinuous
  map such that $r \chi_V \circ p_i \leq h$: indeed, for every
  $x \in X$, if $p_i (x) \in V$ then $h (x) \geq h_i (p_i (x)) > r$.
  By the maximality of $h_i$, $r \chi_V \leq h_i$, which means that
  every point $x$ of $p_i^{-1} (h_i^{-1} (]r, \infty]))$ is such that
  $r \chi_V (p_i (x)) = r \leq h_i (x)$.


  4. Since $h_i \circ p_i \leq h$ for every $i \in I$,
  $h^{-1} (]r, \infty]) \supseteq \dcup_{i \in I} p_i^{-1} (h_i^{-1}
  (]r, \infty]))$.  For the reverse inclusion, let $x$ be any point in
  $h^{-1} (]r, \infty])$.  We pick $t \in \Rp$ such that
  $r < t < h (x)$, and we let $U \eqdef h^{-1} (]t, \infty])$.  We
  recall that there is a largest open subset $U_i$ of $X_i$ such that
  $p_i^{-1} (U_i) \subseteq U$, for every $i \in I$, and that
  $\dcup_{i \in I} p_i^{-1} (U_i) = U$.  Hence $x \in p_i^{-1} (U_i)$
  for some $i \in I$.  We note that $t \chi_{U_i} \circ p_i \leq h$,
  since every point mapped by $p_i$ into $U_i$ is in
  $p_i^{-1} (U_i) \subseteq U = h^{-1} (]t, \infty])$.  By maximality
  of $h_i$, $t \chi_{U_i} \leq h_i$.  Then
  $h_i (p_i (x)) \geq t \chi_{U_i} (p_i (x)) = t$, since
  $x \in p_i^{-1} (U_i)$.  Since $t > r$, it follows that
  $x \in p_i^{-1} (h_i^{-1} (]r, \infty]))$.

  5. It suffices to observe that, for every $x \in X$, for every $i
  \in I$, for every $r \in \Rp$, $r < \dsup_{i \in I} (h_i \circ p_i)
  (x)$ if and only if $x \in \dcup_{i \in I} p_i^{-1} (h_i^{-1} (]r,
  \infty]))$, $r < h (x)$ if and only if $x \in h^{-1} (]r, \infty])$,
  and to apply item~4.  \qed
\end{proof}

\begin{lemma}
  \label{lemma:T:projlim:subcont}
  Let $T$ be a subcontinuation functor.  Given any projective
  $T$-situation as given in Definition~\ref{defn:situation}, the
  comparison map $\varphi$ is a topological embedding.
\end{lemma}
\begin{proof}
  Let $\varphi \colon TX \to Z$, where $Z, {(q_i)}_{i \in I}$ is the
  canonical projective limit of the projective system
  $(T {p_{ij}} \colon T {X_j} \to T {X_i})_{i \sqsubseteq j \in I}$,
  and $X, {(p_i)}_{i \in I}$ is that of
  ${(p_{ij} \colon X_j \to X_i)}_{i \sqsubseteq j \in I}$.

  For every $h \in \Lform X$, for every $i \in I$, let
  $h_i \in \Lform {X_i}$ be the largest such that
  $h_i \circ p_i \leq h$, as given in Lemma~\ref{lemma:hi}.
  For every subbasic open subset $[h > r]$ of $TX$, with
  $h \in \Lform X$ and $r \in \Rp$, we wish to show that $[h > r]$ is
  the inverse image under $\varphi$ of some open subset of $Z$.  We
  note that for every $F \in TX$, $F \in [h > r]$ if and only if
  $F (\dsup_{i \in I} (h_i \circ p_i)) > r$ (by Lemma~\ref{lemma:hi},
  item~4), if and only if $\dsup_{i \in I} F (h_i \circ p_i) > r$
  (since $F$ is Scott-continuous) if and only if
  $F (h_i \circ p_i) > r$ for some $i \in I$; so
  $[h > r] = \dcup_{i \in I} [h_i \circ p_i > r]$.  Now we note that
  $[h_i \circ p_i > r] = \varphi^{-1} (q_i^{-1} ([h_i > r]))$.
  Indeed, $q_i \circ \varphi = T {p_i}$, so the elements of
  $\varphi^{-1} (q_i^{-1} ([h_i > r]))$ are exactly the elements
  $F \in TX$ such that $(q_i \circ \varphi) (F) \in [h_i > r]$, namely
  such that $T {p_i} (F) (h_i) > r$, equivalently such that
  $F (h_i \circ p_i) > r$; those are exactly the elements of
  $[h_i \circ p_i > r]$.  Hence
  $[h > r] = \dcup_{i \in I} \varphi^{-1} (q_i^{-1} ([h_i > r])) =
  \varphi^{-1} (\dcup_{i \in I} q_i^{-1} ([h_i > r]))$, showing that
  $\varphi$ is full.

  Finally, the preorder of specialization of $TX$ is given by
  $F \leq F'$ if and only if for all $h \in \Lform X$ and $r \in \Rp$,
  $F \in [h > r]$ implies $F' \in [h > r]$, if and only if
  $F (h) \leq F' (h)$ for every $h \in \Lform X$.  This is
  antisymmetric, so $TX$ is $T_0$. It follows that $\varphi$ is a
  topological embedding.  \qed
\end{proof}
Directed suprema, in fact arbitrary suprema, of elements of $KX$ are
again in $KX$, where $K$ is the continuation functor, and $X$ is an
arbitrary space.  This is because arbitrary suprema of lower
semicontinuous maps are lower semicontinuous.  A \emph{subdcpo} of a
dcpo $P$ is a subset $A$ of $P$ such that the supremum of any directed
family $D \subseteq A$, taken in $P$, belongs to $A$.  This entails
that $A$ is itself a dcpo, but the property is strictly stronger.  By
some abuse of language, we will extend this to subdcpos of $KX$,
implicitly seing the latter as a dcpo.
\begin{proposition}
  \label{prop:ep:subcont}
  Let $T$ be a subcontinuation functor.  Given any projective
  $T$-situation as given in Definition~\ref{defn:situation}, whose
  projective system is an ep-system, and if $TX$ is a subdcpo of $KX$,
  then the comparison map $\varphi$ is a homeomorphism.
\end{proposition}
\begin{proof}
  We take the same notations as in Definition~\ref{defn:situation}
  ($T$-situations).  Let $e_{ij}$ be embeddings associated with each
  of the projections $p_{ij}$.  By \cite[Lemma~4.1]{JGL:kolmogorov},
  each $p_i$ is a projection, and there are associated embeddings
  $e_i \colon X_i \to X$, such that $e_j \circ e_{ij} = e_i$ for all
  $i \sqsubseteq j \in I$.  Moreover,
  ${(e_i (p_i (\vec x)))}_{i \in I, \sqsubseteq}$ is a monotone net
  with supremum equal to $\vec x$ for every $\vec x \in X$.

  Using Lemma~\ref{lemma:T:projlim:subcont}, it remains to show that
  $\varphi$ is surjective.

  Let ${(F_i)}_{i \in I}$ be any element of $Z$, that is, each $F_i$
  is in $T X_i$ and $T {p_{ij}} (F_j) = F_i$ for all
  $i \sqsubseteq j \in I$.  The elements $T {e_i} (F_i) \in TX$ form a
  monotone net, namely $T {e_i} (F_i) \leq T {e_j} (F_j)$ for all
  $i \sqsubseteq j \in I$.  Indeed, this follows from the fact that
  $e_i \circ p_{ij} \leq e_j$ (because $e_j \circ e_{ij} = e_i$
  $e_{ij} \circ p_{ij} \leq \identity {X_j}$, and $e_i$ is continuous,
  hence monotonic).  Then, for every $h \in \Lform X$,
  $T {e_j} (F_j) (h) = F_j (h \circ e_j) \geq F_j (h \circ e_i \circ
  p_{ij})$ (since composition with the continuous map $h$ is
  monotonic, and $F_j$ is continuous hence monotonic as well)
  $= T {p_{ij}} (F_j) (h \circ e_i) = F_i (h \circ e_i) = T {e_i}
  (F_i) (h)$.

  Since $TX$ is a subdcpo of $KX$, the monotone net
  ${(T {e_j} (F_j))}_{j \in I, \sqsubseteq}$ has a pointwise supremum,
  which is in $TX$.  Let us call it $F$.  We show that
  $\varphi (F) = {(F_i)}_{i \in I}$, or equivalently, that
  $T {p_i} (F) = F_i$ for every $i \in I$.  We consider any
  $h \in \Lform X$, and we aim to prove that
  $T {p_i} (F) (h) = F_i (h)$, namely that
  $F (h \circ p_i) = F_i (h)$, or equivalently, that
  $\dsup_{j \in J} F_j (h \circ p_i \circ e_j) = F_i (h)$.  By taking
  $j \eqdef i$ and recalling that $p_i \circ e_i = \identity {X_i}$,
  we see that the left-hand side is larger than or equal to the
  right-hand side.  For the other inequality, we consider any
  $j \in J$ and we show that
  $F_j (h \circ p_i \circ e_j) \leq F_i (h)$.  Let us pick $k \in I$
  such that $i, j \sqsubseteq k$.  Then
  $p_i \circ e_j \circ p_{jk} \leq p_{ik}$: indeed,
  $p_i \circ e_j = p_{ik} \circ p_k \circ e_k \circ e_{jk} = p_{ik}
  \circ e_{jk}$, so
  $p_i \circ e_j \circ p_{jk} \leq p_{ik} \circ e_{jk} \circ p_{jk}
  \leq p_{ik}$, using implicitly that continuous maps are monotonic.
  Therefore $F_j (h \circ p_i \circ e_j)$, which is equal to
  $T {p_{jk}} (F_k) (h \circ p_i \circ e_j) = F_k (h \circ p_i \circ
  e_j \circ p_{jk})$ is less than or equal to $F_k (h \circ p_{ik})$
  (since $F_j$ and $h$ are themselves continuous hence monotonic), and
  the latter is equal to $T {p_{ik}} (F_k) (h) = F_i (h)$.  \qed
\end{proof}

\section{Superlinear previsions and retracts}
\label{sec:prev-powerc}

Previsions form models of mixed non-deterministic and probabilistic
choice \cite{Gou-csl07}, and are an elaboration on Walley's notion of
prevision in economics \cite{Walley:prev}.  We will borrow most of
what we need from \cite{JGL-mscs16}, see also the errata
\cite{JGL:mscs16:errata}.  A \emph{prevision} on a space $X$ is a
Scott-continuous map $F \colon \Lform X \to \creal$ that is
\emph{positively homogeneous} in the sense that $F (ah)=aF(h)$ for all
$a \in \Rp$ and $h \in \Lform X$.  There is a space $\Pred X$ of
previsions on $X$, whose topology is generated by sets
$[h > r] \eqdef \{F \mid F (h) > r\}$, $h \in \Lform X$, $r \in \Rp$.

For example, any continuous valuation $\nu$ on $X$ gives rise to a
prevision $G \colon h \mapsto \int h \,d\nu$.  Such a prevision is
\emph{linear}, in the sense that $G (h+h') = G (h) + G (h')$ for all
$h, h' \in \Lform X$.  Let $\Pred_\Nature X$ be the subspace of
$\Pred X$ of linear previsions.  Conversely, every linear prevision
$G \in \Pred_\Nature X$ gives rise to a continuous valuation
$U \mapsto G (\chi_U)$, where $\chi_U$ is the characteristic map of
the open set $U$, and the two constructions are inverse of each other.
Additionally, those two constructions define continuous maps between
$\Val X$ and $\Pred_\Nature X$ \cite[Satz~4.16]{Tix:bewertung}.  We
will therefore equate continuous valuations with linear previsions.

A prevision is \emph{sublinear} (resp., \emph{superlinear}) if and
only if $G (h+h') \leq G (h) + G (h')$ (resp., $\geq$) for all
$h, h' \in \Lform X$.  As in \cite{JGL-mscs16}, we write
$\Pred_{\AN} X$ for the subspace of $\Pred X$ consisting of sublinear
previsions, and $\Pred_{\DN} X$ for the subspace of $\Pred X$
consisting of superlinear previsions.

Among the continuous valuations, there are the probability valuations
and the subprobability valuations.  Similarly, we say that a prevision
$F$ is \emph{subnormalized} (resp., \emph{normalized}) iff
$F (\one+h) \leq \one+F (h)$ (resp., $=$) for every $h \in \Lform X$,
where $\one$ is the constant function with value $1$.  The
homeomorphism between $\Val X$ and $\Pred_\Nature X$ restricts to
homeomorphisms between $\Val_{\leq 1} X$ (resp., $\Val_1 X$) and the
subspace $\Pred_\Nature^{\leq 1} X$ (resp., $\Pred_\Nature^1 X$) of
subnormalized (resp., normalized) linear previsions on $X$.  We write
$\Pred_{\AN}^{\leq 1} X$, $\Pred_{\DN}^{\leq 1} X$,
$\Pred_{\DN}^{\leq 1} X$, $\Pred_{\DN}^1 X$ for the corresponding
spaces of (sub)normalized, sublinear/superlinear previsions.  In
general, we write $\Pred_{\AN}^\bullet X$ or $\Pred_{\DN}^\bullet X$,
where $\bullet$ can be nothing, ``$\leq 1$'', or ``$1$''.

All those constructions define endofunctors on $\Topcat$, whose action
$\Pred f$ on morphisms $f \colon X \to Y$ is given by
$\Pred f (F) (h) \eqdef F (h \circ f)$.  We write $\Pred f$ without
any $\bullet$ superscript or any subscript $\Nature$, $\AN$ or $\DN$
because the action is defined in the same way for all functors.  It is
easy to check that $\Pred f$ is a morphism from
$\Pred_\Nature^\bullet X$ to $\Pred_\Nature^\bullet Y$ for every
continuous map $f \colon X \to Y$ and similarly with $\AN$ or $\DN$ in
place of $\Nature$.  Hence all prevision functors are subcontinuation
functors in the sense of Definition~\ref{defn:subcont}.

Additionally, this construction is compatible with the homeomorphisms
$\Val_\bullet X \cong \Pred_\Nature^\bullet X$, meaning that those
homeomorphisms are natural.  Explicitly, for every
$F \in \Pred_\Nature^\bullet X$, letting $\nu$ be the associated
continuous valuation defined by $\nu (U) \eqdef F (\chi_U)$ for every
$U \in \Open X$, for every continuous map $f \colon X \to Y$, the
continuous valuation $\nu'$ associated with $\Pred f (F)$ is equal to
$f [\nu]$: for every $V \in \Open Y$,
$\nu' (V) = \Pred f (F) (\chi_V) = F (\chi_V \circ f) = F
(\chi_{f^{-1} (V)}) = \nu (f^{-1} (V)) = f [\nu] (V)$.

Our plan for establishing projective limit preservation theorems for
prevision functors---apart from the case of ep-systems, which will
follow from Proposition~\ref{prop:ep:subcont}---is to rely on the fact
that spaces of previsions on $X$ are retracts of
$\SV (\Val_\bullet X)$, resp.\ $\HV (\Val_\bullet X)$ under some
conditions \cite[Proposition~3.11, Proposition
3.22]{JGL-mscs16,JGL:mscs16:errata} and to reuse our limit preservation
theorems for $\SV$, $\HV$, and $\Val_\bullet$.

A \emph{retraction} on a category $\catc$ (of $Y$ \emph{onto} $X$) is
a pair $\xymatrix{X \ar@<1ex>[r]^s & Y \ar@<1ex>[l]^r}$ such that
$r \circ s = \identity X$.  (Hence an ep-pair is a special case of a
retraction.)  We also say that $r$, by itself, is the retraction, with
associated \emph{section} $s$, and that $X$ is a \emph{retract} of
$Y$.

We call a \emph{natural retraction}
$\xymatrix{S \ar@<1ex>[r]^{s} & T \ar@<1ex>[l]^{r}}$ of a functor
$T \colon \catc \to \catd$ onto a functor $S \colon \catc \to \catd$
any retraction in the category $\catd^\catc$ of functors from $\catc$
to $\catd$.  Explicitly, this is a collection of retractions
$\xymatrix{SX \ar@<1ex>[r]^{s_X} & TX \ar@<1ex>[l]^{r_X}}$, one for
each object $X$ of $\catc$, which are natural in $X$.

This will be fine for $\Pred_\DN^\bullet$, but we will need the
following refinement in the case of $\Pred_\AN^\bullet$.

Given a diagram $F \colon \cati \to \catc$ with a limit
$X, {(p_i)}_{i \in |\cati|}$, there is a small category $\cati_*$
obtained by adjoining a fresh object $*$ to $\cati$, with unique
morphisms from $*$ to all objects of $\cati_*$, and there is a functor
$F_* \colon \cati_* \to \catc$ that extends $F$ and such that
$F_* (*) = X$, $F_* (* \to X_i) = p_i$ for every $i \in |\cati|$.
Below, and as is customary in category theory, we write $SF$ for
$S \circ F$, and similarly with $TF$.  We write $S_{|\catk}$ for the
restriction of $S$ to $\catk$, too.
\begin{definition}
  \label{defn:Frel:retr}
  Let $S$, $T$ be two functors from a category $\catc$ to a category
  $\catd$.  Given a diagram $F \colon \cati \to \catc$, with a limit
  $X, {(p_i)}_{i \in |\cati|}$, an \emph{$F$-relative natural
    retraction} of $T$ onto $S$ is a natural retraction of
  $T_{|\catk}$ onto $S_{|\catk}$, for some subcategory $\catk$ of
  $\catc$ that contains the image of $F_*$.
\end{definition}
In other words, instead of requiring the natural retraction to exist
on the whole category $\catd$, we only require it to exist on a
sufficiently large subcategory $\catk$.  In all cases we will
encounter, $\catk$ will be a full subcategory of $\catd$.  A subtle
point of this definition is that $\catk$ should contain not just the
objects $F (i)$, $i \in |\cati|$, but also the limit $X$ (and also,
all the required morphisms between them, which will hardly be a
problem if $\catk$ is a full subcategory of $\catd$).  For example,
consider a natural retraction $r, s$ on the category of locally
compact sober spaces, and assume that each $F (i)$ is locally compact
sober: this is not enough to make it an $F$-natural retraction, since
the limit itself may fail to be locally compact
\cite[Proposition~3.4]{JGL:proj:class}.


\begin{lemma}
  \label{lemma:retract:limits}
  Let $F \colon \cati \to \catc$ be a diagram with a limit
  $X, {(p_i)}_{i \in |\cati|}$, and let $S$ and $T$ be two functors
  from $\catc$ to a category $\catd$.  If there is an $F_*$-relative
  natural retraction
  $\xymatrix{S \ar@<1ex>[r]^{s} & T \ar@<1ex>[l]^{r}}$, and if
  $TX, {(T p_i)}_{i \in |\cati|}$ is a limit of $TF$, then
  $SX, {(S p_i)}_{i \in |\cati|}$ is a limit of $SF$.
\end{lemma}
\begin{proof}
  It is clear that $SX, {(S {p_i})}_{i \in |\cati|}$ is a cone of
  $SF$.  In order to show that it is universal, let
  $Y, {(q_i)}_{i \in |\cati|}$ be another cone of $SF$.  Then
  $Y, {(s_{F (i)} \circ q_i)}_{i \in |\cati|}$ is a cone of $TF$: for
  every morphism $\varphi \colon j \to i$ in $\cati$,
  $TF (\varphi) \circ s_{F (j)} \circ q_j = s_{F (i)} \circ SF
  (\varphi) \circ q_j = s_{F (i)} \circ q_i$ by $F_*$-relative
  naturality of $s$ and the definition of a cone of $SF$.  By
  assumption, $TX, {(T {p_i})}_{i \in I}$ is a limit of $TF$, so there
  is a unique morphism $f \colon Y \to TX$ such that
  $ T {p_i} \circ f = s_{F (i)} \circ q_i$ for every $i \in |\cati|$.
  Then $r_{F (i)} \circ T {p_i} \circ f = q_i$ for every
  $i \in |\cati|$, since $r_{F (i)}$ and $s_{F (i)}$ form a
  retraction.  By $F_*$-relative naturality of $r$,
  $r_{F (i)} \circ T {p_i} \circ f = S {p_i} \circ r_X \circ f$, so we
  have found a morphism $g$ such that $S {p_i} \circ g = q_i$ for
  every $i \in |\cati|$, namely $r_X \circ f$.  This is the only one:
  given any morphism $g \colon Y \to SX$ such that
  $S {p_i} \circ g = q_i$ for every $i \in |\cati|$, we must have
  $s_{F (i)} \circ S {p_i} \circ g = s_{F (i)} \circ q_i$ for every
  $i \in |\cati|$, namely
  $T {p_i} \circ s_X \circ g = s_{F (i)} \circ q_i$ for every
  $i \in |\cati|$, by $F_*$-relative naturality of $s$.  By the
  uniqueness of $f$, $f = s_X \circ g$, so $r_X \circ f = g$ since
  $r_X \circ s_X = \identity X$.  Hence $g$ is unique.  \qed
\end{proof}

By Proposition~3.22 of \cite{JGL-mscs16}, for every topological space
$X$, and letting $\bullet$ be nothing, ``$\leq 1$'', or ``$1$'', there
is a retraction
$r_{\DN} \colon \SV (\Pred_\Nature^\bullet X) \to \Pred_{\DN}^\bullet
X$, defined by $r_{\DN} (Q) (h) \eqdef \min_{G \in Q} G (h)$, with
associated section $s_{\DN}$ defined by
$s_{\DN}^\bullet (F) \eqdef \{G \in \Pred_\Nature^\bullet X \mid G
\geq F\}$.  (The ordering $\leq$ between previsions is the
specialization ordering, which is pointwise, and $\geq$ is the
opposite ordering.)  We write them $r_{\DN\,X}$ and
$s_{\DN\,X}^\bullet$ in order to make the dependency on $X$ explicit,
reserving the notations $r_{\DN}$ and $s_{\DN}^\bullet$ for the
families of maps $r_{\DN\,X}$, resp.\ $s_{\DN\,X}^\bullet$, where $X$
ranges over topological spaces.

This retraction even cuts down to a homeomorphism between
$\SV^{cvx} (\Pred_\Nature^\bullet X)$ and $\Pred_\DN^\bullet X$
\cite[Theorem~4.15]{JGL-mscs16}, where the former denotes the subspace
of $\SV (\Pred_\Nature^\bullet X)$ consisting of convex compact
saturated subsets of $\Pred_\Nature^\bullet X$.  (A subset $A$ of the
latter is \emph{convex} if and only if for all $G_1, G_2 \in A$, for
every $r \in [0, 1]$, $r G_1 + (1-r) G_2 \in A$.)
\begin{lemma}
  \label{lemma:rDP:nat}
  The transformations $r_\DN$ and $s_\DN^\bullet$ are natural.
\end{lemma}
\begin{proof}
  Let $f \colon X \to Y$ be any continuous map.  For $r_\DN$, we need
  to show that for every $Q \in \SV (\Pred_\Nature^\bullet X)$, for
  every $h \in \Lform Y$,
  $r_{\DN\,Y} (\SV (\Pred f) (Q)) (h) = \Pred f (r_{\DN\,X} (Q)) (h)$.
  The left-hand side is equal to
  $\min_{G' \in \SV (\Pred f) (Q)} G' (h) = \min_{G' \in \upc \{\Pred
    f (G) \mid G \in Q\}} G' (h) = \min_{G' \in \{\Pred f (G) \mid G
    \in Q\}} G' (h) = \min_{G \in Q} \Pred f (G) (h)$, while the
  right-hand side is equal to
  $r_{\DN\,X} (Q) (h \circ f) = \min_{G \in Q} G (h \circ f)$, and
  those are equal.

  For $s_\DN^\bullet$, we must show that for every
  $F \in \Pred_\DN^\bullet X$,
  $s_{\DN\,Y}^\bullet (\Pred f (F)) = \SV (\Pred f) \allowbreak
  (s_{\DN\,X}^\bullet (F))$.  The left-hand side is in
  $\SV^{cvx} (\Pred_\Nature^\bullet Y)$, and we claim that so is the
  right-hand side; it suffices to show that it is convex.  We consider
  any two elements $G'_1$, $G'_2$ of
  $\SV (\Pred f)) (s_{\DN\,X}^\bullet (F))$, and $r \in [0, 1]$.  By
  definition, there are elements $G_1$, $G_2$ of
  $s_{\DN\,X}^\bullet (F)$ such that $\Pred f (G_1) \leq G'_1$ and
  $\Pred f (G_2) \leq G'_2$.  Since $s_{\DN\,X}^\bullet (F)$ is
  convex, $rG_1+ (1-r)G_2 \in s_{\DN\,X}^\bullet (F)$.  It is easy to
  see that
  $r G'_1 + (1-r) G'_2 \geq r \Pred f (G_1) + (1-r) \Pred g (G_2) =
  \Pred f (rG_1 + (1-r) G_2)$, so
  $r G'_1 + (1-r) G'_2 \in \SV (\Pred f) (s_{\DN\,X}^\bullet (F))$.
  Since $r_{\DN\,X}$ and $r_{\DN\,Y}$ are homeomorphisms (with domains
  $\SV^{cvx} (\Pred_\Nature^\bullet X)$, resp.\
  $\SV^{cvx} (\Pred_\Nature^\bullet Y)$), in order to show that
  $s_{\DN\,Y}^\bullet (\Pred f (F)) = \SV (\Pred f)
  (s_{\DN\,X}^\bullet (F))$, it suffices to show that
  $r_{\DN\,Y} (s_{\DN\,Y}^\bullet (\Pred f (F))) = r_{\DN\,Y} (\SV
  (\Pred f) (s_{\DN\,X}^\bullet (F)))$.  The left-hand side is equal
  to $\Pred f (F)$, and the right-hand side is equal to
  $\Pred f (r_{\DN\,X} (s_{\DN\,X}^\bullet (F))$ (by naturality of
  $r_\DN$), hence to $\Pred f (F)$.  \qed
\end{proof}

\begin{theorem}
  \label{thm:DN:projlim}
  Let $\bullet$ be nothing, ``$\leq 1$'' or ``$1$''.  The projective
  limit of a projective system
  ${(p_{ij} \colon X_j \to X_i)}_{i \sqsubseteq j \in I}$ of
  topological spaces is preserved by $\Pred_{\DN}^\bullet$ if and only
  if it is preserved by $\Val_\bullet$.  In particular, it is
  preserved under any of the three sets of conditions of
  Theorem~\ref{thm:V:projlim}.
\end{theorem}
\begin{proof}
  We start with the if direction.  Let $X, {(p_i)}_{i \in I}$ be the
  canonical projective limit of
  ${(p_{ij} \colon X_j \to X_i)}_{i \sqsubseteq j \in I}$.  If
  $\Val_\bullet X, {(\Val {p_i})}_{i \in I}$ is a projective limit of
  ${(\Val {p_{ij}} \colon \Val_\bullet {X_j} \to \Val_\bullet
    {X_i})}_{i \sqsubseteq j \in I}$, then
  $\Pred_{\Nature}^\bullet X, {(\Pred {p_i})}_{i \in I}$ is a
  projective limit of
  ${(\Pred {p_{ij}} \colon \Pred_{\Nature}^\bullet {X_j} \to
    \Pred_{\Nature}^\bullet {X_i})}_{i \sqsubseteq j \in I}$.  Indeed,
  we recall that there is a natural homeomorphism between
  $\Val_\bullet$ and $\Pred_\Nature^\bullet$.  The spaces
  $\Val_\bullet {X_i}$ are all sober (see Remark~\ref{rem:sat:sober}),
  hence we can use Theorem~\ref{thm:Q:projlim} and conclude that
  $\SV {(\Pred_{\Nature}^\bullet X)}, {(\SV {(\Pred {p_i})})}_{i \in
    I}$ is a projective limit of
  $(\SV {(\Pred {p_{ij}})} \colon \SV {(\Pred_{\Nature}^\bullet
    {X_j})} \to \SV {(\Pred_{\Nature}^\bullet {X_i})})_{i \sqsubseteq
    j \in I}$.  We now use Lemma~\ref{lemma:retract:limits} with
  $S \eqdef \Pred_{\DN}^\bullet$ and
  $T \eqdef \SV \Pred_\Nature^\bullet$, and the natural retraction
  $(r_\DN, s_\DN^\bullet)$---it is natural by
  Lemma~\ref{lemma:rDP:nat}.

  In the only if direction, if
  $\Pred_\DN^\bullet X, {(\Pred_\DN^\bullet {p_i})}_{i \in I}$ is a
  projective limit of the projective system
  $(\Pred {p_{ij}} \colon \Pred_{\DN}^\bullet {X_j} \to
  \Pred_{\DN}^\bullet {X_i})_{i \sqsubseteq j \in I}$, then we claim
  that the comparison map $\varphi \colon \Val_\bullet X \to Z$ is
  surjective, where $Z$ is the canonical projective limit of
  ${(\Val {p_{ij}} \colon \Val_\bullet {X_j} \to \Val_\bullet
    {X_i})}_{i \sqsubseteq j \in I}$.  This will be enough to show
  that $\varphi$ is a homeomorphism, using
  Proposition~\ref{prop:V:projlim}.  Because of the natural
  homeomorphism $\Val_\bullet \cong \Pred_\Nature^\bullet$, we reason
  with linear previsions instead of continuous valuations.  Let
  ${(G_i)}_{i \in I}$ be an element of $Z$.  By assumption, there is a
  (unique) superlinear prevision $F$ on $X$ such that
  $\Pred {p_i} (F) = G_i$ for every $i \in I$, namely such that
  $F (h_i \circ p_i) = G_i (h_i)$ for every $h_i \in \Lform {X_i}$,
  for every $i \in I$.  For every $h \in \Lform X$, we build $h_i$ as
  in Lemma~\ref{lemma:hi}; since $F$ is Scott-continuous, we obtain
  that
  $F (h) = \dsup_{i \in I} F (h_i \circ p_i) = \dsup_{i \in I} G_i
  (h_i)$.  Given any two maps $h, h' \in \Lform X$,
  $h+h' = \dsup_{i \in I} (h_i \circ p_i) + \dsup_{i \in I} (h'_i
  \circ p_i) = \dsup_{i \in I} (h_i + h'_i) \circ p_i$, so a similar
  argument shows that
  $F (h+h') = \dsup_{i \in I} F ((h_i+h'_i) \circ p_i) = \dsup_{i \in
    I} G_i (h_i+h'_i)$.  Since $G_i$ is linear, the latter is equal to
  $\dsup_{i \in I} (G_i (h_i) + G_i (h'_i)) = \dsup_{i \in I} G_i
  (h_i) + \dsup_{i \in I} G_i (h'_i) = F (h) + F (h')$.  Hence $F$ is
  sublinear.  Being in $\Pred_{\DN}^\bullet X$, it is superlinear,
  hence linear.  Hence $F$ is in $\Pred_\Nature^\bullet X$, and it was
  built so that $\Pred {p_i} (F) = G_i$ for every $i \in I$, so
  $\varphi (F) = {(G_i)}_{i \in I}$.  \qed
\end{proof}

Superlinear previsions form a model of mixed demonic non-deterministic
and probabilistic choice.  Another, earlier model, due to
\cite{Mislove:nondet:prob,Tix:PhD,TKP:nondet:prob,DBLP:journals/acta/McIverM01},
is the composition $\SV^{cvx} \Val_\bullet$.  We have already
mentioned the fact that $\SV^{cvx} \Val_\bullet X$ is homeomorphic to
$\Pred_\DN^\bullet X$ for every space $X$
\cite[Theorem~4.15]{JGL-mscs16}; naturality was overlooked there, and
is dealt with by Lemma~\ref{lemma:rDP:nat}.  Together with the natural
homeomorphism $\Val_\bullet \cong \Pred_\Nature^\bullet$, this allows
us to obtain the following.
\begin{corollary}
  \label{corl:QV:projlim}
  Let $\bullet$ be nothing, ``$\leq 1$'' or ``$1$''.  The projective
  limit of a projective system
  ${(p_{ij} \colon X_j \to X_i)}_{i \sqsubseteq j \in I}$ of
  topological spaces is preserved by $\SV^{cvx} \Val_\bullet$ if and
  only if it is preserved by $\Val_\bullet$.
\end{corollary}
We refer to Theorem~\ref{thm:V:projlim} to what conditions ensure that
such limit preservation results hold.

\section{Intermission: $\Val_\bullet$ preserves local compactness and
  proper maps, and projective limits that yield $\odot$-consonant spaces}
\label{sec:interm-pres-local}

Before we go on with the $\Pred_{\AN}$ sublinear prevision functor, we
need to prove a few theorems about the $\Val_\bullet$ functors: that
it preserves local compactness, and that it preserves proper maps.  We
also need to show that certain spaces known as $\odot$-consonant
(sober) spaces are preserved by projective limits with proper bonding
maps, and that $\omega$-projective limits of locally compact sober
spaces are $\odot$-consonant.

\subsection{On the preservation of local compactness by
  $\Val_\bullet$}
\label{sec:pres-local-comp}

We start with local compactness.  It is known that $\Val_{\leq 1}$
preserves various properties: stable compactness
\cite[Theorem~39]{AMJK:scs:prob}, being a continuous dcpo
\cite[Theorem~5.2]{Jones:proba}, being a quasi-continuous dcpo
\cite[Theorem~5.1]{JGL:qcontval}, for example.  Some of these
preservation theorems extend over to $\Val_1$ or to $\Val$, but a
conspicuously absent property in the list is local compactness.  We
address this now.

The proof relies on capacities, as studied in \cite{GL:duality}.  But
that paper considers integrals of lower semicontinuous maps from $X$
to $\Rp$ (not $\creal$), hence does not cover $\Lform X$.  Instead, we
will refer to \cite[Section~4]{JGL:qcontval}, where we can find some
of the following information; we will prove the rest.  For every
monotone map $\nu \colon \Open X \to \creal$, for every
$h \in \Lform X$, there is a \emph{Choquet integral}
$\int_{x \in X} h (x) \,d\nu$, defined as the indefinite Riemann
integral $\int_0^\infty \nu (h^{-1} (]t, \infty]))\,dt$.
\begin{lemma}
  \label{lemma:choq}
  The following properties hold.
  \begin{enumerate}
  \item The Choquet integral $\int_{x \in X} h (x) \,d\nu$ is linear
    in $\nu$, monotonic and even Scott-continuous in $\nu$.
  \item\label{choq:cont:h} If $\nu$ is Scott-continuous, then the
    Choquet integral is Scott-continuous in $h$.
  \item\label{choq:lin:h} If $\nu$ is a continuous valuation, then the
    Choquet integral is also linear in $h$.
  \item\label{choq:chi} For every $U \in \Open X$,
    $\int_{x \in X} \chi_U (x) \,d\nu = \nu (U)$.
  \item\label{choq:int} Given any continuous valuation $\nu^*$ on
    $\SV X$, there is a Scott-continuous map
    $\nu \colon \Open X \to \creal$ defined by
    $\nu (U) \eqdef \nu^* (\Box U)$ for every $U \in \Open X$.  Then,
    for every $h \in \Lform X$, the map
    $h^* \colon Q \mapsto \min_{x \in Q} h (x)$ is in $\Lform {\SV X}$
    and
    $\int_{x \in X} h (x) \,d\nu = \int_{Q \in \SV X} h^*
    (Q)\,d\nu^*$.
  \item\label{choq:ugame} For every compact saturated subset $Q$ of
    $X$, the \emph{unanimity game}
    $\ugame Q \colon \Open X \to \creal$, which maps every
    $U \in \Open X$ to $1$ if $Q \subseteq U$ and to $0$ otherwise, is
    a Scott-continuous map from $\Open X$ to $\creal$.
  \item\label{choq:F} Letting
    $\nu \eqdef \sum_{j=1}^m a_j \ugame {Q_j}$,
    where each $Q_j$ is compact saturated and $a_j \in \Rp$, the map
    $F \colon \Lform X \to \creal$ defined by:
    \begin{align*}
      F (h) &\eqdef \int_{x \in X} h (x) \,d\nu = \sum_{j=1}^m a_j
              \min_{x \in Q_j} h (x)
    \end{align*}
    for every $h \in \Lform X$ is a superlinear prevision.
\end{enumerate}
\end{lemma}
\begin{proof}
  1. The fact that the Choquet integral is linear in $\nu$, namely
  that it commutes with scalar products by non-negative real numbers
  and with addition of continuous valuations, follows from the
  linearity of indefinite Riemann integration.  It is also monotonic
  in $\nu$.  In order to show Scott-continuity, we consider a directed
  family ${(\nu_i)}_{i \in I}$, with (pointwise) supremum $\nu$, and
  we observe that
  $\int_{x \in X} h (x) \,d\nu = \int_0^\infty \dsup_{i \in I} \nu_i
  (h^{-1} (]t, \infty])) \,dt$.  The key is that the integrand
  $t \mapsto \nu_i (h^{-1} (]t, \infty]))$ is antitone (all antitone
  maps are Riemann-integrable), and that indefinite Riemann
  integration of antitone maps $f$ is Scott-continuous in $f$, see
  \cite[Lemma~4.2]{Tix:bewertung}.  Therefore
  $\int_{x \in X} h (x) \,d\nu = \dsup_{i \in I} \int_0^\infty \nu_i
  (h^{-1} (]t, \infty])) \,dt = \dsup_{i \in I} \int_{x \in X} h (x)
  \,d\nu_i$.

  2.  The proof works as Tix's original proof of the same result in
  the special case where $\nu$ is a continuous valuation
  \cite[Satz~4.4]{Tix:bewertung}, and also relies on
  \cite[Lemma~4.2]{Tix:bewertung}.  Explicitly, let
  ${(h_i)}_{i \in I}$ be a directed family in $\Lform X$, with
  (pointwise) supremum $h$.  For every $t \in \Rp$,
  $h^{-1} (]t, \infty]) = \{x \in X \mid \dsup_{i \in I} h_i (x) > t\}
  = \dcup_{i \in I} h_i^{-1} (]t, \infty])$.  Therefore
  $\int_{x \in X} h (x) \,d\nu = \int_0^\infty \nu (\dcup_{i \in I}
  h_i^{-1} (]t, \infty])) \,dt = \int_0^\infty \dsup_{i \in I} \nu
  (h_i^{-1} (]t, \infty])) \,dt = \dsup_{i \in I} \int_0^\infty \nu
  (h_i^{-1} (]t, \infty])) \,dt = \dsup_{i \in I} \int_{x \in X} h_i
  (x) \,d\nu$, using the Scott-continuity of $\nu$ and the
  Scott-continuity of indefinite Riemann integration of antitone maps.

  3. This is a result of Tix \cite[Satz~4.4]{Tix:bewertung}.

  4. $\int_{x \in X} \chi_U (x) \,d\nu = \int_0^\infty \nu
  (\chi_U^{-1} (]t, \infty])) \,dt = \int_0^1 \nu (U) \,dt +
  \int_1^\infty 0 \,dt = \nu (U)$.

  5.  This is as with \cite[Lemma~7.5]{GL:duality}.  The fact that
  $\nu^*$ is Scott-continuous follows from the fact that $\nu$ is, and
  that the $\Box$ operator is, too.  For the latter, observe that for
  every directed family ${(U_i)}_{i \in I}$ of open subsets of $X$,
  for every $Q \in \SV X$, $Q \in \Box {\dcup_{i \in I} U_i}$ if and
  only if $Q \subseteq \dcup_{i \in I} U_i$, which is equivalent to
  $Q \subseteq U_i$ (namely, $Q \in \Box {U_i}$) for some $i \in I$,
  because $Q$ is compact.

  For every $Q \in \SV X$, the minimum of $h (x)$ when $x$ ranges over
  $Q$ is reached, since $Q$ is compact and non-empty.  For every
  $t \in \Rp$, $Q \in {h^*}^{-1} (]t, \infty])$ if and only if
  $h^* (Q) > t$.  The latter certainly implies that $h (x) > t$ for
  every $x \in Q$, hence that $Q \in \Box {h^{-1} (]t, \infty])}$.
  Conversely, if $Q \in \Box {h^{-1} (]t, \infty])}$, then let us pick
  $x \in Q$ such that $h (x)$ is the least value reached by $h$ on
  $Q$; then $h^* (Q) = h (x) > t$, so
  $Q \in {h^*}^{-1} (]t, \infty])$.  Hence we have shown that
  ${h^*}^{-1} (]t, \infty]) = \Box {h^{-1} (]t, \infty])}$.  This
  implies that $h^*$ is in $\Lform {\SV X}$, in particular.

  Now
  $\int_{Q \in \SV X} h^* (Q) \,d\nu^* = \int_0^\infty \nu^*
  ({h^*}^{-1} (]t, \infty])) \,dt = \int_0^\infty \nu^* (\Box h^{-1}
  (]t, \infty]) \,dt = \int_0^\infty \nu (h^{-1} (]t, \infty]) \,dt =
  \int_{x \in X} h (x) \,d\nu$.

  6. Monotonicity is clear.  For every directed family
  ${(U_i)}_{i \in I}$ of open subsets of $X$,
  $\ugame Q (\dcup_{i \in I} U_i)=1$ if and only if
  $Q \in \Box {\dcup_{i \in I} U_i}$, which is equivalent to the
  existence of an $i \in I$ such that $Q \in \Box {U_i}$
  (equivalently, $\ugame Q (U_i)=1$), as we have seen at the beginning
  of the proof of item~5.


  7. That would be a consequence of \cite[Propositions~7.2,
  7.6]{GL:duality}, except for the fact that our functions $h$ take
  their values in $\creal$.  We verify that
  $\int_{x \in X} h (x) \,d\nu = \sum_{j=1}^m a_j \int_{x \in X} h (x)
  \,d\ugame {Q_i} = \sum_{j=1}^m a_j \min_{x \in Q_j} h (x)$: the
  first equality is by item~1, and the second one is because
  $\int_{x \in X} h (x) \,d\ugame {Q_i}$ is equal to
  $\int_0^\infty \ugame {Q_i} (h^{-1} (]t, \infty])) \,dt =
  \int_0^{\min_{x \in Q_i} h (x)} 1 \,dt + \int_{\min_{x \in Q_i} h
    (x)}^\infty 0 \,dt = {\min_{x \in Q_i} h (x)}$.

  It is easy to see that $F (h)$ is superlinear, because of the laws
  $\min_{x \in Q_i} ah(x) = a \min_{x \in Q_i} h (x)$ (for every
  $a \in \Rp$) and
  $\min_{x \in Q_i} (h (x) + h' (x)) \geq \min_{x \in Q_i} h (x) +
  \min_{x \in Q_i} h' (x)$.  Scott-continuity comes from the fact that
  $F (h) = \int_{x \in X} h (x)\,d\nu$, where
  $\nu \eqdef \sum_{j=1}^m a_j \ugame {Q_j}$, that $\nu$ is
  Scott-continuous (by item~\ref{choq:ugame}), and by using
  item~\ref{choq:cont:h}.  \qed
\end{proof}

As we will see, $\Val_\bullet$ does not just preserve local
compactness, it maps core-compact spaces to locally compact sober
spaces.  A space $X$ is \emph{core-compact} if and only if $\Open X$
is a continuous dcpo; every locally compact space is core-compact
\cite[Theorem 5.2.9]{JGL-topology}.  The connection between the two
notions can be made more precise as follows.  Every topological space
$X$ has a \emph{sobrification} $\Sober X$ (or $X^s$), which is the
free sober space over $X$ \cite[Theorem 8.2.44]{JGL-topology}; then
$X$ is core-compact if and only if $\Sober X$ is locally compact
\cite[Proposition~8.3.11]{JGL-topology}.  $\Sober X$ can be built as
the collection of irreducible closed subsets, with the topology whose
open sets (all of them, not just a base) are
$\diamond U \eqdef \{F \in \Sober X \mid F \cap U \neq \emptyset\}$,
$U \in \Open X$.  In particular,
$\diamond \colon U \mapsto \diamond U$ is an order-isomorphism between
$\Open X$ and $\Open {\Sober X}$.  This induces a homeomorphism
between $\Val_\bullet X$ and $\Val_\bullet {\Sober X}$.

\begin{theorem}
  \label{thm:V:loccomp}
  For every core-compact space $X$, $\Val X$ and $\Val_{\leq 1} X$ are
  locally compact and sober.  If $X$ is also compact, then $\Val_1 X$
  is locally compact sober.
\end{theorem}
\begin{proof}
  By Remark~\ref{rem:sat:sober}, all the spaces $\Val_\bullet X$ are
  sober.

  Replacing $X$ by $\Sober X$ if necessary, we may assume that $X$ is
  locally compact and sober.  Then the upper Vietoris topology on
  $\SV X$ coincides with the Scott topology on $\Smyth X$ (with the
  reverse inclusion ordering $\supseteq$), by Lemma~8.3.26 of
  \cite{JGL-topology}, and $\Smyth X$ itself is a continuous dcpo
  \cite[Proposition~8.3.25]{JGL-topology}.  A fundamental theorem due
  to Jones \cite[Theorem~5.2]{Jones:proba} states that for every
  continuous dcpo $P$, $\Val_{\leq 1} P$ is a continuous dcpo under
  the stochastic ordering, and that a basis is given by the simple
  valuations, namely those of the form
  $\sum_{i=1}^n a_i \delta_{x_i}$, where each $a_i$ is in $\Rp$ and
  $x_i \in P$.  A similar result holds for $\Val P$ \cite[Theorem
  IV-9.16]{GHKLMS:contlatt}, and for $\Val_1 P$ provided that $P$ is
  also pointed \cite[Corollary~3.3]{Edalat:int}.  We will apply those
  results to $P \eqdef \Smyth X$, and we notice that if $X$ is
  compact, then $P$ is pointed, as $X$ itself will be the least
  element of $P$ in that case.

  Let $\nu \in \Val_\bullet X$, and let $\mathcal U$ be any open
  neighborhood of $\nu$.  Then $\nu$ is in some finite intersection of
  subbasic open sets $\bigcap_{i=1}^n [U_i > r_i]$ that is included in
  $\mathcal U$, where each $U_i$ is open in $X$ and $r_i \in \Rp$.  We
  consider
  $\mu \eqdef \Val {\eta^\Smyth_X} (\nu) \in \Val_\bullet {\SV X}$.
  We recall that $\eta^\Smyth_X$ is the unit of the $\SV$ monad, and
  maps every point $x \in X$ to $\upc x \in \SV X$.  For every open
  subset $U$ of $X$,
  $\mu (\Box U) = \eta^\Smyth_X [\nu] (\Box U) = \nu
  ((\eta^\Smyth_X)^{-1} (\Box U)) = \nu (U)$.  It follows that $\mu$
  is in $\bigcap_{i=1}^n [\Box {U_i} > r_i]$.  The latter is open in
  the upper Vietoris topology on
  $\Val_\bullet {\SV X} = \Val_\bullet P$, hence in the Scott topology
  of the stochastic ordering.  Since $\Val_\bullet P$ is a continuous
  dcpo with a basis of simple valuations, there is a simple valuation
  $\xi^* \eqdef \sum_{j=1}^m a_j \delta_{Q_j}$ in $\Val_\bullet P$
  that is way below $\mu$ and in $\bigcap_{i=1}^n [\Box {U_i} > r_i]$.

  We build a superlinear prevision $F$ on $X$ by letting $F (h)$ be
  equal to $\sum_{j=1}^m a_j \min_{x \in Q_j} h (x)$ for every
  $h \in \Lform X$.  (See Lemma~\ref{lemma:choq}, item~\ref{choq:F}.)
  Equivalently, $F (h) = \int_{x \in X} h (x) \,d\xi$, where
  $\xi \eqdef \sum_{j=1}^m a_j \ugame {Q_j}$.  The notations $\xi$,
  $\xi^*$ are justified by the fact that for every $U \in \Open X$,
  $\xi (U) = \xi^* (\Box U)$.
  
  Then $s_\DN^\bullet (F)$ is a compact saturated subset of
  $\Pred_\Nature^\bullet X$, as we have seen in
  Section~\ref{sec:prev-powerc}.  Equating $\Pred_\Nature^\bullet X$
  with $\Val_\bullet X$, $s_\DN^\bullet (F)$ is the subset of those
  $\nu' \in \Val_\bullet X$ such that for every $h \in \Lform X$,
  $F (h) \leq \int_{x \in X} h (x) \,d\nu'$.  We claim that $\nu$ is
  in the interior of $s_\DN^\bullet (F)$, and that $s_\DN^\bullet (F)$
  is included in $\mathcal U$; this will end our proof.

  We start by showing that $s_\DN^\bullet (F) \subseteq \mathcal U$.
  Let $\nu'$ be any element of $s_\DN^\bullet (F)$.  In other words,
  for every $h \in \Lform X$,
  $\int_{x \in X} h (x) \,d\nu' \geq F (h) = \sum_{j=1}^m a_j \min_{x
    \in Q_j} h (x)$.  For each $i \in \{1, \cdots, n\}$, we apply the
  latter to $h \eqdef \chi_{U_i}$.  We realize that
  $\int_{x \in X} \chi_{U_i} (x) \,d\nu' = \nu' (U_i)$
  (Lemma~\ref{lemma:choq}, item~\ref{choq:chi}), and that
  $\min_{Q_j} \chi_{U_i}$ is equal to $1$ if $Q_j \subseteq U_i$, to
  $0$ otherwise, so
  $F (\chi_{U_i}) = \sum_{\substack{1\leq j\leq m\\Q_j \subseteq U_i}}
  a_j$.  Therefore
  $\nu' (U_i) \geq \sum_{\substack{1\leq j\leq m\\Q_j \subseteq U_i}}
  a_j$.  We recall that $\xi^* = \sum_{j=1}^m a_j \delta_{Q_j}$ is in
  $\bigcap_{i=1}^n [\Box {U_i} > r_i]$, so for every
  $i \in \{1, \cdots, n\}$, $\xi^* (\Box {U_i}) > r_i$, namely
  $\sum_{\substack{1\leq j\leq m\\Q_j \subseteq U_i}} a_j > r_i$.
  Hence $\nu' \in \bigcap_{i=1}^n [U_i > r_i] \subseteq \mathcal U$.

  Next, we verify that $\nu$ is in the interior of
  $s_\DN^\bullet (F)$.  We use the fact that $\xi^*$ is way below
  $\mu$, equivalently that $\mu$ is in the open set $\uuarrow \xi^*$.
  Since $\mu = \Val {\eta^\Smyth_X} (\nu)$, $\nu$ is in
  $(\Val {\eta^\Smyth_X})^{-1} (\uuarrow \xi^*)$, which is open since
  $\Val {\eta^\Smyth_X}$ is continuous.  It remains to show that
  $(\Val {\eta^\Smyth_X})^{-1} (\uuarrow \xi^*)$ is included in
  $s_\DN^\bullet (F)$.

  For every $\nu' \in (\Val {\eta^\Smyth_X})^{-1} (\uuarrow \xi^*)$,
  by definition $\xi^*$ is way below, in particular below
  $\mu' \eqdef \Val {\eta^\Smyth_X} (\nu')$.  The latter is such that
  $\mu' (\Box U) = \eta^\Smyth_X [\nu'] (\Box U) = \nu'
  ((\eta^\Smyth_X)^{-1} (\Box U)) = \nu' (U)$ for every
  $U \in \Open X$.  Hence we may write $\mu'$ as ${\nu'}^*$ and apply
  Lemma~\ref{lemma:choq}, item~\ref{choq:int}, so that
  $\int_{x \in X} h (x)\,d\nu' = \int_{Q \in \SV X} \min_{x \in Q} h
  (x) \,d{\nu'}^*$.  Since $\xi^* \leq \mu' = {\nu'}^*$, the latter is
  larger than or equal to
  $\int_{Q \in \SV X} \min_{x \in Q} h (x) \,d\xi^* = \int_{x \in X} h
  (x) \,d\xi = F (h)$.  (We use Lemma~\ref{lemma:choq},
  item~\ref{choq:int} on the pair $\xi$, $\xi^*$ for that.)  We have
  shown that $\int_{x \in X} h (x)\,d\nu' \geq F (h)$ for every
  $h \in \Lform X$, so $\nu' \in s_\DN^\bullet (F)$, as promised.
  \qed
\end{proof}

\subsection{Proper maps and quasi-adjoints}
\label{sec:proper-maps-quasi}

In order to see that $\Val$ preserves proper maps, we rely on the
following notion, a very close cousin of the quasi-retractions of
\cite[Section~4]{JGL:QRB}, which were used to characterize proper
surjective maps there.  We recall that
$\eta^\Smyth_X \colon X \to \SVz X$ is the unit of the $\SVz$ monad,
and that it maps every $x \in X$ to $\upc x$.  We also recall that the
specialization ordering on spaces of the form $\SVz X$ is
\emph{reverse} inclusion $\supseteq$.

\begin{definition}
  \label{defn:qadj}
  A \emph{quasi-adjoint} to a continuous map $r \colon X \to Y$ is a
  continuous map $\varsigma \colon Y \to \SVz X$ such that:
  \begin{enumerate}[label=(\alph*)]
  \item $\eta^\Smyth_Y \leq \SVz r \circ \varsigma$, namely
    $\upc y \supseteq \SVz r (\varsigma (y))$ for every
    $y \in Y$, and
  \item $\varsigma \circ r \leq \eta^\Smyth_X$, namely
    $x \in \varsigma (r (x))$ for every $x \in X$.
  \end{enumerate}
\end{definition}

\begin{lemma}
  \label{lemma:proper}
  For a continuous map $r \colon X \to Y$, the following are
  equivalent:
  \begin{enumerate}
  \item $r$ is a proper map;
  \item $\dc r [F]$ is closed for every closed subset $F$ of $X$ and
    $r^{-1} (\upc y)$ is compact for every $y \in Y$;
  \item $r$ has a quasi-adjoint.
  \end{enumerate}
  The quasi-adjoint $\varsigma$, if it exists, is uniquely determined
  by $\varsigma (y) = r^{-1} (\upc y)$ for every $y \in Y$.
\end{lemma}
\begin{proof}
%
  The equivalence between items~1 and~2 can be found in
  \citep[Lemma~VI-6.21]{GHKLMS:contlatt}.
  
  We make the following observation: $(*)$ for every map
  $\varsigma \colon Y \to \SVz X$ such that
  $\varsigma (y) = r^{-1} (\upc y)$ for every $y \in Y$, for every
  open subset $U$ of $X$, the complement of $\varsigma^{-1} (\Box U)$
  in $Y$ is equal to $\dc r [F]$, where $F$ is the complement of $U$.
  Indeed, for every $y \in Y$, $y \not\in \varsigma^{-1} (\Box U)$ if
  and only if $\varsigma (y) = r^{-1} (\upc y)$ is not included in
  $U$, if and only if there is an $x \in F$ such that $y \leq r (x)$.

  $3 \limp 2$.  Let $\varsigma$ be a quasi-adjoint of $r$.  We claim
  that $\varsigma (y) = r^{-1} (\upc y)$, which will also show the
  final uniqueness result.  For every $x \in \varsigma (y)$, $r (x)$
  is in $\upc r [\varsigma (y)] \subseteq \upc y$ (by (a)), so
  $y \leq r (x)$, namely $x \in r^{-1} (\upc y)$.  Conversely, for
  every $x \in r^{-1} (\upc y)$, we have $y \leq r (x)$.  Since
  $\varsigma$ is continuous hence monotonic,
  $\varsigma (y) \supseteq \varsigma (r (x)$.  By (b),
  $x \in \varsigma (r (x))$, so $x \in \varsigma (y)$.

  It follows that, since $\varsigma (y) \in \SVz X$ by
  assumption, $r^{-1} (\upc y)$ is compact.  For every closed subset
  $F$ of $X$, we consider its complement $U$.  Since $\varsigma$ is
  continuous, $\varsigma^{-1} (\Box U)$ is open.  But its complement
  is precisely $\dc r [F]$, by $(*)$, so $\dc r [F]$ is closed.

  $1 \limp 3$.  Let $\varsigma (y) \eqdef r^{-1} (\upc y)$ for every
  $y \in Y$.  This is compact saturated since $r$ is proper, hence
  perfect.  Hence $\varsigma$ is a map from $Y$ to $\SVz X$.

  We check that $\varsigma$ is continuous.  For every open subset $U$
  of $X$, $\varsigma^{-1} (\Box U)$ is the complement of $\dc r [F]$,
  where $F \eqdef X \diff U$, by $(*)$.  Since $r$ is proper,
  $\dc r [F]$ is closed, so $\varsigma^{-1} (\Box U)$ is open.

  Let us check (a).  For every $y \in Y$,
  $(\SVz r \circ \varsigma) (y) = \upc r [\varsigma (y)] = \upc r
  [r^{-1} (\upc y)]$.  Every element $y'$ of that set is such that
  $y' \geq r (x)$ for some $x \in X$ such that $r (x) \geq y$, so
  $y' \in \upc y$.

  Let us check (b).  For every $x \in X$, we need to check that
  $x \in \varsigma (r (x)) = r^{-1} (\upc r (x))$, or equivalently
  that $r (x) \geq r (x)$, which is obvious.  \qed
\end{proof}

\subsection{On the preservation of proper maps by $\Val$}
\label{sec:pres-prop-maps}

\begin{lemma}
  \label{lemma:Phi}
  For every topological space $X$, there is a continuous map
  $\Phi \colon \Val_\bullet {\SV X} \to \Pred_\DN^\bullet X$ defined
  by
  $\Phi (\mu) (h) \eqdef \int_{Q \in \SV X} \min_{x \in Q} h (x)
  \,d\mu$ for every $h \in \Lform X$.
\end{lemma}
\begin{proof}
  Given $\mu \in \Val_\bullet {\SV X}$, we may define
  $\nu^* \eqdef \mu$ and $\nu (U) \eqdef \nu^* (\Box U)$ for every
  $U \in \Open X$.  Then $\nu$ is Scott-continuous and
  $\Phi (\mu) (h) = \int_{x \in X} h (x) \,d\nu$ by
  Lemma~\ref{lemma:choq}, item~\ref{choq:int}, so $\Phi (\mu)$ is
  Scott-continuous in $h$ by Lemma~\ref{lemma:choq},
  item~\ref{choq:cont:h}.

  We claim that $\Phi (\mu)$ is positively homogeneous.  We write
  $h^* \colon \SV X \to \creal$ for the map
  $Q \mapsto \min_{x \in Q} h (x)$.  This is in $\Lform {\SV X}$, by
  Lemma~\ref{lemma:choq}, item~\ref{choq:int}.

  For every $a \in \Rp$, $(ah)^* = a h^*$.  Therefore
  $\Phi (\mu) (ah) = \int_{Q \in \SV X} (ah)^* (Q) \,d\mu = \int_{Q
    \in \SV X} a h^* (Q) \,d\mu = a \Phi (\mu) (h)$, by linearity of
  integration (Lemma~\ref{lemma:choq}, item~\ref{choq:lin:h}).

  We claim that $\Phi (\mu)$ is superlinear.  Let
  $h, h' \in \Lform X$.  Then
  $\Phi (\mu) (h + h') = \int_{Q \in \SV X} \min_{x \in Q} (h (x) +
  h' (x)) \,d\mu$.  For every $Q \in \SV X$,
  $\min_{x \in Q} (h (x) + h' (x)) \geq \min_{x \in Q} h (x) + \min_{x
    \in Q'} h' (x)$.  Since $\mu$ is a continuous valuation,
  integration with respect to $\mu$ is linear
  (Lemma~\ref{lemma:choq}, item~\ref{choq:lin:h}), and monotonic (as a consequence of
  Lemma~\ref{lemma:choq}, item~\ref{choq:cont:h}) in the integrated function, so
  $\Phi (\mu) (h+h') \geq \int_{Q \in \SV X} \min_{x \in Q} h (x)
  \,d\mu + \int_{Q \in \SV X} \min_{x \in Q} h' (x) \,d\mu = \Phi
  (\mu) (h) + \Phi (\mu) (h')$.

  Hence $\Phi (\mu)$ is a superlinear prevision.  We note that
  $\min_{x \in Q} (1+h(x)) = 1 + \min_{x \in Q} h (x)$, for every
  non-empty compact saturated subset of $X$ and for every
  $h \in \Lform X$.  If $\mu (X) \leq 1$, then for every
  $h \in \Lform X$,
  $\Phi (\mu) (\one+h) = \int_{Q \in \SV X} \min_{x \in Q} (1 + h
  (x)) \,d\mu = \int_{Q \in \SV X} 1 \,d\mu + \int_{Q \in \SV X}
  \min_{x \in Q} h (x) \,d\mu \leq 1 + \Phi (\mu) (h)$.
  Similarly, if $\mu (X)=1$, then
  $\Phi (\mu) (\one+h) = 1 + \Phi (\mu) (h)$.  Therefore $\Psi$ is
  a map from $\Val_\bullet {\SV X}$ to $\Pred_\DN^\bullet X$.

  It remains to show that $\Phi$ is continuous.  For every
  $h \in \Lform X$, for every $r \in \Rp$,
  $\Phi^{-1} ([h > r]) = [h^* > r]$, where
  $h^* \colon \SV X \to \creal$ is defined by
  $h^* (Q) \eqdef \min_{x \in Q} h (x)$, as above.  Note that $h^*$ is
  in $\Lform {\SV X}$, by Lemma~\ref{lemma:choq}, item~\ref{choq:int}.
  \qed
\end{proof}
\begin{corollary}
  \label{corl:Phi}
  For every topological space $X$, there is a continuous map from
  $\Val_\bullet {\SV X}$ to $\SV {\Val_\bullet X}$, which
  maps every $\mu \in \Val_\bullet {\SV X}$ to the collection of
  continuous valuations $\nu \in \Val_\bullet X$ such that
  $\nu (U) \geq \mu (\Box U)$ for every $U \in \Open X$.
\end{corollary}
\begin{proof}
  We equate $\Val_\bullet X$ with $\Pred_\Nature^\bullet X$.  Then the
  map $s_{\DN\,X}^\bullet \circ \Phi$ is continuous, and maps every
  $\mu \in \Val_\bullet {\SV X}$ to the collection
  $\{\nu \in \Val_\bullet X \mid \forall h \in \Lform X, \int_{Q \in
    \SV X} \min_{x \in Q} h (x) \,d\mu \leq \int_{x \in X} h (x)
  \,d\nu\}$.  Let $\xi \colon \Open X \to \creal$ be defined by
  $\xi (\Box U) \eqdef \mu (U)$ for every $U \in \Open X$, so that we
  may write $\mu$ as $\xi^*$, following the convention of
  Lemma~\ref{lemma:choq}, item~\ref{choq:int}.  By this item, for
  every $\nu \in \Val_\bullet X$,
  $\nu \in (s_{\DN\,X}^\bullet \circ \Phi) (\mu)$ if and only if
  $\int_{x \in X} h (x) \,d\xi \leq \int_{x \in X} h (x)\,d\nu$.  By
  taking $h \eqdef \chi_U$ for an arbitrary open subset $U$ of $X$,
  the latter implies $\xi (U) \leq \nu (U)$.  Conversely, if
  $\xi (U) \leq \nu (U)$ for every $U \in \Open X$,
  $\int_{x \in X} h (x) \,d\xi = \int_0^\infty \xi (h^{-1} (]t,
  \infty])) \,dt \leq \int_0^\infty \nu (h^{-1} (]t, \infty])) \,dt =
  \int_{x \in X} h (x)\,d\nu$.  Hence
  $(s_{\DN\,X}^\bullet \circ \Phi) (\mu)$ is exactly
  $\{\nu \in \Val_\bullet X \mid \forall U \in \Open X, \nu (U) \geq
  \mu (\Box U)\}$.  \qed
\end{proof}

In order to apply the theory of quasi-adjoints, we need to replace
$\SV$ by $\SVz$ in the result above.  We will do this by using the
following trick.

For any topological space $X$, let $X^\top$ be the space obtained from
$X$ by adding a fresh element $\top$, and whose non-empty open subsets
are the sets $U \cup \{\top\}$, $U \in \Open X$.  The specialization
preordering $\leq^\top$ of $X^\top$ is such that $x \leq^\top y$ if
and only if $y=\top$ or $x, y \in X$ and $x \leq y$ in $X$.

As we had said in Section~\ref{sec:smyth-hyperspace}, we will reserve
the notation $\Box U$ for $\{Q \in \SV X \mid Q \subseteq U\}$, and
use the notation $\Box_0 U$ for $\{Q \in \SVz X \mid Q \subseteq U\}$.
Hence $\Box_0 U = \Box U \cup \{\emptyset\}$, for every
$U \in \Open X$.

\begin{lemma}
  \label{lemma:Xtop}
  For every topological space $X$, the map
  $t \colon Q \mapsto Q \cup \{\top\}$ is a homeomorphism of $\SVz X$
  onto $\SV X$.
\end{lemma}
\begin{proof}
  For every compact saturated subset $Q$ of $X$, $Q \cup \{\top\}$ is
  certainly saturated and non-empty.  Any open cover of
  $Q \cup \{\top\}$ can be trimmed to one that does not contain the
  empty set, hence one of the form ${(U_i \cup \{\top\})}_{i \in I}$
  with each $U_i$ open in $X$; then ${(U_i)}_{i \in I}$ is an open
  cover of $Q$, from which we can extract a finite subcover.  This
  shows that $Q \cup \{\top\}$ is compact in $X^\top$.

  For every $U \in \Open X$,
  $t^{-1} (\Box_0 {(U \cup \{\top\})}) = \Box U$, so $t$ is full and
  continuous.  Since $\SVz X$ is $T_0$, $t$ is a topological
  embedding.

  It remains to show that $t$ is surjective.  Given any non-empty
  compact saturated subset $Q'$ of $X^\top$, $Q'$ must contain some
  point, which is below $\top$, so $Q'$ must also contain $\top$.  But
  $\{\top\}$ is open in $X^\top$, so its complement $X$ is closed in
  $X^\top$, and therefore $Q \eqdef Q' \cap X$ is compact.  Since it
  is included in the subspace $X$, $Q$ is compact in $X$, too.  It is
  clearly saturated in $X$, and $Q' = Q \cup \{\top\}$, so $t$ is
  surjective.  \qed
\end{proof}

\begin{lemma}
  \label{lemma:calF}
  Let $\bullet$ be nothing, ``$\leq 1$'' or ``$1$''.  For every
  topological space $X$, let $\mathcal F_X$ be the subset of
  $\Val_\bullet {(X^\top)}$ consisting of those elements $\nu$ such that
  $\nu (\{\top\}) = 0$.  $\mathcal F_X$ is a closed subspace of
  $\Val_\bullet {(X^\top)}$, and there is a continuous map
  $c \colon \SV {\Val_\bullet {(X^\top)}} \to \SVz {\mathcal F_X}$
  defined by $c (\mathcal Q) \eqdef \mathcal Q \cap \mathcal F_X$ for
  every $\mathcal Q \in \SV {\Val_\bullet {(X^\top)}}$.
\end{lemma}
\begin{proof}
  First, the definition of $\mathcal F_X$ makes sense, and notably the
  condition $\nu (\{\top\}) = 0$, because $\{\top\}$ is open in
  $X^\top$.  Second, $\mathcal F_X$ is the complement of
  $[\{\top\} > 0]$, hence is closed in $\Val_\bullet {(X^\top)}$.

  For every $\mathcal Q \in \SV {\Val_\bullet {(X^\top)}}$,
  $c (\mathcal Q)$ is compact in $\Val_\bullet {(X^\top)}$ and
  included in $\mathcal F_X$, hence compact in $\mathcal F_X$ seen as
  a subspace.  The specialization ordering of $\mathcal F_X$ is the
  restriction of the stochastic ordering, so $c (\mathcal Q)$ is
  saturated in $\mathcal F_X$: it suffices to show that for all
  $\nu \in \mathcal Q \cap \mathcal F_X$ and $\nu' \in \mathcal F_X$
  such that $\nu \leq \nu'$, $\nu'$ is in $\mathcal Q$, which follows
  from the fact that $\mathcal Q$ is saturated in
  $\Val_\bullet {(X^\top)}$.  Hence $c$ defines a map from
  $\SV {\Val_\bullet {(X^\top)}}$ to $\SVz {\mathcal F_X}$.  For every
  open subset $\mathcal U$ of $\Val_\bullet {(X^\top)}$,
  $c^{-1} (\Box_0 {(\mathcal U \cap \mathcal F_X)}) = \{\mathcal Q \in
  \SV {\Val_\bullet {(X^\top)}} \mid \mathcal Q \cap \mathcal F_X
  \subseteq \mathcal U\} = \Box {([\{\top\} > 0] \cup \mathcal U)}$;
  indeed, $\mathcal Q \cap \mathcal F_X \subseteq \mathcal U$ if and
  only if $\mathcal Q$ is included in the union of the complement
  $[\{\top\} > 0]$ of $\mathcal F_X$ with $\mathcal U$.  Hence $c$ is
  continuous.  \qed
\end{proof}

\begin{lemma}
  \label{lemma:nu-}
  Let $\bullet$ be nothing, ``$\leq 1$'' or ``$1$''.  For every
  topological space $X$, let $\mathcal F_X$ be as in
  Lemma~\ref{lemma:calF}.  For every $\nu \in \mathcal F_X$, there is
  a unique $\nu^- \in \Val_\bullet X$ such that $i [\nu^-] = \nu$,
  where $i$ is the inclusion map from $X$ into $X^\top$.  The map
  $\_^- \colon \nu \mapsto \nu^-$ is continuous from $\mathcal F_X$ to
  $\Val_\bullet X$.
\end{lemma}
\begin{proof}
  Let $\nu \in \mathcal F_X$.  If $\nu^-$ exists, then for every
  $U \in \Open X$, we must have
  $\nu (U \cup \{\top\}) = \nu^- (i^{-1} (U \cup \{\top\}) = \nu^-
  (U)$, showing uniqueness.  As far as existence is concerned, we
  define $\nu^- (U)$ as $\nu (U \cup \{\top\})$ for every
  $U \in \Open X$.  This is a strict map precisely because
  $\nu \in \mathcal F_X$, and it is clear that $\nu^-$ is modular and
  Scott-continuous.  Additionally, $\nu^- (X) \leq 1$ if and only if
  $\nu (X^\top) \leq 1$, and similarly with $=$ instead of $\leq$.
  
  This defines a map $\_^- \colon \nu \mapsto \nu^-$ from
  $\mathcal F_X$ to $\Val_\bullet X$, and it remains to see that it is
  continuous: the inverse image of a subbasic open set $[U > r]$, with
  $U \in \Open X$ and $r \in \Rp$, is $[U \cup \{\top\} > r]$.  \qed
\end{proof}

We recall that $\Box_0 U$ denotes $\Box U \cup \{\emptyset\}$, and is
a canonical subbasic open subset of $\SVz X$, where $U \in \Open X$.
\begin{proposition}
  \label{prop:Phi}
  For every topological space $X$, there is a continuous map from
  $\Val_\bullet {\SVz X}$ to $\SVz {\Val_\bullet X}$, which maps every
  $\mu \in \Val_\bullet {\SVz X}$ to the collection of continuous
  valuations $\nu \in \Val_\bullet X$ such that
  $\nu (U) \geq \mu (\Box_0 U)$
  for every $U \in \Open X$.
\end{proposition}
\begin{proof}
  Let us call $f$ the continuous map from
  $\Val_\bullet {\SV {(X^\top)}} \to \SV {\Val_\bullet {(X^\top)}}$
  that we obtain from Corollary~\ref{corl:Phi} applied to the space
  $X^\top$.  We form the composition:
  \[
    \xymatrix{\Val_\bullet {\SVz X} \ar[r]^{\Val_\bullet t} &
      \Val_\bullet {\SV {(X^\top)}} \ar[r]^f & \SV {\Val_\bullet
        {X^\top}} \ar[r]^c & \SVz {\mathcal F_X} \ar[r]^{\SVz {\_^-}} &
      \SVz {\Val_\bullet X}}.
  \]
  where $t$ is from Lemma~\ref{lemma:Xtop}, $c$ is from
  Lemma~\ref{lemma:calF}, and $\_^-$ is from Lemma~\ref{lemma:nu-}.
  This composition is continuous.

  Given any $\mu \in \Val_\bullet {\SVz X}$, let $Q$ be its image by
  that composition.  It remains to show that
  $Q = \{\nu \in \Val_\bullet X \mid \forall U \in \Open X, \nu (U)
  \geq \mu (\Box_0 U)\}$.  The image of $\mu$ by $\Val_\bullet t$ is
  the continuous valuation
  $\mathcal U \mapsto \mu (t^{-1} (\mathcal U)) = \mu (\{Q \in \SVz X
  \mid Q \cup \{\top\} \in \mathcal U\})$.  In particular, for every
  open subset $U$ of $X$,
  $\Val_\bullet t (\mu) (\Box {(U \cup \{\top\})}) = \mu (\{Q \in \SVz
  X \mid Q \cup \{\top\} \subseteq U \cup \{\top\}\}) = \mu (\Box_0
  U)$, while $\Val_\bullet t (\mu) (\emptyset) = 0$.  The image of
  $\Val_\bullet t (\mu)$ by $f$ is the collection of all
  $\nu \in \Val_\bullet {(X^\top)}$ such that for every
  $U \in \Open X$,
  $\nu (U \cup \{\top\}) \geq \Val_\bullet t (\mu) (\Box {(U \cup
    \{\top\})})$ (and
  $\nu (\emptyset) \geq \Val_\bullet t (\mu) (\emptyset)$, which is
  automatically true); in other words, the collection of all
  $\nu \in \Val_\bullet {(X^\top)}$ such that for every
  $U \in \Open X$, $\nu (U \cup \{\top\}) \geq \mu (\Box_0 U)$.  This
  is mapped by $c$ to the collection $\mathcal N$ of all
  $\nu \in \Val_\bullet {(X^\top)}$ satisfying the same condition and
  such that $\nu (\{\top\}) = 0$.  For any $\nu \in \mathcal N$, we
  have $\nu^- (U) = \nu (U \cup \{\top\}) \geq \mu (\Box_0 U)$ for
  every $U \in \Open X$.

  The set $Q$ is the upward closure of the collection of continuous
  valuations $\nu^-$ obtained this way.  In particular, for every
  $\nu' \in Q$, for every $U \in \Open X$, $\nu' (U)$ is larger than
  or equal to $\nu^- (U)$ for some $\nu \in \mathcal N$, and therefore
  $\nu' (U) \geq \mu (\Box_0 U)$.  Conversely, for every
  $\nu' \in \Val_\bullet X$ such that $\nu' (U) \geq \mu (\Box_0 U)$
  for every $U \in \Open X$, let $\nu \in \Val_\bullet {(X^\top)}$ be
  defined by $\nu (U \cup \{\top\}) \eqdef \nu' (U)$ for every
  $U \in \Open X$, and $\nu (\emptyset) \eqdef 0$.  It is easy to
  check that $\nu$ is indeed in $\Val_\bullet {(X^\top)}$, and that
  $\nu (\{\top\}) = 0$, so that $\nu \in \mathcal F_X$.  Additionally,
  $\nu^- = \nu'$.  Therefore $\nu'$ is in $Q$.  Hence $Q$ coincides
  with the collection
  $\{\nu' \in \Val_\bullet X \mid \forall U \in \Open X, \nu' (U) \geq
  \mu (\Box_0 U)\}$, as promised.  \qed
\end{proof}
The following somehow generalizes Theorem~6.5 of \cite{JGL:QRB}, which
states that $\Val_1$ preserves proper surjective maps between stably
compact spaces.  We do not deal with surjectivity.
\begin{theorem}
  \label{thm:V:proper}
  Let $\bullet$ be nothing, ``$\leq 1$'' or ``$1$''.
  For every proper map $r \colon X \to Y$, $\Val r \colon
  \Val_\bullet X \to \Val_\bullet Y$ is proper.
\end{theorem}
\begin{proof}
  By Lemma~\ref{lemma:proper}, $r$ has a quasi-adjoint
  $\varsigma \colon Y \to \SVz X$.  Let $\varsigma'$ be the
  composition
  $\xymatrix{\Val_\bullet Y \ar[r]^(0.4){\Val \varsigma} & \Val_\bullet
    {\SVz X} \ar[r]^g & \SVz {\Val_\bullet X}}$, where $g$ is the
  continuous map of Proposition~\ref{prop:Phi}.  We check that
  $\varsigma'$ is a quasi-adjoint to $\Val r$.


  First, we claim that
  $\eta^\Smyth_{\Val_\bullet Y} \leq \SVz {\Val r} \circ \varsigma'$,
  namely that for every $\nu \in \Val_\bullet Y$,
  $\upc \nu \supseteq \SVz {\Val r} (g (\Val \varsigma (\nu)))$.  For
  every
  $\mu' \in \SVz {\Val r} (g (\Val \varsigma (\nu))) = \upc \Val r [g
  (\varsigma [\nu])]$, there is a $\nu' \in g (\varsigma [\nu])$ such
  that $\mu' \geq r [\nu']$.  By definition of $g$, for every
  $U \in \Open X$,
  $\nu' (U) \geq \varsigma [\nu] (\Box_0 U) = \nu (\varsigma^{-1}
  (\Box_0 U))$.  Therefore, for every $V \in \Open Y$,
  $\mu' (V) \geq r [\nu'] (V) = \nu' (r^{-1} (V)) \geq \nu
  (\varsigma^{-1} (\Box_0 {r^{-1} (V)}))$.  We now observe that
  $V \subseteq \varsigma^{-1} (\Box_0 {r^{-1} (V)} )$: for every
  $y \in V$, $\upc y$ is included in $V$, and since
  $\upc y \supseteq \SVz r (\varsigma (y))$, we have
  $r [\varsigma (y)] \subseteq V$; hence
  $\varsigma (y) \in \Box_0 {r^{-1} (V)}$.  Since $\nu$ is monotonic, we
  conclude that $\mu' (V) \geq \nu (V)$.  Since $V$ is arbitrary in
  $\Open Y$, $\mu' \geq \nu$.  This shows that $\mu' \in \upc \nu$.
  Since $\mu'$ is arbitrary in
  $\SVz {\Val r} (g (\Val \varsigma (\nu)))$, we have shown that
  $\upc \nu \supseteq \SVz {\Val r} (g (\Val \varsigma (\nu)))$.

  Second, we claim that $\nu \in \varsigma' (\Val r (\nu))$ for every
  $\nu \in \Val_\bullet X$.  We have
  $\varsigma' (\Val r (\nu)) = g (\Val (\varsigma \circ r)) (\nu))$,
  and the claim reduces to showing that
  $\nu (U) \geq \Val (\varsigma \circ r) (\nu) (\Box_0 U)$ for every
  $U \in \Open X$.  We compute:
  $\Val (\varsigma \circ r) (\nu) (\Box_0 U) = (\varsigma \circ r)
  [\nu] (\Box_0 (U)) = \nu ((\varsigma \circ r)^{-1} (\Box_0 U))$.
  But $(\varsigma \circ r)^{-1} (\Box_0 U) \subseteq U$: for every
  $x \in (\varsigma \circ r)^{-1} (\Box_0 U)$,
  $(\varsigma \circ r) (x) \subseteq U$, and we conclude since
  $x \in \varsigma (r (x))$.  Since $\nu$ is monotonic,
  $\Val (\varsigma \circ r) (\nu) (\Box_0 U) \leq \nu (U)$, which is
  what we wanted to prove.

  We now know that $\varsigma'$ is a quasi-adjoint to $\Val r$, so
  $\Val r$ is proper by Lemma~\ref{lemma:proper}.  \qed
\end{proof}

\subsection{Projective systems consisting of proper maps}
\label{sec:proj-syst-cons}

\begin{proposition}
  \label{prop:hi:proper}
  Let ${(p_{ij} \colon X_j \to X_i)}_{i \sqsubseteq j \in I}$ be a
  projective system in $\Topcat$, with canonical projective limit
  $X, {(p_i)}_{i \in I}$.  Let us also assume that each $X_j$ is sober
  and that each $p_{ij}$ is proper.  Then:
  \begin{enumerate}
  \item every $p_i$ is proper; we write $\varsigma_i$ for its quasi-adjoint;
  \item for every $i \in I$, for every $U \in \Open X$, the largest
    open subset $U_i$ of $X_i$ such that $p_i^{-1} (U_i) \subseteq U$
    is $\varsigma_i^{-1} (\Box_0 U)$;
  \item for every $i \in I$, for every $h \in \Lform X$, the largest
    function $h_i \in \Lform {X_i}$ such that $h_i \circ p_i \leq h$
    is $h^\dagger \circ \varsigma_i$, where
    $h^\dagger \in \Lform {\SVz X}$ maps every $Q \in \SVz X$ to
    $\min_{x \in Q} h (x)$ if $Q \neq \emptyset$ and $\emptyset$ to
    $\infty$.
  \end{enumerate}
\end{proposition}
\begin{proof}
  We will need to know the following.  A \emph{well-filtered space}
  $Z$ is a topological space such that for every filtered family
  ${(Q_j)}_{j \in J}$ of compact saturated subsets, for every open
  subset $U$ of $Z$, if $\fcap_{j \in J} Q_j \subseteq U$ then
  $Q_j \subseteq U$ for some $j \in J$.  It follows that for every
  filtered family as above, $\fcap_{j \in J} Q_j$ is compact saturated
  \cite[Proposition~8.3.6]{JGL-topology}.  Every sober space is
  well-filtered \cite[Proposition 8.3.5]{JGL-topology}.
  
  1.  We fix $i \in I$.  Using Lemma~\ref{lemma:proper}, we will build
  a quasi-adjoint $\varsigma_i$ to $p_i$.  We know that, for every
  $y \in X_i$, $\varsigma_i (y)$ must be equal to $p_i^{-1} (\upc y)$,
  but we will define it differently, so as to make sure that it is
  compact saturated, and we will then check that it is equal to
  $p_i^{-1} (\upc y)$.

  Let $y \in X_i$.  For every $k \in I$ such that $i \sqsubseteq k$,
  $p_{ik}^{-1} (\upc y)$ is compact (and saturated) since $p_{ik}$ is
  proper.  For every $j \sqsubseteq k$, we let
  $Q_{jk} \eqdef \upc p_{jk} [p_{ik}^{-1} (\upc y)]$.  We claim that
  for all $j, k, k' \in I$ such that
  $i, j \sqsubseteq k \sqsubseteq k'$, $Q_{jk'} \subseteq Q_{jk}$.
  For every $x \in Q_{jk'}$, there is a point $x' \in X_{k'}$ such
  that $p_{jk'} (x') \leq x$ and $y \leq p_{ik'} (x')$.  In other
  words, $p_{jk} (p_{kk'} (x')) \leq x$ and
  $y \leq p_{ik} (p_{kk'} (x'))$, showing that
  $x \in \upc p_{jk} [p_{ik}^{-1} (\upc y)] = Q_{jk}$.  Hence the
  family ${(Q_{jk})}_{k \in \upc i \cap \upc j}$ is filtered.  (We
  write $\upc i$ for the collection of indices $k \in I$ such that
  $i \sqsubseteq k$, and similarly for $\upc j$.  Both $\upc i$ and
  $\upc j$, as well as their intersection, are cofinal in $I$, and in
  particular directed.)  Since $X_j$ is sober hence well-filtered,
  $Q_j \eqdef \fcap_{k \in \upc i \cap \upc j} Q_{jk}$ is therefore a
  compact saturated subset of $X_j$.

  We verify that for all $j \sqsubseteq j' \in I$, $p_{jj'}$ maps
  $Q_{j'}$ to $Q_j$.  It suffices to verify that it maps $Q_{j'k}$ to
  $Q_{jk}$ for every $k \in \upc i \cap \upc j'$.  For every
  $x \in Q_{j'k}$, by definition there is a point $x' \in X_{k}$ such
  that $p_{j'k} (x') \leq x$ and $y \leq p_{ik} (x')$.  Then
  $p_{jj'} (x) \geq p_{jj'} (p_{j'k} (x')) = p_{jk} (x')$, and
  $y \leq p_{ik} (x')$, so
  $p_{jj'} (x) \in \upc p_{jk} [p_{ik}^{-1} (\upc y)] = Q_{jk}$.

  Hence
  ${(p_{jk|Q_k} \colon Q_k \to Q_j))}_{j \sqsubseteq k \in \upc i}$ is
  a projective system of compact spaces, obtained from compact
  saturated subsets $Q_i$ of each $X_i$.  Each $X_i$ is sober, hence
  also every $Q_i$ is sober by Remark~\ref{rem:sat:sober}.  By
  Steenrod's theorem, its canonical projective limit $Q$ is compact.
  One can verify that $Q$ is in fact a compact saturated subset of $X$
  \cite[Lemma~4.3]{JGL:proj:class}.

  We claim that $Q = p_i^{-1} (\upc y)$.  By construction, $Q$ is the
  collection of tuples $\vec x \eqdef {(x_i)}_{i \in I}$ where each
  $x_j \in Q_j$ and for all $j \sqsubseteq k \in I$,
  $x_j = p_{jk} (x_j)$.  For each such tuple, $p_i (\vec x) = x_i$ is
  in $Q_i$, and
  $Q_i \subseteq Q_{ii} = \upc p_{ii} [p_{ii}^{-1} (\upc y)] = \upc
  y$.  Therefore $Q \subseteq p_i^{-1} (\upc y)$.  Conversely, let
  $\vec x \eqdef {(x_i)}_{i \in I}$ be any element of $X$ such that
  $p_i (\vec x) = x_i \in \upc y$.  We claim that $\vec x$ in $Q$,
  namely that for every $j \in I$, $x_j \in Q_j$.  In turn, we need to
  show that for every $k \in \upc i \cap \upc j$,
  $x_j \in \upc p_{jk} [p_{ik}^{-1} (\upc y)]$.  We simply observe
  that $x_j = p_{jk} (x_k)$ (hence in particular
  $x_j \geq p_{jk} (x_j)$) and $x_k \in p_{ik}^{-1} (\upc y)$, since
  $p_{ik} (x_k) = x_i \in \upc y$, by assumption.

  Using Lemma~\ref{lemma:proper}, let $\varsigma_{jk}$ be the
  quasi-adjoint of $p_{jk}$, for all $j \sqsubseteq k \in I$.  We know
  that $\varsigma_{jk} (x) = p_{jk}^{-1} (\upc x)$ for every
  $x \in X_j$.  Hence $Q_{jk}$, as defined above, is equal to
  $\SV {p_{jk}} (\varsigma_{ik} (y))$, and
  $Q_j = \fcap_{k \in \upc i \cap \upc j} \SV {p_{jk}} (\varsigma_{ik}
  (y))$.

  For every $y \in X_i$, let us define $\varsigma_i (y)$ as
  $p_i^{-1} (\upc y)$, namely as the intersection
  $X \cap \prod_{j \in I} \fcap_{k \in \upc i \cap \upc j} \SV
  {p_{jk}} (\varsigma_{ik} (y))$.  We claim that $\varsigma_i$ is
  continuous.  It suffices to show that the inverse image of a basic
  open set $\Box_0 {(p_j^{-1} (U))}$ ($j \in I$, $U \in \Open {X_j}$)
  of $X$ by $\varsigma_i$ is open in $X_i$.  The elements in that
  inverse image are the points $y \in X_i$ such that
  $\fcap_{k \in \upc i \cap \upc j} \SV {p_{jk}} (\varsigma_{ik} (y))
  \subseteq U$.  Since $X_j$ is sober hence well-filtered, the latter
  is equivalent to the existence of $k \in \upc i \cap \upc j$ such
  that $\SV {p_{jk}} (\varsigma_{ik} (y)) \subseteq U$.  But
  $\SV {p_{jk}} (\varsigma_{ik} (y)) \subseteq U$ is equivalent to
  $\SV {p_{jk}} (\varsigma_{ik} (y)) \in \Box_0 U$, which is
  equivalent to
  $y \in {(\SV {p_{jk}} \circ \varsigma_{ik})}^{-1} (\Box_0 U)$.
  Therefore
  $\varsigma_i^{-1} (\Box_0 {(p_j^{-1} (U))}) = \bigcup_{k \in \upc i
    \cap \upc j} {(\SV {p_{jk}} \circ \varsigma_{ik})}^{-1} (\Box_0
  U)$, which is an open set.

  Showing that $\varsigma_i$ is quasi-adjoint to $p_i$ is now a
  formality.
  For every $y \in X_i$, $\SVz {p_i} (\varsigma_i (y)) = \upc p_i
  [p_i^{-1} (\upc y)] \subseteq \upc y$, and for every $x \in X$,
  $\varsigma_i (p_i (x)) = p_i^{-1} [\upc p_i (x)]$ contains $x$.  By
  Lemma~\ref{lemma:proper}, it follows that $p_i$ is proper.

  2. Since $\varsigma_i$ is continuous, $\varsigma_i^{-1} (\Box_0 U)$
  is certainly open.  For every $y \in \varsigma_i^{-1} (\Box_0 U)$,
  $\varsigma_i (y) = p_i^{-1} (\upc y)$ is included in $U$.  Since
  open sets are upwards-closed,
  $\varsigma_i^{-1} (\Box_0 U) = \bigcup_{y \in \varsigma_i^{-1}
    (\Box_0 U)} \upc y$, so
  $p_i^{-1} (\varsigma_i^{-1} (\Box_0 U)) = \bigcup_{y \in
    \varsigma_i^{-1} (\Box_0 U)} p_i^{-1} (\upc y) \subseteq U$.

  Therefore $\varsigma_i^{-1} (\Box_0 U) \subseteq U_i$.  In the
  reverse direction, for every $y \in U_i$,
  $\varsigma_i (y) = p_i^{-1} (\upc y) \subseteq p_i^{-1} (U_i)
  \subseteq U$, so $\varsigma_i (y) \in \Box_0 U$.
  
  3.  We have already defined a very similar function we called $h^*$
  in Lemma~\ref{lemma:choq}, item~\ref{choq:int}.  But that $h^*$ had
  $\SV X$ as domain, while the domain of $h^\dagger$ is $\SVz h$.  We
  had shown that for every $t \in \Rp$,
  ${h^*}^{-1} (]t, \infty]) = \Box {h^{-1} (]t, \infty])}$.  It
  immediately follows that
  ${h^\dagger}^{-1} (]t, \infty]) = \Box_0 {h^{-1} (]t, \infty])}$, a
  basic open subset of $\SVz X$.  Hence $h^\dagger$ is lower
  semicontinuous, namely, in $\Lform {\SVz X}$.

  Let us write $g$ for $h^\dagger \circ \varsigma_i$, and $h_i$ for
  the largest function in $\Lform {X_i}$ such that
  $h_i \circ p_i \leq h$, as described in Lemma~\ref{lemma:hi}.  We
  have
  $g \circ p_i = h^\dagger \circ \varsigma_i \circ p_i \leq h^\dagger
  \circ \eta^\Smyth_X$ (by property~(b) of quasi-adjoints and the fact
  that continuous maps are monotonic), and
  $h^\dagger \circ \eta^\Smyth_X = h$, since for every $x \in X$,
  $h^\dagger (\eta^\Smyth_X (x)) = \min_{y \in \upc x} h (y) = h (x)$.
  Since $h_i$ is the largest element of $\Lform {X_i}$ such that
  $h_i \circ p_i \leq h$, $g \leq h_i$.  Conversely, the operation
  $\_^\dagger$ is monotonic, so
  $g = h^\dagger \circ \varsigma_i \geq (h_i \circ p_i)^\dagger \circ
  \varsigma_i$.  For every $Q \in \SVz X$, either $Q$ is empty and
  $(h_i \circ p_i)^\dagger (Q) = 0 = h_i^\dagger (\SV {p_i} (Q))$, or
  $Q$ is non-empty and
  $(h_i \circ p_i)^\dagger (Q) = \min_{x \in Q} h_i (p_i (x)) =
  \min_{x \in Q, y \geq p_i (x)} h_i (y) = h_i^\dagger (\upc p_i [Q])
  = h_i^\dagger (\SV {p_i} (Q))$.  Therefore
  $(h_i \circ p_i)^\dagger = h_i^\dagger \circ \SV {p_i}$, and hence
  $g \geq h_i^\dagger \circ \SV {p_i} \circ \varsigma_i \geq
  h_i^\dagger \circ \eta^\Smyth_X$ (by property~(a) of quasi-adjoints)
  $= h_i$.  Hence $g=h_i$.  \qed
\end{proof}

\subsection{Projective limits of consonant and $\odot$-consonant
  spaces}
\label{sec:proj-limits-cons}

In a topological space $X$, for every compact saturated subset $Q$,
the collection $\blacksquare Q$ of all open neighborhoods of $Q$ is a
Scott-open subset of $\Open X$.  Any union of such sets
$\blacksquare Q$ is Scott-open, and $X$ is called \emph{consonant} if
and only if the converse holds, namely: for every Scott-open subset
$\mathcal U$ of $\Open X$, for every $U \in \mathcal U$, there is a
compact saturated subset $Q$ of $X$ such that $Q \subseteq U$ and
$\blacksquare Q \subseteq \mathcal U$.  As we said in
Section~\ref{sec:cont-valu}, the notion arises from
\cite{DGL:consonant}; see also \cite[Exercise 5.4.12]{JGL-topology}.

\begin{proposition}
  \label{prop:consonant:lim}
  Let ${(p_{ij} \colon X_j \to X_i)}_{i \sqsubseteq j \in I}$ be a
  projective system of topological spaces, with canonical projective
  limit $X, {(p_i)}_{i \in I}$.  If every $p_{ij}$ is proper and if
  every $X_i$ is consonant and sober, then so is $X$.
\end{proposition}
\begin{proof}
  Let $\mathcal U$ be a Scott-open subset of $\Open X$.  For every
  $i \in I$, let $\mathcal U_i$ be the collection of open subsets $U$
  of $X_i$ such that $p_i^{-1} (U) \in \mathcal U$.  Since $p_i^{-1}$
  commutes with unions, $\mathcal U_i$ is Scott-open in $\Open {X_i}$.
  Now let $U \in \Open X$.  For each $i \in I$, let $U_i$ be the
  largest open subset of $X_i$ such that $p_i^{-1} (U_i) \subseteq U$.
  Then ${(p_i^{-1} (U_i))}_{i \in I, \sqsubseteq}$ is a monotone net
  of open subsets of $X$, whose union is $U$.  Since $\mathcal U$ is
  Scott-open, $p_i^{-1} (U_i) \in \mathcal U$ for some $i \in I$.  In
  other words, $U_i$ is in $\mathcal U_i$.
  
  We use the fact that $X_i$ is consonant: there is a compact
  saturated subset $Q_i$ of $X_i$ such that $Q_i \subseteq U_i$ and
  every open neighborhood of $Q_i$ is in $\mathcal U_i$.  Using
  Proposition~\ref{prop:hi:proper}, item~1, $Q \eqdef p_i^{-1} (Q_i)$ is
  compact saturated in $X$.

  We note that $Q \subseteq U$.  Indeed, for every $x \in Q$, $p_i (x)
  \in Q_i \subseteq U_i$, so $x \in p_i^{-1} (U_i) \subseteq U$.

  We claim that every open neighborhood $V$ of $Q$ in $X$ lies in
  $\mathcal U$.  Since $Q_i$ is upwards-closed it is equal to the
  union of the sets $\upc y$ when $y$ ranges over $Q_i$.  Then
  $p_i^{-1} (\upc y) = \varsigma_i (y)$, where $\varsigma_i$ is the
  quasi-adjoint of $p_i$, so
  $Q = \bigcup_{y \in Q_i} \varsigma_i (y)$.  The fact that
  $Q \subseteq V$ then means that for every $y \in Q_i$,
  $\varsigma_i (y) \in \Box_0 V$, hence that
  $Q_i \subseteq \varsigma_i^{-1} (\Box_0 V)$.  By definition of
  $Q_i$, $\varsigma_i^{-1} (\Box_0 V)$ is then in $\mathcal U_i$.  By
  Proposition~\ref{prop:hi:proper}, item~2, $\varsigma_i^{-1} (\Box_0 V)$
  is the largest open subset $V_i$ of $X_i$ such that
  $p_i^{-1} (V_i) \subseteq V$.  We have seen that
  $V_i \in \mathcal U_i$, so by definition of $\mathcal U_i$,
  $p_i^{-1} (V_i)$ is in $\mathcal U$.  Since $\mathcal U$ is
  upwards-closed, $V$ is in $\mathcal U$.  \qed
\end{proof}

For every topological space $X$, for every $n \in \nat$, let the
\emph{copower} $n \odot X$ be the topological sum (categorical
coproduct) of $n$ copies of $X$.  In other words, $n \odot X$ is the
collection of pairs $(k, x)$ with $1 \leq k \leq n$ and $x \in X$,
with topology generated by the sets $\{k\} \times U$, $U \in \Open X$.
A space $X$ is \emph{$\odot$-consonant} if and only if $n \odot X$ is
consonant for every $n \in \nat$
\cite[Definition~13.1]{dBGLJL:LCScomplete}.  For example, every
LCS-complete space is $\odot$-consonant
\cite[Lemma~13.2]{dBGLJL:LCScomplete}.  There is an $n \odot \_$
endofunctor on $\Topcat$: for every continuous map $f \colon X \to Y$,
$n \odot f$ maps every $(k, x)$ to $(k, f (x))$.

A category $\cati$ is \emph{connected} if and only if it has at least
one object, and every two objects are connected by a zig-zag of
morphisms.  A connected diagram in a category $\catc$ is a functor
from a small connected category $\cati$ to $\catc$.  It is clear that
every projective system is a connected diagram.  The following says
that the copower functor preserves connected limits in $\Topcat$.
\begin{lemma}
  \label{lemma:odot:cont}
  Let $X, {(p_i)}_{i \in |\cati|}$ be the canonical limit of a
  connected diagram $F \colon \cati \to \Topcat$.  For every
  $n \in \nat$, $n \cdot X, {(n \odot p_i)}_{i \in I}$ is a projective
  limit of $(n \odot \_) \circ F$.
\end{lemma}
\begin{proof}
  Let $X', {(q_i)}_{i \in |\cati|}$ be the canonical projective limit
  of $(n \odot \_) \circ F$.  The elements of $X'$ are the tuples
  ${((k, x_i))}_{i \in I}$ such that
  $(k, x_i) = (n \odot p_{ij}) (k, x_j)$ for all
  $i \sqsubseteq j \in I$.  Note that the first component $k$ must be
  the same at all positions $i \in I$, because $\cati$ is connected.
  A base of open subsets of $X'$ is given by the sets
  $p_i^{-1} (\{k\} \times U_i)$, where $i \in I$, $1\leq k\leq n$, and
  $U_i \in \Open {X_i}$.  The map
  $f \colon {((k, x_i))}_{i \in I} \mapsto (k, {(x_i)}_{i \in I})$ is
  bijective.  It is continuous and full because
  $f^{-1} (\{k\} \times p_i^{-1} (U_i)) = p_i^{-1} (\{k\} \times
  U_i)$.  Hence $f$ is a homeomorphism.  Additionally, since $p_i$ and
  $q_i$ are both projections onto coordinate $i$,
  $(n \odot p_i) \circ f = q_i$.  \qed
\end{proof}

\begin{corollary}
  \label{corl:consonant:lim}
  Let ${(p_{ij} \colon X_j \to X_i)}_{i \sqsubseteq j \in I}$ be a
  projective system of topological spaces, with canonical projective
  limit $X, {(p_i)}_{i \in I}$.  If every $p_{ij}$ is proper and if
  every $X_i$ is $\odot$-consonant and sober, then so is $X$.
\end{corollary}
\begin{proof}
  For every $n \in \nat$, $n \odot X, {(n \odot p_i)}_{i \in I}$ is a
  projective limit of
  $(n \odot p_{ij} \colon n \odot X_j \to n \odot X_i)_{i \sqsubseteq
    j \in I}$ by Lemma~\ref{lemma:odot:cont}.  By assumption each
  space $n \odot X_i$ is consonant.  It is sober because any coproduct
  of sober spaces taken in $\Topcat$ is sober \cite[Lemma
  8.4.2]{JGL-topology}.  By Proposition~\ref{prop:consonant:lim},
  $n \odot X$ must then be consonant.  We finally recall that any
  limit of sober spaces taken in $\Topcat$ is sober.  \qed
\end{proof}

\subsection{Projective limits of locally compact sober spaces are
  consonant}
\label{sec:proj-limits-locally}

An $\omega$-projective limit of locally compact sober spaces need not
be locally compact, even for compact, locally compact sober spaces
\cite[Proposition~3.4]{JGL:proj:class}.  We will show that, while
local compactness is lost, the projective limit remains
$\odot$-consonant.

\begin{proposition}
  \label{prop:consonant:lim:omega:lc}
  Let ${(p_{mn} \colon X_n \to X_m)}_{m \leq n \in \nat}$ be a
  projective system of topological spaces, with canonical projective
  limit $X, {(p_n)}_{n \in \nat}$.  If every $X_n$ is locally compact
  and sober, then $X$ is consonant.
\end{proposition}
\begin{proof}
  Let $\mathcal U$ be a Scott-open subset of $\Open X$.  For every
  $n \in \nat$, let $\mathcal U_n$ be the collection of open subsets
  $U$ of $X_n$ such that $p_n^{-1} (U) \in \mathcal U$.  Since
  $p_n^{-1}$ commutes with unions, $\mathcal U_n$ is Scott-open in
  $\Open {X_n}$.  Now let $U \in \Open X$.  For each $n \in \nat$, let
  $U_n$ be the largest open subset of $X_n$ such that
  $p_n^{-1} (U_n) \subseteq U$.  Then
  ${(p_n^{-1} (U_n))}_{n \in \nat, \leq}$ is a monotone net of open
  subsets of $X$, whose union is $U$.  Since $\mathcal U$ is
  Scott-open, $p_n^{-1} (U_n) \in \mathcal U$ for $n$ large enough.
  In other words, $U_n$ is in $\mathcal U_n$ pour $n$ large enough,
  say $n \geq n_0$.

  Since $X_{n_0}$ is locally compact, $U_{n_0}$ is a directed supremum
  of sets $\interior Q$ with $Q$ compact saturated included in
  $U_{n_0}$.  Since $\mathcal U_{n_0}$ is Scott-open, one such set,
  call it $\interior {Q_{n_0}}$, is in $\mathcal U_{n_0}$; $Q_{n_0}$
  is compact saturated and included in $U_{n_0}$.  Then
  $p_{n_0(n_0+1)}^{-1} (\interior {Q_{n_0}})$ is in
  $\mathcal U_{n_0+1}$.  Indeed, this means that
  $p_{n_0+1}^{-1} (p_{n_0(n_0+1)}^{-1} (\interior {Q_{n_0}})) \in
  \mathcal U$, equivalently, that
  $p_{n_0}^{-1} (\interior {Q_{n_0}}) \in \mathcal U$, namely that
  $\interior {Q_{n_0}} \in \mathcal U_{n_0}$.  Since
  $\mathcal U_{n_0+1}$ is Scott-open, there is a compact saturated
  subset $Q_{n_0+1}$ of $X_{n_0+1}$ included in $\interior {Q_{n_0}}$
  whose interior is in $\mathcal U_{n_0+1}$.  We proceed in the same
  way for $n=n_0+2, n_0+3, \cdots$, and we obtain compact saturated
  subsets $Q_n$ of $X_n$ for every $n \geq n_0$ such that $\interior
  {Q_n} \in \mathcal U_n$ and $Q_{n+1} \subseteq \interior {Q_n}$.
  We complete this by letting $Q_m \eqdef \upc p_{mn_0} [Q_{n_0}]$ for
  every $m < n_0$.

  We see each $Q_n$ as a subspace of $X_n$.  Since $X_n$ is sober, by
  Remark~\ref{rem:sat:sober}, $Q_n$ is sober.  By construction,
  ${(p_{mn|Q_n} \colon Q_n \to Q_m)}_{m \leq n \in \nat}$ is a
  projective system, where $p_{mn|Q_n}$ is the restriction of $p_{mn}$
  to $Q_n$, and it is a consequence of Steenrod's theorem that its
  canonical projective limit $Q$ (or rather,
  $Q, {(p_{n|Q})}_{n \in \nat}$) is compact saturated in $X$ (and that
  every compact saturated subset of $X$ is obtained this way, see
  Lemma~4.3 of \cite{JGL:proj:class}).

  We claim that $Q \subseteq U$.  For every
  $x \eqdef {(x_n)}_{n \in \nat}$ in $Q$, we have $x_n \in Q_n$ for
  every $n \in \nat$.  In particular,
  $p_{n_0} (x) = x_{n_0} \in Q_{n_0} \subseteq U_{n_0}$, and since
  $p_{n_0}^{-1} (U_{n_0}) \subseteq U$, $x \in U$.
  
  We verify that every open neighborhood $V$ of $Q$ in $X$ is in
  $\mathcal U$.  Writing $V_m$ for the largest open subset of $X$ such
  that $p_m^{-1} (V_m) \subseteq V$, we have
  $V = \dcup_{m \in \nat} p_m^{-1} (V_m)$.  Since $Q$ is compact, $Q$
  is included in $p_m^{-1} (V_m)$ for some $m \in \nat$.
  Equivalently, $p_m [Q] \subseteq V_m$, hence
  $\upc p_m [Q] \subseteq V_m$, since $V_m$ is upwards-closed.  By
  Lemma~\ref{lemma:steenrod:open}, there is an $n \geq m$ such that
  $\upc p_{mn} [Q_n] \subseteq V_m$, so $p_{mn} [Q_n] \subseteq V_m$.
  For every $n' \geq n$, we have
  $p_{mn'} [Q_{n'}] = p_{mn} [p_{nn'} [Q_{n'}]] \subseteq p_{mn} [Q_n]
  \subseteq V_m$, so $p_{mn} [Q_n] \subseteq V_m$ holds for $n$ large
  enough.  We pick one such that $p_{mn} [Q_n] \subseteq V_m$, namely
  such that $Q_n \subseteq p_{mn}^{-1} (V_m)$, and $n \geq n_0$.
  Since $n \geq n_0$, we know that $\interior {Q_n} \in \mathcal U_n$,
  so $p_{mn}^{-1} (V_m) \in \mathcal U_n$, since $\mathcal U_n$ is
  upwards-closed; therefore
  $p_n^{-1} (p_{mn}^{-1} (V_m)) = p_m^{-1} (V_m)$ is in $\mathcal U$.
  Since $p_m^{-1} (V_m) \subseteq V$ and $\mathcal U$ is
  upwards-closed, $V$ is in $\mathcal U$.
\end{proof}

\begin{theorem}
  \label{prop:oconsonant:lim:omega:lc}
  Let ${(p_{mn} \colon X_n \to X_m)}_{m \leq n \in \nat}$ be a
  projective system of topological spaces, with canonical projective
  limit $X, {(p_n)}_{n \in \nat}$.  If every $X_n$ is locally compact
  and sober, then $X$ is $\odot$-consonant and sober.
\end{theorem}
\begin{proof}
  Any limit of sober spaces, taken in $\Topcat$, is sober
  \cite[Theorem 8.4.13]{JGL-topology}.  In order to see that $X$ is
  $\odot$-consonant, we realize that for every $k \in \nat$,
  $k \odot X, {(k \odot p_n)}_{n \in \nat}$ is a projective limit of
  ${(k \odot p_{mn} \colon k \odot X_n \to k \odot X_m)}_{m \leq n \in
    \nat}$ by Lemma~\ref{lemma:odot:cont}.  Every finite copower of
  locally compact sober spaces is locally compact sober.  In fact, in
  $\Topcat$, every finite coproduct of locally compact spaces is
  locally compact (an easy exercise), and every coproduct of sober
  spaces is sober \cite[Lemma 8.4.2]{JGL-topology}.  We can now apply
  Theorem~\ref{prop:consonant:lim:omega:lc} and we obtain that $k
  \odot X$ is consonant.
\end{proof}

\section{Hoare powercones and sublinear previsions}
\label{sec:sublinear-previsions}

We might think of proceeding in a similar way with the $\Pred_{\AN}$
sublinear prevision functor as with the $\Pred_{\DN}$ superlinear
prevision functor, but we will not.  There is an analogue of the
$(r_\DN, s_\DN^\bullet)$ retraction, but it is only a natural
retraction on some subcategory $\catk$ of $\Topcat$ consisting of
$\AN_\bullet$-friendly spaces.  (We will define this notion below.)
Additionally, contrarily to $\SV$, the $\HV$ functor does not preserve
all projective limits of sober spaces.

Sublinear previsions form a model of mixed angelic non-deterministic
and probabilistic choice.  Another, earlier model, due to
\cite{Mislove:nondet:prob,Tix:PhD,TKP:nondet:prob,DBLP:journals/acta/McIverM01},
is the composition $\HV^{cvx} \Val_\bullet$, where
$\HV^{cvx} (\Val_\bullet X)$ is the subspace of $\HV (\Val_\bullet X)$
consisting of convex non-empty closed sets.  We start with the functor
$\HV^{cvx} \Val_\bullet$.  First, we verify that this is, indeed, a
functor.

We import the following from \cite{Keimel:topcones2}.  A \emph{cone}
is a set with a scalar multiplication operation, by scalars from
$\Rp$, and with an addition operation, satisfying the expected laws.
A \emph{semitopological cone} is a cone with a topology that makes
both scalar multiplication and addition separately continuous, where
$\Rp$ is given the Scott topology.  For example, $\Lform X$, $\Val X$,
$\Pred_\DN X$, $\Pred_\AN X$ are semitopological cones, and
$\Val_\bullet X$, $\Pred_\DN^\bullet X$, $\Pred_\AN^\bullet X$ are
convex subspaces of the latter three.    We
will need the following fact.  In a semitopological cone, the closure of
a convex subsets is convex \cite[Lemma~4.10~(a)]{Keimel:topcones2},
and we obtain the following as an easy consequence.
\begin{fact}
  \label{fact:cl:conv}
  Given any convex subspace $Z$ of a semitopological cone, the closure
  of any convex subset of $Z$ in $Z$ is convex.
\end{fact}

\begin{lemma}
  \label{lemma:V:lin}
  Let $\bullet$ be nothing, ``$\leq 1$'' or ``$1$''.   The
  $\Val_\bullet$ functor preserves convex combinations, namely: for
  every continuous map $f \colon X \to Y$, for every $n \geq 1$, for
  all non-negative real numbers $a_1$, \ldots, $a_n$ summing up to
  $1$, for all $\nu_1, \cdots, \nu_n \in \Val_\bullet X$,
  $\Val_\bullet f (\sum_{i=1}^n a_i \cdot \nu_i) = \sum_{i=1}^n a_i
  \cdot \Val_\bullet f (\nu_i)$.
\end{lemma}
\begin{proof}
  Both sides map every open subset $V$ of $Y$ to $\sum_{i=1}^n a_i
  \nu_i (f^{-1} (V))$.  \qed
\end{proof}

\begin{lemma}
  \label{lemma:Hcvx:functor}
  Let $\bullet$ be nothing, ``$\leq 1$'' or ``$1$''.  The composition
  $\HV^{cvx} \Val_\bullet$ is a functor from $\Topcat$ to $\Topcat$,
  whose action on morphisms is the restriction of $\HV \Val_\bullet$.
\end{lemma}
\begin{proof}
  Using Lemma~\ref{lemma:V:lin}, for every
  $C \in \HV^{cvx} \Val_\bullet X$, $\Val_\bullet f [C]$ is convex,
  and therefore so is its closure $\HV \Val_\bullet f (C)$, by
  Fact~\ref{fact:cl:conv}.  Hence $\HV \Val_\bullet f$ maps elements
  $\HV^{cvx} \Val_\bullet X$ to elements of
  $\HV^{cvx} \Val_\bullet Y$, and we define $\HV^{cvx} \Val_\bullet f$
  as the corresponding restriction of $\HV \Val_\bullet f$.  This is a
  continuous map, and the fact that $\HV^{cvx} \Val_\bullet$ defines a
  functor follows from the fact that $\HV \Val_\bullet$ is a functor.  \qed
\end{proof}

\begin{proposition}
  \label{prop:Hcvx:projlim}
  Let ${(p_{ij} \colon X_j \to X_i)}_{i \sqsubseteq j \in I}$ be a
  projective system of topological spaces, with canonical projective
  limit $X, {(p_i)}_{i \in I}$.  Let $\bullet$ be nothing,
  ``$\leq 1$'' or ``$1$''.

  Then
  ${(\HV^{cvx} \Val_\bullet {p_{ij}} \colon \HV^{cvx} \Val_\bullet
    {X_j} \to \HV^{cvx} \Val_\bullet {X_i})}_{i \sqsubseteq j \in I}$
  is a projective system of topological spaces, and
  $\HV^{cvx} \Val_\bullet X, {(\HV^{cvx} \Val_\bullet {p_i})}_{i \in
    I}$ is a projective limit of it provided that
  $\HV \Val_\bullet X, {(\HV \Val_\bullet {p_i})}_{i \in I}$ is the
  projective limit of the projective system
  ${(\HV \Val_\bullet {p_{ij}} \colon \HV \Val_\bullet {X_j} \to \HV
    \Val_\bullet {X_i})}_{i \sqsubseteq j \in I}$, up to
  homeomorphism.
\end{proposition}
\begin{proof}
  The fact that
  ${(\HV^{cvx} \Val_\bullet {p_{ij}} \colon \HV^{cvx} \Val_\bullet
    {X_j} \to \HV^{cvx} \Val_\bullet {X_i})}_{i \sqsubseteq j \in I}$
  is a projective system of topological spaces, and that
  $\HV^{cvx} \Val_\bullet X, {(\HV^{cvx} \Val_\bullet {p_i})}_{i \in
    I}$ is a cone on that system, follows from the fact that
  $\HV^{cvx} \Val_\bullet$ is a functor
  (Lemma~\ref{lemma:Hcvx:functor}).

  Let $Z, {(q_i)}_{i \in I}$ be the canonical limit of
  ${(\HV \Val_\bullet {p_{ij}} \colon \HV \Val_\bullet {X_j} \to \HV
    \Val_\bullet {X_i})}_{i \sqsubseteq j \in I}$.  We remember that
  $Z$ is a space of $I$-indexed tuples, and that $q_i$ is projection
  onto coordinate $i$.  The canonical limit of
  ${(\HV^{cvx} \Val_\bullet {p_{ij}} \colon \HV^{cvx} \Val_\bullet
    {X_j} \to \HV^{cvx} \Val_\bullet {X_i})}_{i \sqsubseteq j \in I}$
  is $Z', {(q'_i)}_{i \in I}$ where
  $Z' \eqdef \{{(C_i)}_{i \in I} \in Z \mid C_i \text{ is convex for
    every }i \in I\}$, and $q'_i$ is the restriction of $q_i$ to $Z'$.
  By assumption, there is homeomorphism
  $f \colon \HV \Val_\bullet X \to Z$, defined by
  $f (C) \eqdef {(\HV \Val_\bullet p_i (C))}_{i \in I}$ for every
  $C \in \HV \Val_\bullet X$.  By Lemma~\ref{lemma:Hcvx:functor}, this
  restricts to a continuous map
  $f' \colon \HV^{cvx} \Val_\bullet X \to Z'$.  Since $f$ is full, so
  is $f'$: every open subset of $\HV^{cvx} \Val_\bullet X$ can be
  written as $\mathcal U \cap \HV^{cvx} \Val_\bullet X$ for some open
  subset $\mathcal U$ of $\HV \Val_\bullet X$; since $f$ is full,
  $\mathcal U = f^{-1} (\mathcal V)$ for some open subset $\mathcal V$
  of $Z$, and therefore
  $\mathcal U \cap \HV^{cvx} \Val_\bullet X = {f'}^{-1} (\mathcal V
  \cap Z')$.  Since $\HV^{cvx} \Val_\bullet X$ is $T_0$ (its
  specialization preordering is inherited from its superspace
  $\HV \Val_\bullet X$, and is therefore the inclusion ordering), $f'$
  is a topological embedding.

  It remains to show that $f'$ is surjective.  Let ${(C_i)}_{i \in I}$
  be any element of $Z'$; in particular, remember that $C_i$ is closed
  and convex.  Since $f$ is bijective, there is a unique non-empty
  closed subset $C$ of $\Val_\bullet X$ such that
  $C_i = \HV \Val_\bullet p_i (C)$ for every $i \in I$.  We claim that
  $C$ is convex.  In order to see this, we form the closure $C'$ of
  the convex hull $\conv C$ of $C$; the \emph{convex hull} $\conv C$
  is the smallest convex set containing $C$, and consists of the sums
  $\sum_{i=1}^n a_i \cdot x_i$ where $n \geq 1$, the numbers $a_i$ are
  non-negative and sum up to $1$, and each $x_i$ is in $C$.  $C'$ is
  closed and convex by Fact~\ref{fact:cl:conv}.  We will show that
  $C_i = \HV \Val_\bullet p_i (C')$ for every $i \in I$.  Then, by
  uniqueness of $C$, it will follows that $C = C'$, so that $C$ will
  indeed be convex.  Let us fix $i \in I$.  Since $C \subseteq C'$,
  $C_i = \HV \Val_\bullet p_i (C) \subseteq \HV \Val_\bullet p_i
  (C')$.  In the reverse direction,
  $\HV \Val_\bullet p_i (C') = cl (\Val_\bullet p_i [cl (\conv C)])
  \subseteq cl (\Val_\bullet p_i [\conv C])$ (since, for any
  continuous map $f$, and for every set $A$,
  $f [cl (A)] \subseteq cl (f [A])$)
  $\subseteq cl (\conv (\Val_\bullet p_i [C]))$ (by our explicit
  characterization of convex hulls and Lemma~\ref{lemma:V:lin})
  $\subseteq cl (\conv C_i) = C_i$, where the last equality is because
  $C_i$ is closed and convex.

  Now $f'$ is a surjective topological embedding, hence a
  homeomorphism.  Additionally,
  $q'_i \circ f' = \HV^{cvx} \Val_\bullet {p_i}$ for every $i \in I$,
  since $q_i \circ f = \HV \Val_\bullet {p_i}$.  \qed
\end{proof}

\begin{theorem}
  \label{thm:Hcvx:projlim}
  Let ${(p_{ij} \colon X_j \to X_i)}_{i \sqsubseteq j \in I}$ be a
  projective system of topological spaces, with canonical projective
  limit $X, {(p_i)}_{i \in I}$.  Let $\bullet$ be nothing,
  ``$\leq 1$'' or ``$1$''.
  If:
  \begin{enumerate}
  \item the projective system is an ep-system,
  \item or $I$ has a countable cofinal subset and each $X_i$ is
    locally compact sober (and compact if $\bullet$ is ``$1$''),
  \item or every $X_i$ is consonant sober and every $p_{ij}$ is a
    proper map,
  \end{enumerate}
  then
  ${(\HV^{cvx} \Val_\bullet {p_{ij}} \colon \HV^{cvx} \Val_\bullet
    {X_j} \to \HV^{cvx} \Val_\bullet {X_i})}_{i \sqsubseteq j \in I}$
  is a projective system of topological spaces, and
  $\HV^{cvx} \Val_\bullet X, {(\HV^{cvx} \Val_\bullet {p_i})}_{i \in
    I}$ is its projective limit, up to homeomorphism.
\end{theorem}
Case~3 in particular applies when every $X_i$ is LCS-complete;
LCS-complete spaces are even $\odot$-consonant
\cite[Lemma~13.2]{dBGLJL:LCScomplete}, and they are sober
\cite[Proposition~7.1]{dBGLJL:LCScomplete}.

\begin{proof}
  In all cases,
  ${(\Val_\bullet {p_{ij}} \colon \Val_\bullet {X_j} \to \Val_\bullet
    {X_i})}_{i \sqsubseteq j \in I}$ is a projective system of
  topological spaces, and
  $\Val_\bullet X, {(\Val_\bullet {p_i})}_{i \in I}$ is its projective
  limit, up to homeomorphism, by Theorem~\ref{thm:V:projlim}.

  We claim that
  ${(\HV \Val_\bullet {p_{ij}} \colon \HV \Val_\bullet {X_j} \to \HV
    \Val_\bullet {X_i})}_{i \sqsubseteq j \in I}$ is a projective
  system of topological spaces, and
  $\HV \Val_\bullet X, {(\HV \Val_\bullet {p_i})}_{i \in I}$ is its
  projective limit, up to homeomorphism.  This will allow us to
  conclude by Proposition~\ref{prop:Hcvx:projlim}.  In order to show
  the claim, we rely on Theorem~\ref{thm:H:projlim}; let us
  check its assumptions.

  In case~1,
  ${(\Val_\bullet {p_{ij}} \colon \Val_\bullet {X_j} \to \Val_\bullet
    {X_i})}_{i \sqsubseteq j \in I}$ is an ep-system.  Indeed, the
  image of an ep-system by any monotonic functor is an ep-system.
  Therefore case~1 of Theorem~\ref{thm:H:projlim} applies.  In case~2,
  every space $\Val_\bullet {X_i}$ is locally compact and sober by
  Theorem~\ref{thm:V:loccomp} (this is why we require $X_i$ to be
  compact when $\bullet$ is ``$1$''), so case~3 of
  Theorem~\ref{thm:H:projlim} applies.  In case~3, every space
  $\Val_\bullet {X_i}$ is sober (see Remark~\ref{rem:sat:sober}), and
  every map $\Val_\bullet {p_{ij}}$ is proper, by
  Theorem~\ref{thm:V:proper}, so case~2 of Theorem~\ref{thm:H:projlim}
  applies.  \qed
\end{proof}

In a semitopological cone, scalar multiplication is always jointly
continuous, but addition may fail to be.  Let us introduce more
material from \cite{Keimel:topcones2}.  A \emph{topological cone} is
one where addition is jointly continuous.  A semitopological cone $C$
is \emph{locally convex} if and only if for every $x \in C$, every
open neighborhood of $x$ contains a convex open neighborhood of $x$.
It is \emph{locally convex-compact} if and only if for every
$x \in C$, every open neighborhood of $x$ contains a convex compact
saturated neighborhood of $x$.

A space $X$ is \emph{$\AN_\bullet$-friendly}
\cite[Definition~1]{JGL:mscs16:errata}
if and only if:
\begin{itemize}[label=---]
\item $\bullet$ is nothing or ``$\leq 1$'', and $\Lform X$ is locally
  convex;
\item or $\bullet$ is ``$1$'', and either:
  \begin{enumerate}
  \item $\Lform X$ is locally convex and $X$ is compact;
  \item or $\Lform X$ is a locally convex, locally convex-compact,
    sober topological cone;
  \item or $X$ is LCS-complete.
  \end{enumerate}
\end{itemize}
We recall that $\Lform X$ is equipped with its Scott topology.

Every core-compact space is $\AN_\bullet$-friendly, for any value of
$\bullet$ \cite[Remark~2]{JGL:mscs16:errata}.
Hence, in particular, every locally compact space is
$\AN_\bullet$-friendly.  Every LCS-complete space is
$\AN_\bullet$-friendly for any value of $\bullet$, and
$\AN_1$-friendliness implies $\AN$-friendliness
\cite[Remark~3]{JGL:mscs16:errata}.
Also, every LCS-complete space is
$\odot$-consonant \cite[Lemma~13.2]{dBGLJL:LCScomplete}, and
$\Lform X$ is locally convex for every $\odot$-consonant space $X$.
We summarize all this as follows.
\begin{fact}
  \label{fact:assum}
  Every core-compact space and in particular every locally compact
  space, every LCS-complete space and in particular every
  $\odot$-consonant space is $\AN$-friendly (and
  $\AN_{\leq 1}$-friendly).  Every core-compact space, every locally
  compact space, every LCS-complete space, every compact
  $\odot$-consonant space is $\AN_1$-friendly.
\end{fact}

We turn to sublinear previsions.  From
\cite[Proposition~3.11]{JGL-mscs16} (and its errata
\cite{JGL:mscs16:errata}), there is a map
$r_{\AN\,X} \colon \HV (\Pred_\Nature^\bullet X) \to \Pred_\AN^\bullet
X$ and a map $s_{\AN\,X}^\bullet$ in the other direction, defined by
$r_{\AN\,X} (C) (h) \eqdef \sup_{G \in C} G (h)$ for every
$h \in \Lform X$ and
$s_{\AN\,X}^\bullet (F) \eqdef \{G \in \Pred_\Nature^\bullet X \mid G
\leq F\}$, and they form a retraction under the assumption that $X$ is
$\AN_\bullet$-friendly.

For any $\AN_\bullet$-friendly space $X$, $r_{\AN\,X}$ restricts to a
homeomorphism, with inverse $s_{\AN\,X}^\bullet$, between the subspace
$\HV^{cvx} (\Pred_\Nature^\bullet X) \to \Pred_\AN^\bullet X$ of
non-empty closed convex subsets of $\Pred_\Nature^\bullet X$ and
$\Pred_\AN^\bullet X$, see Theorem~4.11 of \cite{JGL-mscs16} and its
errata \cite{JGL:mscs16:errata}.

We write $r_\AN$ for the transformation consisting of all the maps
$r_{\AN\,X}$, when $X$ varies, and similarly with $s_\AN^\bullet$.

\begin{lemma}
  \label{lemma:rAP:nat}
  Let $\bullet$ be nothing, ``$\leq 1$'', or ``$1$''.  The
  transformations $r_\AN$ and $s_\AN^\bullet$ restrict to natural
  transformations between $\HV \Val_\bullet$ and $\Pred^\bullet_\AN$
  (resp., natural isomorphisms between $\HV^{cvx} \Val_\bullet$ and
  $\Pred^\bullet_\AN$) on the full subcategory of $\Topcat$ consisting
  of $\AN_\bullet$-friendly spaces.
\end{lemma}
\begin{proof}
  We will need to use the following observation: $(*)$ for any lower
  semicontinuous map $\psi \colon Z \to \creal$, where $Z$ is any
  topological space, for every $A \subseteq Z$,
  $\sup_{z \in A} \psi (z) = \sup_{z \in cl (A)} \psi (z)$.  Indeed,
  for every $t \in \real$, $t < \sup_{z \in A} \psi (z)$ if and only
  if $\psi^{-1} (]t, \infty])$ intersects $A$,
  $t < \sup_{z \in A} \psi (z)$ if and only if
  $\psi^{-1} (]t, \infty])$ intersects $cl (A)$, and those are
  equivalent conditions since $\psi^{-1} (]t, \infty])$ is open.

  
  Let $f \colon X \to Y$ be any continuous map, where both $\Lform X$
  and $\Lform Y$ are locally convex.  Let us start with $r_\AN$.  We
  need to show that for every $C \in \HV (\Pred_\Nature^\bullet X)$,
  for every $h \in \Lform Y$,
  $r_{\AN\,Y} (\HV (\Pred f) (C)) (h) = \Pred f (r_{\AN\,X} (C)) (h)$.
  The left-hand side is equal to
  $\sup_{G' \in \HV (\Pred f) (Q)} G' (h) = \sup_{G' \in cl (\{\Pred f
    (G) \mid G \in C\})} G' (h) = \sup_{G' \in \{\Pred f (G) \mid G
    \in C\}} G' (h)$ (by $(*)$, since $G' \mapsto G' (h)$ is lower
  semicontinuous, by definition of the weak topology)
  $= \sup_{G \in C} \Pred f (G) (h) = \sup_{G \in C} G (h \circ f) =
  r_{\AN\,X} (C) (h \circ f) = \Pred f (r_{\AN\,X} (C)) (h)$.

  As far as $s_\AN^\bullet$ is concerned, we must show that for every
  $F \in \Pred_\AN^\bullet X$,
  $s_{\AN\,Y}^\bullet (\Pred f (F)) = \HV (\Pred f) \allowbreak
  (s_{\AN\,X}^\bullet (F))$.  The left-hand side is convex, and we
  claim that the right-hand side is, too.  Knowing this, we will be
  able to conclude: since $r_{\AN\,Y}$ restricted to
  $\HV^{cvx} (\Pred_\Nature^\bullet X)$ is a homeomorphism, it is
  enough to show that
  $r_{\AN\,Y} (s_{\AN\,Y}^\bullet (\Pred f (F))) = r_{\AN\,Y} (\HV
  (\Pred f) \allowbreak (s_{\AN\,X}^\bullet (F)))$, and this will
  follow from the naturality of $r_{\AN}$.

  Hence it remains to show that
  $\HV (\Pred f) \allowbreak (s_{\AN\,X}^\bullet (F))$ is convex.
  This is equal to $cl (A)$, where
  $A \eqdef \{\Pred f (G) \mid G \in s_{\AN\,X}^\bullet (F)\}$.  Since
  $s_{\AN\,X}^\bullet (F)$ is convex and $\Pred f$ commutes with
  scalar multiplication and with addition, $A$ is convex.  By
  Fact~\ref{fact:cl:conv}, $cl (A)$ is convex, too.  \qed
\end{proof}


We can now transport Theorem~\ref{thm:Hcvx:projlim} to the world of
sublinear previsions, as follows.
\begin{theorem}
  \label{thm:AN:projlim}
  Let ${(p_{ij} \colon X_j \to X_i)}_{i \sqsubseteq j \in I}$ be a
  projective system of topological spaces, with canonical projective
  limit $X, {(p_i)}_{i \in I}$.  Let $\bullet$ be nothing,
  ``$\leq 1$'' or ``$1$''.
  If:
  \begin{enumerate}
  \item the projective system is an ep-system,
  \item or $I$ has a countable cofinal subset and each $X_i$ is
    locally compact sober (and compact, if $\bullet$ is ``$1$''),
  \item or every $X_i$ is $\odot$-consonant sober (and compact if
    $\bullet$ is ``$1$'') and every $p_{ij}$ is a proper map,
  \end{enumerate}
  then
  ${(\Pred_\AN^\bullet {p_{ij}} \colon \Pred_\AN^\bullet {X_j} \to
    \Pred_\AN^\bullet {X_i})}_{i \sqsubseteq j \in I}$ is a projective
  system of topological spaces, and
  $\Pred_\AN^\bullet X, {(\Pred_\AN^\bullet {p_i})}_{i \in I}$ is its
  projective limit, up to homeomorphism.
\end{theorem}

\begin{proof}
  In case~1, we use Proposition~\ref{prop:ep:subcont}, noticing that
  $\Pred_\AN^\bullet X$ is a subdcpo of $KX$, namely that pointwise
  directed suprema of (subnormalized, normalized) sublinear previsions
  are again (subnormalized, normalized) sublinear previsions.

  In cases~2 and~3, we apply the corresponding cases of
  Theorem~\ref{thm:Hcvx:projlim}.  To this end, we need to verify that
  $r_\AN$ and $s^\bullet_\AN$ are a natural homeomorphism on a
  subcategory of $\Topcat$ that contains the spaces $X_i$ and the
  limit $X$; this is a special case of
  Lemma~\ref{lemma:retract:limits} where the retraction is in fact a
  homeomorphism.  In light of Lemma~\ref{lemma:rAP:nat}, it suffices
  to show that every $X_i$ is $\AN_\bullet$-friendly, as well as $X$.

  In case~2, every locally compact sober space is
  $\AN_\bullet$-friendly, and that is the case of each $X_i$.  $X$ may
  fail to be locally compact, but it is $\odot$-consonant by
  Proposition~\ref{prop:oconsonant:lim:omega:lc}.  When $\bullet$ is
  ``$1$'', it is also compact by Steenrod's theorem.  In any case, $X$
  is $\AN_\bullet$-friendly by Fact~\ref{fact:assum}.

  In case~3, every $p_{ij}$ is proper, and every $X_i$ is
  $\odot$-consonant sober, so $X$ is, too, by
  Corollary~\ref{corl:consonant:lim}.  When $\bullet$ is ``$1$'',
  every $X_i$ is compact sober, so $X$ is, too, by Steenrod's theorem.
  By Fact~\ref{fact:assum}, all the spaces and $X_i$ and $X$ are
  therefore $\AN_\bullet$-friendly.
\end{proof}


\section{Forks}
\label{sec:forks}

We arrive at our final functors, which mix probabilistic and erratic
non-determinism.  A \emph{fork} on a space $X$ is any pair
$(F^-, F^+)$ of a superlinear prevision $F^-$ on $X$ and of a
sublinear prevision $F^+$ on $X$ satisfying \emph{Walley's condition}:
\[
  F^- (h+h') \leq F^- (h) + F^+ (h') \leq F^+ (h+h')
\]
for all $h, h' \in \Lform X$ \cite{Gou-csl07,KP:predtrans:pow}.  A
fork is \emph{subnormalized}, resp.\ \emph{normalized} if and only if
both $F^-$ and $F^+$ are.

We write $\Pred_{\ADN} X$ for the set of all forks on $X$, and
$\Pred^{\leq 1}_{\ADN} X$, $\Pred^1_{\ADN} X$ for their subsets of
subnormalized, resp.\ normalized, forks.  The \emph{weak topology} on
each is the subspace topology induced by the inclusion into the larger
space $\Pred_{\DN} X \times \Pred_{\AN} X$.  A subbase of the weak
topology is composed of two kinds of open subsets: $[h > r]^-$,
defined as $\{(F^-, F^+) \mid F^- (h) > r\}$, and $[h > r]^+$, defined
as $\{(F^-, F^+) \mid F^+ (h) > r\}$, where $h \in \Lform X$,
$r \in \real^+$.  The specialization ordering of spaces of forks is
the product ordering $\leq \times \leq$, where $\leq$ denotes the
pointwise ordering on previsions.  In particular, all those spaces of
forks are $T_0$.

It is easy to see that, whether $\bullet$ is nothing, ``$\leq 1$'', or
``$1$'', $\Pred_\ADN^\bullet$ defines an endofunctor on $\Topcat$,
whose action on morphisms is given by
$\Pred_\ADN^\bullet f \eqdef (\Pred f, \Pred f)$.

\begin{lemma}
  \label{lemma:Fork:projlim}
  Let $\bullet$ be nothing, ``$\leq 1$'', or ``$1$'', and $T$ be the
  $\Pred_\ADN^\bullet$ functor.  The comparison map
  $\varphi \colon TX \to Z$ of any projective $T$-situation is a
  topological embedding.
\end{lemma}
\begin{proof}
  Let $Z^\sharp, {(q_i^\sharp)}_{i \in I}$ be the canonical projective
  limit of
  $(\Pred_\DN^\bullet {p_{ij}} \colon \Pred_\DN^\bullet {X_j} \to
  \Pred_\DN^\bullet {X_i})_{i \sqsubseteq j \in I}$ and
  $\varphi^\sharp \colon \Pred_\DN^\bullet X \to Z^\sharp$ be the
  comparison map.  Similarly with $Z^\flat, {(q_i^\flat)}_{i \in I}$
  and $\Pred_\AN^\bullet$.  We also take the notations ($Z$,
  $\varphi$, $q_i$) from Definition~\ref{defn:situation}, with
  $T \eqdef \Pred_\ADN^\bullet$.

  By definition (see Definition~\ref{defn:situation}),
  $\varphi^\sharp$ maps every $F^- \in \Pred_\DN^\bullet X$ to
  ${(\Pred p_i (F^-))}_{i \in I}$, $\varphi^\flat$ maps every
  $F^+ \in \Pred_\AN^\bullet X$ to ${(\Pred p_i (F^+))}_{i \in I}$,
  and $\varphi$ maps every fork $(F^-, F^+) \in \Pred_\ADN^\bullet X$
  to ${(\Pred p_i (F^-), \Pred p_i (F^+))}_{i \in I}$.  Also, the maps
  $q_i^\sharp$, $q_i^\flat$, $q_i$ are just projection onto coordinate
  $i$, just like $p_i$.

  As a consequence, for every $i \in I$, for every
  $h_i \in \Lform {X_i}$, for every $r \in \Rp$, for every
  $\pm \in \{-, +\}$,
  ${\varphi}^{-1} (q_i^{-1} ([h_i > r]^\pm)) = [h_i \circ p_i]^\pm$.
  Indeed, $(F^-, F^+)$ is in the left-hand side if and only if
  $\Pred {p_i} (F^\pm) (h_i) > r$, if and only if
  $F^\pm (h_i \circ p_i) > r$, if and only if
  $(F^-, F^+) \in [h_i \circ p_i > r]^\pm$.

  A subbase of the topology on $\Pred_\ADN^\bullet X$ is given by the
  sets $[h > r]^\pm$ where $h \in \Lform X$, $r \in \Rp$, and
  $\pm \in \{+, -\}$.  For every $i \in I$, let $h_i$ be the largest
  map in $\Lform {X_i}$ such that $h_i \circ p_i \leq h$, as given in
  Lemma~\ref{lemma:hi}.  Now $[h > r]^\pm$ is the collection of
  (subnormalized, normalized) forks $(F^-, F^+)$ such that
  $F^\pm (h) > r$, or equivalently such that
  $F^\pm (h_i \circ p_i) > r$ for some $i \in I$, using item~6 of that
  lemma and the Scott-continuity of $F^\pm$.  In other words,
  $[h > r]^\pm = \dcup_{i \in I} [h_i \circ p_i > r]^\pm$, and we have
  seen that this is equal to
  $\dcup_{i \in I} {\varphi}^{-1} (q_i^{-1} ([h_i > r]^\pm))$, hence
  to ${\varphi}^{-1} (\dcup_{i \in I} q_i^{-1} ([h_i > r]^\pm))$.
  Therefore $\varphi$ is full.

  Since $\Pred_\ADN^\bullet X$ is $T_0$, $\varphi$ is a topological
  embedding.  \qed
\end{proof}

\begin{theorem}
  \label{thm:Fork:projlim}
  Let $\bullet$ be nothing, ``$\leq 1$'', or ``$1$''.  Let
  ${(p_{ij} \colon X_j \to X_i)}_{i \sqsubseteq j \in I}$ be a
  projective system of topological spaces, with canonical projective
  limit $X, {(p_i)}_{i \in I}$.  If 
  $\Pred_\DN^\bullet X$ is a projective limit of
  $(\Pred_\DN^\bullet {p_{ij}} \colon \Pred_\DN^\bullet {X_j} \to
    \Pred_\DN^\bullet {X_i})_{i \sqsubseteq j \in I}$ and if
  $\Pred_\AN^\bullet X$ is a projective limit of
  ${(\Pred_\AN^\bullet {p_{ij}} \colon \Pred_\AN^\bullet {X_j} \to
    \Pred_\AN^\bullet {X_i})}_{i \sqsubseteq j \in I}$, then
  $\Pred_\ADN^\bullet X$ is a projective limit of
  $(\Pred_\ADN^\bullet {p_{ij}} \colon \Pred_\ADN^\bullet {X_j} \to
  \Pred_\ADN^\bullet {X_i})_{i \sqsubseteq j \in I}$.
\end{theorem}
\begin{proof}
  Let $Z^\sharp, {(q_i^\sharp)}_{i \in I}$ be the canonical projective
  limit of
  $(\Pred_\DN^\bullet {p_{ij}} \colon \Pred_\DN^\bullet {X_j} \to
  \Pred_\DN^\bullet {X_i})_{i \sqsubseteq j \in I}$ and
  $\varphi^\sharp \colon \Pred_\DN^\bullet X \to Z^\sharp$ be the
  comparison map.  Similarly with $Z^\flat, {(q_i^\flat)}_{i \in I}$
  and $\Pred_\AN^\bullet$, with
  $Z^\natural, {(q_i^\natural)}_{i \in I}$ and $\Pred_\ADN^\bullet$.
  By assumption, $\varphi^\sharp$ and $\varphi^\flat$ are
  homeomorphisms.  Relying on Lemma~\ref{lemma:Fork:projlim}, it
  remains to show that $\varphi$ is surjective.

  Let ${(F_i^-, F_i^+)}_{i \in I}$ be any element of $Z$.  This means
  that every $(F_i^-, F_i^+)$ is in $\Pred_\ADN^\bullet {X_i}$, and
  that for all $i \sqsubseteq j \in I$,
  $(F_i^-, F_i^+) = \Pred_\ADN^\bullet {p_{ij}} (F_j^-, F_j^+) =
  (\Pred {p_{ij}} (F_j^-), \Pred {p_{ij}} (F_j^+))$.  In particular,
  ${(F_i^-)}_{i \in I}$ is in $Z^\sharp$, hence is equal to
  $\varphi^\sharp (F^-)$ for some $F^- \in \Pred_\DN^\bullet X$, and
  ${(F_i^+)}_{i \in I}$ is in $Z^\flat$, hence is equal to
  $\varphi^\flat (F^+)$ for some $F^+ \in \Pred_\AN^\bullet X$.
  Explicitly, this means that $\Pred {p_i} (F^+) = F_i^+$ and
  $\Pred {p_i} (F^-) = F_i^-$ for every $i \in I$, namely that for
  every $h_i \in \Lform {X_i}$,
  $F^\pm (h_i \circ p_i ) = F_i^\pm (h_i)$ (with $\pm$ equal to $-$ or
  to $+$).

  We claim that $(F^-, F^+)$ is in $\Pred_\ADN^\bullet X$.  It
  suffices to check Walley's condition.  For all $h, h' \in \Lform X$,
  we write $h_i$ for the largest map in $\Lform {X_i}$ such that $h_i
  \circ p_i \leq h$, for every $i \in I$, and similarly with $h'_i$.  Then:
  \begin{align*}
    F^- (h + h')
    & = F^- (\dsup_{i \in I} (h_i \circ p_i) + \dsup_{i \in I} (h'_i
      \circ p_i))
    & \text{Lemma~\ref{lemma:hi}, item~5} \\
    & = F^- (\dsup_{i \in I} (h_i + h'_i) \circ p_i)
    & \text{$+$ is Scott-continuous} \\
    & = \dsup_{i \in I} F^- ((h_i + h'_i) \circ p_i)
    & \text{$F^-$ is Scott-continuous} \\
    & = \dsup_{i \in I} F_i^- (h_i + h'_i) \\
    & \leq \dsup_{i \in I} (F_i^- (h_i) + F_i^+ (h'_i))
    & \mskip-60mu\text{Walley's condition on $(F_i^-, F_i^+)$} \\
    & = \dsup_{i \in I} F_i^- (h_i) + \dsup_{i \in I} F_i^+ (h'_i) \\
    & = \dsup_{i \in I} F^- (h_i \circ p_i) + \dsup_{i \in I} F^+ (h'_i
      \circ p_i) \\
    & = F^- (h) + F^+ (h'),
  \end{align*}
  by using the Scott-continuity of $F^-$ and $F^+$, and
  Lemma~\ref{lemma:hi}, item~5.  The inequality $F^- (h) + F^+ (h')
  \leq F^+ (h+h')$ is proved similarly.

  We have now found an element $(F^-, F^+)$ of $\Pred_\ADN X$ such
  that $\Pred {p_i} (F^+) = F_i^+$ and $\Pred {p_i} (F^-) = F_i^-$ for
  every $i \in I$, hence such that $\Pred_\ADN^\bullet {p_i} (F^-,
  F^+) = (F_i^-, F_i^+)$ for every $i \in I$.  Hence $\varphi (F^-,
  F^+) = {(F_i^-, F_i^+)}_{i \in I}$.  \qed
\end{proof}

We apply Theorem~\ref{thm:Fork:projlim} and list conditions under
which $\Pred_\DN^\bullet$ preserves projective limits (equivalently,
$\Val_\bullet$, by Theorem~\ref{thm:DN:projlim}, hence the conditions
of Theorem~\ref{thm:V:projlim}), and under which $\Pred_\AN^\bullet$
also preserves projective limits; in other words, we appeal to
Theorem~\ref{thm:DN:projlim} and to Theorem~\ref{thm:AN:projlim}, and
we obtain the following.
\begin{corollary}
  \label{corl:Fork:projlim}
  Let ${(p_{ij} \colon X_j \to X_i)}_{i \sqsubseteq j \in I}$ be a
  projective system of topological spaces, with canonical projective
  limit $X, {(p_i)}_{i \in I}$.  Let $\bullet$ be nothing,
  ``$\leq 1$'' or ``$1$''.  If:
  \begin{enumerate}
  \item the projective system is an ep-system,
  \item or $I$ has a countable cofinal subset and each $X_i$ is
    locally compact sober (and compact, if $\bullet$ is ``$1$''),
  \item or every $X_i$ is $\odot$-consonant sober (and compact, if
    $\bullet$ is ``$1$'') and every $p_{ij}$ is a proper map,
  \end{enumerate}
  then
  ${(\Pred_\ADN^\bullet {p_{ij}} \colon \Pred_\ADN^\bullet {X_j} \to
    \Pred_\ADN^\bullet {X_i})}_{i \sqsubseteq j \in I}$ is a
  projective system of topological spaces, and
  $\Pred_\ADN^\bullet X, {(\Pred_\ADN^\bullet {p_i})}_{i \in I}$ is
  its projective limit, up to homeomorphism.
\end{corollary}

\section{The $\Plotkinn^{cvx} \Pred_\Nature^\bullet$ functor}
\label{sec:plotk-pred_n-funct}

We will transport this result to the matching model of convex lenses
over spaces of continuous (subprobability, probability) valuations of
\cite{TKP:nondet:prob}, through a suitable homeomorphism.  However,
this will work in a restricted setting.

There is such a homeomorphism between
$\Plotkinn^{cvx} \Pred_\Nature^\bullet X$, the subspace of
$\Plotkinn \Pred_\Nature^\bullet X$ of convex lenses, and
$\Pred_\ADN^\bullet X$, for every space $X$ such that $\Lform X$ is
locally convex \emph{and} has a almost open addition map
\cite[Theorem~4.17]{JGL-mscs16}.  The latter property means that for
all open subsets $\mathcal U$ and $\mathcal V$ of $\Lform X$,
$\upc (\mathcal U + \mathcal V) = \{f \in \Lform X \mid \exists g \in
\mathcal U, h \in \mathcal V, f \geq g+h\}$ is open.  When $\bullet$
is ``$1$'', we also need to require $X$ to be compact.  The
homeomorphism is the restriction of a retraction
$(r_{\ADN\,X}, s_{\ADN\,X}^\bullet)$, where
$r_{\ADN\,X} \colon \Plotkinn \Pred_\Nature^\bullet X \to \Pred_\ADN
X$ maps every lens $L$ to the fork
$(h \mapsto \inf_{G \in L} G (h), h \mapsto \sup_{G \in L} G (h))$,
and
$s_{\ADN\,X} \colon \Pred_\ADN X \to \Plotkinn \Pred_\Nature^\bullet
X$ maps $(F^-, F^+)$ to
$\{G \in \Pred_\Nature^\bullet X \mid F^- \leq G \leq F^+\}$.  The
fact that it is a retraction is also predicated on the fact that
$\Lform X$ is locally convex, has an almost open addition map, and
that $X$ is compact if $\bullet$ is ``$1$''
\cite[Proposition~3.32]{JGL-mscs16}.  We only consider the
homeomorphisms, not the retractions, here.
\begin{lemma}
  \label{lemma:rADP:nat}
  The transformations $r_\ADN$ and $s_\ADN^\bullet$ between
  $\Plotkinn^{cvx} \Pred_\Nature^\bullet$ and $\Pred_\ADN^\bullet$ are
  $\catk^\bullet$-relative natural, where $\catk^\bullet$ is the full subcategory of
  $\Topcat$ consisting of spaces $X$ such that $\Lform X$ is locally
  convex, with an almost open addition map, and such that $X$ is
  compact in case $\bullet$ is ``$1$''.
\end{lemma}
\begin{proof}
  Since these transformations consist of mutually inverse
  homeomorphisms, it suffices to show that $r_\ADN$ is natural.
  Lemma~4.6 of \cite{JGL-mscs16} states that: $(*)$ for every lens
  $L \in \Plotkinn \Pred_\Nature^\bullet X$ (in particular for any
  convex lens), for every $h \in \Lform X$,
  $\sup_{G \in \upc L} G (h) = \sup_{G \in L} G (h)$ and
  $\inf_{G \in cl (L)} G (h) = \inf_{G \in L} G (h)$.  The action of
  the $\Plotkinn^{cvx}$ functor on a morphism $g$ maps any lens $L$ to
  $\upc g [L] \cap cl (g [L])$ \cite[Proposition
  4.33]{TKP:nondet:prob}.  Hence, for every continuous map
  $f \colon X \to Y$, for every
  $L \in \Plotkinn^{cvx} \Pred_\Nature^\bullet X$,
  $r_\ADN (\Plotkinn^{cvx} \Pred f (L)) = (h \mapsto \inf_{G \in \upc
    (\upc \Pred f [L] \cap cl (\Pred f [L]))} G (h), \allowbreak h
  \mapsto \sup_{G \in cl (\upc \Pred f [L] \cap cl (\Pred f [L]))} G
  (h))$.  For every $h \in \Lform X$,
  \begin{align*}
    & \inf_{G \in \upc (\upc \Pred f [L] \cap cl (\Pred f [L]))} G (h) \\
    & = \inf_{G \in \upc \Pred f [L] \cap cl (\Pred f [L])} G (h)
    & \text{by $(*)$} \\
    & \geq \inf_{G \in \upc \Pred f [L]} G (h)
    & \text{since } \upc \Pred f [L] \cap cl (\Pred f [L]) \subseteq
      \upc \Pred f [L] \\
    & = \inf_{G \in \Pred f [L]} G (h)
    & \text{by $(*)$} \\
    & \geq \inf_{G \in \upc \Pred f [L] \cap cl (\Pred f [L])} G (h)
    & \text{since } \Pred f [L] \subseteq \upc \Pred f [L] \cap cl (\Pred f [L]),
  \end{align*}
  so all those values are equal.  In particular,
  $\inf_{G \in \upc (\upc \Pred f [L] \cap cl (\Pred f [L]))} G (h) =
  \inf_{G \in \Pred f [L]} G (h)$.  Similarly,
  $\sup_{G \in cl (\upc \Pred f [L] \cap cl (\Pred f [L]))} G (h) =
  \sup_{G \in \Pred f [L]} G (h)$.  Then
  $\inf_{G \in \Pred f [L]} G (h) = \inf_{G' \in L} G' (h \circ f)$,
  and
  $\sup_{G \in \Pred f [L]} G (h) = \sup_{G' \in L} G' (h \circ f)$,
  so
  $r_\ADN (\Plotkinn^{cvx} \Pred f (L)) = (h \mapsto \inf_{G' \in L}
  G' (h \circ f), h \mapsto \sup_{G' \in L} G' (h \circ f))$.  We
  compare this to
  $\Pred_\ADN^\bullet f (r_\ADN (L)) = \Pred_\ADN^\bullet f (h'
  \mapsto \inf_{G' \in L} G' (h'), h' \mapsto \sup_{G' \in L} G'
  (h'))$, and we find that those are equal.  \qed
\end{proof}

Let us write $\Lformco X$ for the space $\Lform X$,
but with the compact-open topology instead of the Scott topology.  The
compact-open topology is generated by open subsets
$[Q > r] \eqdef \{h \in \Lform X \mid \forall x \in Q, h (x) > r\}$,
where $Q$ ranges over the compact saturated of $X$ and $r \in \Rp$.
We can even restrict to basic open subsets $[Q > r]$ where $Q$ is
compact saturated, since $[Q > r] = [\upc Q > r]$.

\begin{lemma}
  \label{lemma:LcoX:consistent}
  For every weakly Hausdorff, coherent space $X$, addition is almost
  open on $\Lformco X$, viz., for all open subsets $\mathcal U$ and
  $\mathcal V$ of $\Lformco X$, $\upc (\mathcal U + \mathcal V)$ is
  open in $\Lformco X$.
\end{lemma}
\begin{proof}
  For every compact saturated subset $Q$ of $X$, let
  $\langle Q \searrow r\rangle$ be the function that maps every
  element of $Q$ to $r$, and all others to $0$.  A \emph{co-step
    function} is a pointwise supremum of a finite family of such
  functions \cite[Section~2]{EEK:waybelow}.  We also define a relation
  $\prec$ on functions from $X$ to $\creal$ by $f \prec g$ if and only
  if for every $x \in X$, $f (x) \ll g (x)$, where $\ll$ is the
  way-below relation on $\creal$---namely, $r \ll s$ if and only if
  $r=0$ or $r<s$.  Finally, we write $\Uuarrow f$ for $\{g \in \Lform
  X \mid f \prec g\}$.

  We claim that the sets $\Uuarrow f$ form a base of the compact-open
  topology on $\Lformco X$, where $f$ ranges over the co-step
  functions.  Let
  $f \eqdef \sup_{i=1}^n \langle Q_i \searrow r_i \rangle$, where each
  $Q_i$ is compact saturated and each $r_i \in \Rp$.  Without loss of
  generality, we assume that $r_i > 0$.  Then, for every
  $g \in \Lform X$, $f \prec g$ if and only if
  $g \in \bigcap_{i=1}^n [Q_i > r_i]$.  Indeed, if $f \prec g$, then
  for every $i \in \{1, \cdots, n\}$, for every $x \in Q_i$,
  $f (x) \geq r_i$ and $f (x) \ll g (x)$, so $g (x) > r_i$, using the
  fact that $r_i > 0$.  Conversely, if
  $g \in \bigcap_{i=1}^n [Q_i > r_i]$, then for every $x \in X$, let
  $I \eqdef \{i \in \{1, \cdots, n\} \mid x \in Q_i\}$.  If $I$ is
  empty, then $f (x) = 0 \ll g (x)$.  Otherwise, let $i \in I$ be such
  that $r_i$ is largest.  Then $f (x) = r_i$, and since
  $g \in [Q_i > r_i]$, we have $f (x) < g (x)$; in any case,
  $f \prec g$.

  Any co-step function $f$ takes only finitely many values, and
  $f^{-1} ([r, \infty])$ is compact saturated for every
  $r \in \Rp \diff \{0\}$.  Conversely, if $f$ is any function from
  $X$ to $\creal$ that takes only finitely many values and is such
  that $f^{-1} ([r, \infty])$ is compact saturated for every
  $r \in \Rp \diff \{0\}$, then we claim that $f$ is a co-step
  function.  Indeed, it suffices to list and sort the non-zero values
  taken by $f$ as $r_1 > \cdots > r_n > 0$, to define
  $Q_i \eqdef f^{-1} ([r_i, \infty])$ for every
  $i \in \{1, \cdots, n\}$, and to verify that
  $f = \sup_{i=1}^n \langle Q_i \searrow r_i \rangle$.  In order to
  see this, we lett
  $g \eqdef \sup_{i=1}^n \langle Q_i \searrow r_i \rangle$, and we
  show that $f^{-1} ([r, \infty]) = g^{-1} ([r, \infty])$ for every
  $r \in \Rp \diff \{0\}$.  Both $f$ and $g$ take their values in the
  same set $\{0, r_1, \cdots, r_n\}$, so it is enough to verify that
  $f^{-1} ([r_i, \infty]) = g^{-1} ([r_i, \infty])$ for every
  $i \in \{1, \cdots, n\}$.  But
  $g^{-1} ([r_i, \infty]) = Q_1 \cup \cdots \cup Q_i$, since
  $r_i > r_{i+1} > \cdots > r_n$; it is easy to see that, since
  $r_1 > \cdots > r_n$, we also have
  $Q_1 \subseteq \cdots \subseteq Q_n$, so
  $g^{-1} ([r_i, \infty]) = Q_i = f^{-1} ([r_i, \infty])$.

  It follows that the sum of any two co-step functions is a co-step
  function.  Indeed, if $f$ and $g$ are co-step functions, with values
  taken in the finite sets $A$ and $B$ respectively, then $f+g$ takes
  its values in $A+B$, and for every $r \in \Rp \diff \{0\}$,
  $(f+g)^{-1} ([r, \infty]) = \bigcup_{a \in A, b \in B, a+b \geq r}
  (f^{-1} ([a, \infty]) \cap g^{-1} ([b, \infty]))$.  The latter is
  compact saturated because the union is finite, and because
  $f^{-1} ([a, \infty]) \cap g^{-1} ([b, \infty])$ is compact
  saturated.  Indeed, $a+b \geq r > 0$ implies that $a$ and $b$ cannot
  both be $0$.  If $a=0$, then that is equal to
  $f^{-1} ([a, \infty]) \cap g^{-1} ([b, \infty]) = g^{-1} ([b,
  \infty])$ is compact saturated (since $b > 0$); similarly similarly
  if $b=0$; and if $a, b > 0$, then
  $f^{-1} ([a, \infty]) \cap g^{-1} ([b, \infty])$ is compact
  saturated because $X$ is coherent.

  Now let $f \in \upc (\mathcal U + \mathcal V)$.  There are two lower
  semicontinuous map $g \in \mathcal U$ and $h \in \mathcal V$ such
  that $f \geq g+h$.  Since $g \in \mathcal U$, there is a co-step
  function $g_0$ such that $g \in \Uuarrow g_0 \subseteq \mathcal U$.
  Similarly, there is a co-step function $h_0$ such that
  $h \in \Uuarrow h_0 \subseteq \mathcal V$.  Let
  $f_0 \eqdef g_0 + h_0$: $f_0$ is a co-step function, and
  $f_0 = g_0 + h_0 \prec g+h \leq f$, so $f \in \Uuarrow f_0$.  It
  remains to show that $\Uuarrow f_0$ is included in
  $\upc (\Uuarrow g_0 + \Uuarrow h_0)$, which will imply that it is
  included in $\upc (\mathcal U + \mathcal V)$.

  Let $f'$ be any element of $\Uuarrow f_0$.  Let $A$ (resp.\ $B$) be
  the set of values taken by $g_0$ (resp.\ $h_0$).  We recall that
  $f_0^{-1} ([r, \infty]) = \bigcup_{a \in A, b \in B, a+b \geq r}
  (g_0^{-1} ([a, \infty]) \cap h_0^{-1} ([b, \infty]))$ for every
  $r \in \Rp \diff \{0\}$.  For every pair of values $a \in A$,
  $b \in B$ such that $a+b > 0$, for every
  $x \in g_0^{-1} ([a, \infty]) \cap h_0^{-1} ([b, \infty])$, we have
  $f_0 (x) = g_0 (x) + h_0 (x) \geq a+b$, so $f' (x) > a+b$.  Since
  $g_0^{-1} ([a, \infty]) \cap h_0^{-1} ([b, \infty])$ is compact,
  $\min_{x \in g_0^{-1} ([a, \infty]) \cap h_0^{-1} ([b, \infty])} f'
  (x)$ exists and is strictly larger than $a+b$.  Hence it is also
  strictly larger than $(1+\epsilon_{a,b}) (a+b)$ for some number
  $\epsilon_{a,b} > 0$.  Letting
  $\epsilon \eqdef \min_{a \in A, b \in B, a+b > 0} \epsilon_{a,b}$,
  we have obtained that there is a number $\epsilon > 0$ such that for
  all $a \in A$ and $b \in B$ such that $a+b > 0$,
  $g_0^{-1} ([a, \infty]) \cap h_0^{-1} ([b, \infty]) \subseteq
  {f'}^{-1} (](1+\epsilon) (a+b), \infty])$.
  
  We claim that there are an open neighborhood $U_{a,b}$ of
  $g_0^{-1} ([a, \infty])$ and an open neighborhood $V_{a,b}$ of
  $h_0^{-1} ([a, \infty])$ such that
  $U_{a,b} \cap V_{a,b} \subseteq {f'}^{-1} (](1+\epsilon) (a+b),
  \infty])$.  When $a, b > 0$, this is because $X$ is weakly
  Hausdorff.  If $a=0$, we simply take $U_{a,b} \eqdef X$ and
  $V_{a,b} \eqdef {f'}^{-1} (](1+\epsilon) (a+b), \infty])$, and
  symmetrically if $b=0$.

  For every $a \in A$, let
  $U_a \eqdef \bigcap_{b \in B, a+b>0} U_{a,b}$, and for every
  $b \in B$, let $V_b \eqdef \bigcap_{a \in A, a+b>0} V_{a,b}$.  Those
  are open sets, since the intersections are finite.  For every
  $a \in A$, $g_0^{-1} ([a, \infty])$ is included in $U_a$, and
  similarly for every $b \in B$, $h_0^{-1} ([a, \infty])$ is included
  in $V_b$.  Additionally, for all $a \in A$ and $b \in B$ such that
  $a+b > 0$,
  $U_a \cap V_b \subseteq U_{a,b} \cap V_{a,b} \subseteq {f'}^{-1}
  (](1+\epsilon) (a+b), \infty])$.

  Let $g' \eqdef \sup_{a \in A} (1+\epsilon) a \chi_{U_a}$ and
  $h' \eqdef \sup_{b \in B} (1+\epsilon) b \chi_{U_b}$.  Those are
  suprema of characteristic maps of lower semicontinuous maps, hence
  are lower semicontinuous maps.  Since
  $g_0^{-1} ([a, \infty]) \subseteq U_a$ for every $a \in A$,
  $g_0 \prec g'$: for every $x \in X$, either $g_0 (x)=0$ or not, and
  in the latter case, let $a \eqdef g_0 (x) \in A \diff \{0\}$; then
  $x \in g_0^{-1} ([a, \infty])$, so $x \in U_a$, and hence
  $g' (x) \geq (1+\epsilon) a > a = g_0 (x)$.  Similarly,
  $h_0 \prec h'$.  Finally, $g'+h' \leq f'$: for every $x \in X$,
  either $g' (x)=h' (x)=0$ and this is clear, or
  $g' (x)=(1+\epsilon)a$ for some $a \in A$ such that $x \in U_a$ and
  $h' (x) = (1+\epsilon)b$ for some $b \in B$ such that $x \in V_b$,
  and $a+b > 0$.  Since $U_a \cap V_b \subseteq {f'}^{-1}
  (](1+\epsilon) (a+b), \infty])$, $f' (x) > (1+\epsilon) (a+b) = g'
  (x) + h' (x)$.

  Hence $g' \in \Uuarrow g_0$, $h' \in \Uuarrow h_0$, and
  $f' \geq g'+h'$.  Therefore
  $f' \in \upc (\Uuarrow g_0 + \Uuarrow h_0)$.  Since $f'$ is
  arbitrary in $\Uuarrow f_0$, $\Uuarrow f_0$ is included in
  $\upc (\Uuarrow g_0 + \Uuarrow h_0)$, hence in
  $\upc (\mathcal U + \mathcal V)$, as promised.  \qed
\end{proof}

\begin{corollary}
  \label{corl:LX:consistent}
  For every $\odot$-consonant, weakly Hausdorff, coherent space $X$,
  $\Lform X$ is locally convex and addition is almost open on
  $\Lform X$.
\end{corollary}
\begin{proof}
  The compact-open topology is always coarser than the Scott topology,
  and coincides with it when $X$ is $\odot$-consonant
  \cite[Proposition~13.4]{dBGLJL:LCScomplete}.  $\Lform X$ is locally
  convex by \cite[Lemma~13.6]{dBGLJL:LCScomplete}, and almost openness
  is by Lemma~\ref{lemma:LcoX:consistent}.  \qed
\end{proof}

With all that, we transport the result of
Corollary~\ref{corl:Fork:projlim} from forks to convex lenses over
spaces of continuous valuations as follows.  The conditions are
stricter than what we are accustomed to, as we need the maps $p_{ij}$
to be proper.
\begin{theorem}
  \label{thm:PV:projlim}
  Let ${(p_{ij} \colon X_j \to X_i)}_{i \sqsubseteq j \in I}$ be a
  projective system of topological spaces, with canonical projective
  limit $X, {(p_i)}_{i \in I}$.  Let $\bullet$ be nothing,
  ``$\leq 1$'' or ``$1$''.  If every $X_i$ is $\odot$-consonant and
  locally strongly sober (and compact, namely strongly sober, if
  $\bullet$ is ``$1$'') and if every $p_{ij}$ is a proper map, then
  ${(\Plotkinn^{cvx} \Val {p_{ij}} \colon \Plotkinn^{cvx} \Val_\bullet
    {X_j} \to \Plotkinn^{cvx} \Val_\bullet {X_i})}_{i \sqsubseteq j
    \in I}$ is a projective system of topological spaces, and
  $\Plotkinn^{cvx} \Val_\bullet X, {(\Plotkinn^{cvx} \Val {p_i})}_{i
    \in I}$ is its projective limit, up to homeomorphism.
\end{theorem}
\begin{proof}
  We recall that the locally strongly sober spaces are exactly the
  weakly Hausdorff, coherent, and sober spaces
  \cite[Theorem~3.5]{JGL:wHaus}.

  By Corollary~\ref{corl:Fork:projlim} (case~3), $\Pred_\ADN^\bullet$ preserves
  limits of such projective systems.  Lemma~\ref{lemma:rADP:nat} gives
  use a $\catk^\bullet$-natural retraction (even isomorphism) of
  $\Pred_\ADN^\bullet$ onto $\Plotkinn^{cvx} \Pred_\Nature$, or
  equivalently onto $\Plotkinn^{cvx} \Val$, where $\catk^\bullet$ is the full
  subcategory of $\Topcat$ consisting of spaces $X$ such that
  $\Lform X$ is locally convex, with an almost open addition map, and
  such that $X$ is compact in case $\bullet$ is ``$1$''.  We can then
  apply Lemma~\ref{lemma:retract:limits} and conclude, provided we can
  show that not only the spaces $X_i$ are in $\catk^\bullet$, but also $X$.
  The spaces $X_i$ are $\odot$-consonant, weakly Hausdorff and
  coherent, so $\Lform {X_i}$ is locally convex and addition is almost
  open on it by Corollary~\ref{corl:LX:consistent}.  The classes of
  locally strongly sober spaces and of strongly sober spaces are
  projective \cite[Theorem 5.1]{JGL:proj:class}, so $X$ is locally
  strongly sober (and compact if $\bullet$ is ``$1$''), namely, weakly
  Hausdorff, coherent, and sober (and compact if $\bullet$ is
  ``$1$'').  By Corollary~\ref{corl:consonant:lim}, $X$ is also
  $\odot$-consonant.  Therefore, by
  Corollary~\ref{corl:LX:consistent}, $\Lform X$ is also locally
  convex with an almost open addition map.  \qed
\end{proof}

\section{Conclusion}
\label{sec:conclusion}

Looking back on what we did, it is apparent that we have dealt with
each functor at hand by its own specific techniques.  It would be
nicer if there were a general projective limit preservation theorem
that would entail the results we have obtained.  This is rather
unlikely: the conditions we have obtained for our results differ for
each functor we have considered, and those conditions were shown to be
necessary---at least in the first part of this paper
(Sections~\ref{sec:general-setting}---\ref{sec:lenses}).

The following are a few remaining open questions:
\begin{enumerate}
\item Theorem~\ref{thm:H:projlim}, item~3 requires each space $X_i$ to
  be locally compact sober, and we have shown that a similar result
  would fail for spaces that are not completely Baire.  Would the
  conclusion of the theorem still hold if each $X_i$ were assumed to
  be quasi-Polish? domain-complete? LCS-complete?
\item Theorem~\ref{thm:QL:projlim} states that projective limits of
  sober spaces preserved by $\HV$ are preserved by $\AvalV$ and
  $\QLV$.  Does the converse hold, namely is it true that projective
  limits of sober spaces preserved by $\QLV$ are preserved by $\HV$?
  We know that this is true in a special case
  (Remark~\ref{rem:QL:projlim:nec}), but we conjecture that this is
  false in general.
\item Corollary~\ref{corl:consonant:lim} states that a projective
  limit of $\odot$-consonant sober spaces and proper bonding maps is
  $\odot$-consonant.  Is it necessary that the bonding maps are proper
  for this to hold?  If not, then the conclusion of
  Theorem~\ref{thm:PV:projlim} would also hold for limits of
  projective systems with arbitrary continuous bonding maps (i.e., not
  proper maps), provided that the index set has a countable cofinal
  subset, and that each $X_i$ is $\odot$-consonant and locally
  strongly sober (and compact, namely strongly sober, if $\bullet$ is
  ``$1$''); the proof would be the same, using case~2 of
  Corollary~\ref{corl:Fork:projlim} instead of case~3.
\end{enumerate}








\section*{Competing interests}

The author declares none.

\bibliographystyle{elsarticle-harv}
\ifarxiv

\else
\bibliography{projlim}
\fi







\end{document}
